\documentclass[reqno]{amsart}
\usepackage[english]{babel}
\usepackage{amssymb,amsmath,hyperref}
\usepackage{bbold}
\usepackage{amsrefs}
\usepackage{stackrel}
\usepackage[foot]{amsaddr}
\usepackage{mathrsfs}
\usepackage{cleveref, enumerate}

\DeclareMathOperator{\st}{st} 

\newcommand{\forallst}{\forall^{\st{}}}
\newcommand{\existsst}{\exists^{\st{}}}

\newcommand{\Sh}{\ensuremath{\protect{S_{\st{}}}}}

\newcommand{\intern}{\ensuremath{{\usftext{int}}}}

\newcommand{\HAC}{\ensuremath{{\usftext{HAC}}}}
\newcommand{\I}{\ensuremath{\usftext{I}}}

\newcommand{\R}{\ensuremath{\usftext{R}}}
\newcommand{\NCR}{\ensuremath{\usftext{NCR}}}

\newcommand{\T}{\ensuremath{\mathcal{T}}}

\newcommand{\usftext}[1]{\textsf{\upshape #1}}

\newcommand{\QFAC}{\ensuremath{{\usftext{QF-AC}}}} 

\newcommand{\ACA}{\ensuremath{\usftext{ACA}}} %
\newcommand{\RCA}{\ensuremath{\usftext{RCA}}} %

\newcommand{\ZFC}{\ensuremath{\usftext{ZFC}}}
\newcommand{\IST}{\ensuremath{\usftext{IST}}}

\newcommand{\WKL}{\ensuremath{\usftext{WKL}}}
\newcommand{\FAN}{\ensuremath{\usftext{FAN}}}
\newcommand{\tup}{\underline} 

\def\bdefi{\begin{defi}\rm}
\def\edefi{\end{defi}}
\def\bnota{\begin{nota}\rm}
\def\enota{\end{nota}}
\def\brem{\begin{rem}\rm}
\def\erem{\end{rem}}

\def\STP{\textup{STP}}

\def\RCA{\textup{RCA}}

\def\WKL{\textup{WKL}}

\def\HAC{\textup{HAC}}

\def\bye{\end{document}}

\def\P{{\mathcal  P}}

\def\N{{\mathbb  N}}

\def\R{{\mathbb  R}}

\def\FAN{\textup{FAN}}

\def\R{{\mathbb{R}}}
\def\({\textup{(}}
\def\){\textup{)}}

\def\st{\textup{st}}
\def\asa{\leftrightarrow}

\def\di{\rightarrow}

\def\eps{\varepsilon}

\def\ACA{\textup{ACA}}

\newbox\gnBoxA
\newdimen\gnCornerHgt
\setbox\gnBoxA=\hbox{$\ulcorner$}
\global\gnCornerHgt=\ht\gnBoxA
\newdimen\gnArgHgt

\def\bdefi{\begin{defi}\rm}
\def\edefi{\end{defi}}
\def\bnota{\begin{nota}\rm}
\def\enota{\end{nota}}
\def\brem{\begin{rem}\rm}
\def\erem{\end{rem}}

\def\FIVE{\Pi_{1}^{1}\text{-CA}_{0}}

\def\ATR{\textup{ATR}}

\def\STP{\textup{STP}}

\def\INT{\textup{INT}}

\def\RCA{\textup{RCA}}

\def\RCAo{\textup{RCA}_{0}^{\omega}}

\def\O{\mathfrak{O}}

\def\WKL{\textup{WKL}}

\def\IVT{\textup{IVT}}

\def\bye{\end{document}}

\def\P{{\mathcal  P}}

\def\N{{\mathbb  N}}

\def\R{{\mathbb  R}}

\def\FAN{\textup{FAN}}

\def\MUC{\textup{MUC}}

\def\MCT{\textup{MCT}}

\def\R{{\mathbb{R}}}
\def\({\textup{(}}
\def\){\textup{)}}

\def\asa{\leftrightarrow}

\def\di{\rightarrow}

\def\eps{\varepsilon}

\def\ACA{\textup{ACA}}
\def\paai{\Pi_{1}^{0}\textup{-}\usftext{TRANS}}
\def\Paai{\Pi_{1}^{1}\textup{-}\usftext{TRANS}}

\newtheorem{thm}{Theorem}

\newtheorem{cor}[thm]{Corollary}
\newtheorem{defi}[thm]{Definition}
\newtheorem{rem}[thm]{Remark}
\newtheorem{nota}[thm]{Notation}
\newtheorem{exa}[thm]{Example}

\newtheorem{princ}[thm]{Principle}

\newtheorem{tempie}[thm]{Template}
\newtheorem{ack}[thm]{Acknowledgement}

\newtheorem*{tempo*}{Template}

\newcommand\be{\begin{equation}}
\newcommand\ee{\end{equation}}

\newbox\gnBoxA
\newdimen\gnCornerHgt
\setbox\gnBoxA=\hbox{$\ulcorner$}
\global\gnCornerHgt=\ht\gnBoxA
\newdimen\gnArgHgt

\def\Godelnum #1{%
	\setbox\gnBoxA=\hbox{$#1$}%
	\gnArgHgt=\ht\gnBoxA%
	\ifnum \gnArgHgt<\gnCornerHgt
		\gnArgHgt=0pt%
	\else
		\advance \gnArgHgt by -\gnCornerHgt%
	\fi
	\raise\gnArgHgt\hbox{$\ulcorner$} \box\gnBoxA %
		\raise\gnArgHgt\hbox{$\urcorner$}}
\def\bdefi{\begin{defi}\rm}
\def\edefi{\end{defi}}
\def\bnota{\begin{nota}\rm}
\def\enota{\end{nota}}
\def\brem{\begin{rem}\rm}
\def\erem{\end{rem}}

\def\FIVE{\Pi_{1}^{1}\text{-\textsf{CA}}_{0}}

\def\ATR{\textup{\textsf{ATR}}}

\def\STP{\textup{\textsf{STP}}}

\def\H{\textup{\textsf{H}}}
\def\RCA{\textup{\textsf{RCA}}}

\def\RCAo{\textup{\textsf{RCA}}_{0}^{\omega}}

\def\ef{\textup{\textsf{ef}}}
\def\ns{\textup{\textsf{ns}}}
\def\O{\mathfrak{O}}

\def\WKL{\textup{\textsf{WKL}}}

\def\IVT{\textup{IVT}}
\def\IVT{\textup{\textsf{IVT}}}

\def\T{\mathcal{T}}

\def\bye{\end{document}}

\def\P{\textup{\textsf{P}}}

\def\N{{\mathbb  N}}

\def\R{{\mathbb  R}}

\def\I{{\textsf{\textup{I}}}}

\def\FAN{\textup{\textsf{FAN}}}

\def\MUC{\textup{\textsf{MUC}}}

\def\MCT{\textup{\textsf{MCT}}}
\def\ULC{\textup{\textsf{ULC}}}

\def\R{{\mathbb{R}}}
\def\({\textup{(}}
\def\){\textup{)}}

\def\st{\textup{st}}

\def\asa{\leftrightarrow}

\def\di{\rightarrow}

\def\eps{\varepsilon}

\def\ACA{\textup{\textsf{ACA}}}
\def\paai{\Pi_{1}^{0}\textup{-\textsf{TRANS}}}
\def\Paai{\Pi_{1}^{1}\textup{-\textsf{TRANS}}}

\def\QFAC{\textup{\textsf{QF-AC}}}

\def\CI{{\mathfrak{CI}}}

\def\META{\textup{\textsf{META}}}
\def\SCF{\textup{\textsf{SCF}}}
\def\HBL{\textup{\textsf{HBL}}}
\def\MTE{\textup{\textsf{MTE}}}

\def\DIV{\textup{\textsf{DIV}}}
\def\TOR{\textup{\textsf{TOR}}}
\def\PST{\textup{\textsf{PST}}}
\def\her{\textup{\textsf{her}}}

\def\CCI{\textup{\textsf{CCI}}}
\def\SWT{\textup{\textsf{SWT}}}

\def\CSU{\textup{\textsf{CSU}}}
\def\HEI{\textup{\textsf{HEI}}}

\def\DIN{\textup{\textsf{DINI}}}
\def\MU{\textup{\textsf{MU}}}
\def\MUO{\textup{\textsf{MUO}}}
\def\PICA{\textup{\textsf{PICA}}}
\def\DINI{\textup{\textsf{DINI}}}
\def\HAC{\textup{\textsf{HAC}}}
\def\FTC{\textup{\textsf{FTC}}}

\def\INT{\textup{\textsf{int}}}
\def\CRI{\textup{\textsf{CRI}}}

\numberwithin{equation}{section}
\numberwithin{thm}{section}

\usepackage{comment}

\begin{document}

\title{The unreasonable effectiveness of Nonstandard Analysis}
\author{Sam Sanders}
\address{Department of Mathematics, TU Darmstadt, Germany}
\email{sasander@me.com}
\begin{abstract}
As suggested by the title, the aim of this paper is to uncover the vast \emph{computational content} of \emph{classical} Nonstandard Analysis. 
To this end, we formulate a template $\CI$ which converts a theorem of `pure' Nonstandard Analysis, i.e.\ formulated solely with the \emph{nonstandard} definitions (of continuity, integration, differentiability, convergence, compactness, et cetera), into the associated \emph{effective} theorem.  The latter constitutes a theorem of computable mathematics \emph{no longer involving Nonstandard Analysis}.  
To establish the huge scope of $\CI$, we apply this template to representative theorems from the \emph{Big Five} categories from \emph{Reverse Mathematics}.   The latter foundational program provides a classification of the majority of theorems from `ordinary', that is non-set theoretical, mathematics into the aforementioned five categories.  The \emph{Reverse Mathematics zoo} gathers exceptions to this classification, and is studied in \cite{samzoo, samzooII} using $\CI$.  Hence, the template $\CI$ is seen to apply to essentially \emph{all of ordinary mathematics}, thanks to the Big Five classification (and associated zoo) from Reverse Mathematics.            
Finally, we establish that certain `highly constructive' theorems, called Herbrandisations, also imply the original theorem of Nonstandard Analysis from which they were obtained via $\CI$.   
\end{abstract}


\maketitle

\thispagestyle{empty}

\vspace{-0.2cm}
\section{Introduction}\label{intro}
\noindent
As suggested by the title, our aim is to uncover the \emph{vast {computational} content of {classical} Nonstandard Analysis}.
The following quotes serve as a good starting point of and motivation for our enterprise.  
\begin{quote}
\emph{It has often been held that nonstandard analysis is highly non-constructive, thus somewhat suspect, depending as it does upon the ultrapower construction to produce a model \textup{[\dots]} On the other hand, nonstandard \emph{praxis} is remarkably constructive; having the extended number set we can proceed with explicit calculations.} (Emphasis in original: \cite{NORSNSA}*{p.\ 31})
\end{quote}
\begin{quote}
\emph{Those who use nonstandard arguments often say of their proofs that they are ``constructive modulo an ultrafilter''; implicit in this statement is the suggestion that such arguments might give rise to genuine constructions}. (\cite{rossenaap}*{p.\ 494})
\end{quote}
Similar observations are made in \cites{watje, rosse,kifar, kieken, Oss3, Oss2, sc, nsawork2, venice,fath, jep}.  The reader may interpret the word \emph{constructive} as the mainstream/classical notion `effective', or as the foundational notion from Bishop's \emph{Constructive Analysis} (\cite{bish1}).  
As will become clear, both cases will be treated below (and separated carefully).      

\smallskip

To uncover the computational content of Nonstandard Analysis alluded to in the above quotes, we shall introduce a template $\CI$ 
in Section \ref{detail} which converts a theorem of \emph{pure} Nonstandard Analysis into the associated `constructive' theorem;   
here, a theorem of `pure' Nonstandard Analysis is one formulated solely with the \emph{nonstandard} definitions (of continuity, convergence, etc) rather than the usual `epsilon-delta' definitions.  

\smallskip

We shall make use of the classification provided by \emph{Reverse Mathematics} (see Section \ref{RM}) to establish that the scope of the template $\CI$ includes most of non-set theoretical mathematics.  In view of this \emph{huge} scope, the `unreasonable effectiveness of Nonstandard Analysis' is no exaggeration, and even comes as a surprise in light of the claims by Bishop and Connes regarding the constructive nature of Nonstandard Analysis, as discussed in Section \ref{foef}.     
We discuss the aim of this paper in detail in Section \ref{haim} by way of an elementary example.

\smallskip

On a historical note, the late Grigori Mints has repeatedly pushed the author to investigate the computational content of classical Nonstandard Analysis.  
In particular, Mints conjectured the existence of results analogous or similar to Kohlenbach's \emph{proof mining} program (\cite{kohlenbach3}).  
The latter program has its roots in Kreisel's pioneering work on the `unwinding' of proofs, where the latter's goal is similar to ours:
\begin{quote}
\emph{To determine the constructive \(recursive\) content or the constructive equivalent of the non-constructive concepts and theorems used in mathematics}, particularly arithmetic and analysis.  (Emphasis in original on \cite{kreimiearivier}*{p.\ 155})
\end{quote}
We discuss the connection of our results to proof mining in Section \ref{swisch} below.    

\smallskip

Finally, Horst Osswald has qualified the observation from the above quotes as \emph{Nonstandard Analysis is locally constructive}, to be understood as the fact that the mathematics performed in the nonstandard world is highly constructive while the principles needed to `jump between' the nonstandard world and usual mathematics, are highly non-constructive in general (see \cite{nsawork2}*{\S7}, \cite{Oss3}*{\S1-2}, or \cite{Oss2}*{\S17.5}).  The results in this paper shall be seen to vindicate both the Mints and Osswald view.  This paper contains numerous (rather) technical results and the uninitiated reader may refer to \cite{SB} for a very gentle introduction and more general overview.  

\subsection{Aim and results}\label{haim}
In this section, we discuss the aim of this paper in more detail.
Conceptually speaking, the aim of this paper is to formulate a template $\CI$ which takes as input the proof of a mathematical theorem from `pure' Nonstandard Analysis, i.e.\ formulated solely with nonstandard definitions (of continuity, integration, differentiability, convergence, \dots), and outputs a proof of the associated \emph{effective version} of this theorem.  The template $\CI$ works on proofs inside Nelson's syntactic approach to Nonstandard Analysis, called \emph{internal set theory}; the latter was first introduced in \cite{wownelly} and discussed in Section \ref{IIST}.  We make essential use of the results in \cite{brie}, as discussed in Section \ref{PIPI}.    

\smallskip

Intuitively speaking, the `effective version' of a mathematical theorem is obtained by replacing all its existential quantifiers by functionals providing the objects claimed to exist.     
In other words, the object claimed to exist by the theorem at hand can be \emph{computed} (in a specific technical sense) from the other data present in the theorem.   
We believe our aim to be best further illustrated by way of an example, as follows.  
\begin{exa}\label{exam}\rm
Assuming basic familiarity with Nelson's internal set theory (see Section \ref{IIST}), we illustrate our aim using the theorem:  \emph{a uniformly continuous function on the unit interval is Riemann integrable there}.  The previous theorem formulated with the \emph{nonstandard }
definitions of continuity and integration is:  
\begin{align}
(\forall f:\R\di \R)\big[(\forall x, y&\in [0,1])[x\approx y \di f(x)\approx f(y)] \label{NST}\\
&\di  (\forall \pi, \pi' \in P([0,1]))(\|\pi\|,\| \pi'\|\approx 0  \di S_{\pi}(f)\approx S_{\pi'}(f) )  \big],\notag
\end{align}
where as suggested by the notation, $\pi$ and $\pi'$ are discrete partitions of the unit interval, and $\|\pi\|$ is the mesh of a partition, i.e.\ the largest distance between two adjacent partition points;  
the number $S_{\pi}(f)$ is the \emph{Riemann sum} associated with the partition $\pi$, i.e.\ $\sum_{i=0}^{M-1}f(t_{i}) (x_{i}-x_{i+1}) $ for $\pi=(0, t_{0}, x_{1},t_{1},  \dots,x_{M-1}, t_{M-1}, 1)$.

\smallskip
\noindent
By contrast, the \emph{effective version} of the theorem is as follows:
\begin{align}\textstyle
\textstyle~(\forall f: \R\di \R, g ,n)\Big[&(\forall \textstyle x, y \in [0,1],k)(|x-y|<\frac{1}{g(k)} \di |f(x)-f(y)|\leq\frac{1}{k})\label{EST}\\
&\textstyle\di  (\forall \pi, \pi' \in P([0,1]))\big(\|\pi\|,\| \pi'\|< \frac{1}{t(g,n)}  \di |S_{\pi}(f)- S_{\pi'}(f)|\leq \frac{1}{n} \big)  \Big],\notag
\end{align}
where $t$ is primitive recursive (in the sense of G\"odel's $\textsf{T}$; see Section \ref{base}).  Note that \eqref{EST} \emph{does not} involve Nonstandard Analysis.
In Section~\ref{frakkk}, we show how a proof of the nonstandard version \eqref{NST} can be converted into a proof of the effective version \eqref{EST};  furthermore, the term $t$ in \eqref{EST} can be `read off' from the proof of \eqref{NST}. 
Essential to this enterprise is the system $\P$ from Section~\ref{P}, a
fragment of Nelson's internal set theory based on G\"odel's system $\textsf{T}$ (see \cites{avi2, godel3} for the latter).    
\end{exa}
The previous example illustrates the goal of our paper; in general, we wish to formulate a template $\CI$ to convert the proof inside $\P$ of a mathematical theorem \emph{formulated solely using nonstandard definitions} (of continuity, integration, differentiability, convergence, compactness, et cetera) as in \eqref{NST} into a proof of the associated effective version as in \eqref{EST}, in which all existential quantifiers have been removed.    
We shall formulate this template $\CI$ in Section \ref{detail}, and apply it to a number of mathematical theorems in Sections \ref{main} and \ref{RMSTUD}.  
In particular, we study representative theorems from the \emph{Big Five} categories of \emph{Reverse Mathematics};  the latter foundational program is introduced in Section \ref{RM}.    
Exceptions to the Big Five categories are gathered in the \emph{Reverse Mathematics zoo}, and are studied in \cites{samzoo, samzooII} using $\CI$.  These results suggest that the scope of our template includes essentially the whole of \emph{ordinary}, that is non-set theoretical, mathematics.  

\smallskip

Furthermore, it turns out that deriving the effective from the nonstandard version of a mathematical theorem, can often be done \emph{inside Bishop's Constructive Analysis} (\cite{bish1}), as is clear from e.g.\ Corollary~\ref{cordejedi} in Section \ref{frakkk}.  
This is surprising as such a derivation (of the effective version from the nonstandard one) involves bringing the quantifiers implicit in `$\approx$' (as in e.g.\ \eqref{NST}) to the front, which in general requires a non-constructive principle called \emph{independence of premises}.     
Essential to this constructive development is the formal system $\H$, which is the \emph{constructive version} of $\P$, as discussed in detail in Section~\ref{P}.  

\smallskip

In the case of compactness, multiple \emph{different} nonstandard formulations are possible, as discussed in Section \ref{compaq}.  When running the first formulation through the template~$\CI$, one rediscovers \emph{totally boundedness}, 
the preferred notion of compactness in constructive and computable analysis (see Section \ref{cont1} and \ref{cont2}).  
When running the second formulation (see Section \ref{dinki}) through the template $\CI$, one rediscovers certain equivalences from {Reverse Mathematics}.  
These equivalences are even \emph{explicit}\footnote{An implication $(\exists \Phi)A(\Phi)\di (\exists \Psi)B(\Psi)$ is \emph{explicit} if there is a term $t$ in the language such that additionally $(\forall \Phi)[A(\Phi)\di B(t(\Phi))]$, i.e.\ $\Psi$ can be explicitly defined in terms of $\Phi$.\label{dirkske}}.  In general, running non-explicit equivalences (usually having very simple proofs) between nonstandard theorems through $\CI$, one obtains \emph{explicit} Reverse Mathematics equivalences, 
as discussed in Sections \ref{X}, \ref{dinki}, \ref{dinkitoes},~\ref{suskeenwiske},~and~\ref{algea}.    

\smallskip

A natural question in light of the aforementioned results is whether there is an effective theorem which derives from the nonstandard theorem \emph{and also vice versa}, i.e.\ \emph{can we re-obtain the nonstandard version from a suitable effective version?}  
To this end, we introduce the notion of \emph{Herbrandisation} of a nonstandard theorem.  This `highly constructive' version is derived from the original nonstandard theorem, and in turn implies the latter.  
Thus, we establish an `algorithmic' two-way street between `soft' analysis dealing with qualitative analysis (in the guise of Nonstandard Analysis) and `hard' analysis dealing with quantitative results (embodied by the Herbrandisations).   
These results are discussed in detail in Sections \ref{frakkk} and \ref{X}.  

\smallskip

In a nutshell, we formulate a template $\CI$ in this paper which converts (a proof of) a theorem in `pure' Nonstandard Analysis, i.e.\ formulated solely with nonstandard definitions, into (a proof of) the `constructive' version of the associated theorem.  We show that this template applies to most of ordinary mathematics, as captured by the Big Five categories of Reverse Mathematics.  
Hence, we observe that Nonstandard Analysis is indeed `unreasonably effective', as suggested by the title.  
Our results are \emph{higher-order} in nature, as witnessed by the third-order term $t$ from Example~\ref{exam}.  One can also obtain second-order results, as discussed in \cite{samtamc} and Remark \ref{tamcremark}.    

\smallskip

Finally, the results in this paper should be contrasted with the current `mainstream' view of Nonstandard Analysis:  
one usually thinks of the universe of standard objects as `the usual world of mathematics', which can be studied `from the outside' using nonstandard objects such as infinitesimals.     
In this richer framework, proofs can be much shorter than those from standard (=non-Nonstandard) analysis;  furthermore, there are \emph{conservation results} guaranteeing that theorems of usual mathematics \emph{proved using Nonstandard Analysis} can also be proved \emph{without using the latter}.  Thus, the starting and end point (in the mainstream view) is always the universe of standard objects, i.e.\ usual mathematics.  By contrast, our starting point is \emph{pure} Nonstandard Analysis and our end point is constructive mathematics.              

\subsection{Methodology}\label{kintro}
In this section, we discuss the methodology used in this paper.  In particular, we try to explain \emph{why} we observe the close connection between the nonstandard and effective versions of mathematical theorems.  
The following argument is somewhat vague but can (and will) be formalised using the formal systems $\P$ or $\H$ from \cite{brie}, which in turn will be introduced in Section \ref{base}.    
\begin{enumerate}
\item The \emph{nonstandard} definitions of continuity, integrability, convergence, etc.\ in Nonstandard Analysis can be brought into the `normal form' $(\forall^{\st}x)(\exists^{\st}y)\varphi(x, y)$, where $\varphi$ does not involve `st'.    
This can be done in $\P$ and usually also in $\H$.  
\item The aforementioned `normal form' is closed under \emph{modus ponens}; indeed, it is not difficult to show that an implication of the form:
\[
(\forall^{\st}x_{0})(\exists^{\st}y_{0})\varphi_{0}(x_{0}, y_{0})\di (\forall^{\st}x_{1})(\exists^{\st}y_{1})\varphi_{1}(x_{1}, y_{1}),
\]
can \emph{also} be brought into a normal form $(\forall^{\st}x)(\exists^{\st}y)\varphi(x, y)$.  Hence, it seems theorems formulated solely with nonstandard definitions can be brought into the latter normal form. 
This can always be done in $\H$ and $\P$.     
\item The normal form $(\forall^{\st}x)(\exists^{\st}y)\varphi(x, y)$ has exactly the right structure to yield the effective version $(\forall x)\varphi(x, t(x))$.   
In particular, from the proof of the normal form $(\forall^{\st}x)(\exists^{\st}y)\varphi(x, y)$ (inside $\H$ or $\P$), a term $s$ can be `read off' such that $(\forall x)(\exists y\in s(x))\varphi(x, y)$ 
has a proof inside a system \emph{involving no Nonstandard Analysis}.  The term $t$ is then defined in terms of $s$.         
\end{enumerate}
It goes without saying that most technical details have been omitted from the above sketch, this in order to promote intuitive understanding.  
Nonetheless, the previous three steps form the skeleton of the template $\CI$ introduced in Section~\ref{detail}.  
For the remainder of this paper, the notion of `normal form' shall \emph{always} refer to a formula of the form $(\forall^{\st}x)(\exists^{\st}y)\varphi(x, y)$ with $\varphi$ internal, i.e.\ without `st'.  

\smallskip

In light of the previous observations, the class of normal forms, and hence the scope of $\CI$, seems to be \emph{very large}.

\section{Background}\label{base}
In this section, we provide some background concering Nelson's \emph{internal set theory} and Friedman's foundational program \emph{Reverse Mathematics}.  
\subsection{Internal set theory and its fragments}\label{P}
In this section, we discuss Nelson's \emph{internal set theory}, first introduced in \cite{wownelly}, and its fragments $\P$ and $\H$ from \cite{brie}.  
The latter fragments are essential to our enterprise, especially Corollary~\ref{consresultcor} below.  
\subsubsection{Internal set theory 101}\label{IIST}
In Nelson's \emph{syntactic} approach to Nonstandard Analysis (\cite{wownelly}), as opposed to Robinson's semantic one (\cite{robinson1}), a new predicate `st($x$)', read as `$x$ is standard' is added to the language of \textsf{ZFC}, the usual foundation of mathematics.  
The notations $(\forall^{\st}x)$ and $(\exists^{\st}y)$ are short for $(\forall x)(\st(x)\di \dots)$ and $(\exists y)(\st(y)\wedge \dots)$.  A formula is called \emph{internal} if it does not involve `st', and \emph{external} otherwise.   
The three external axioms \emph{Idealisation}, \emph{Standard Part}, and \emph{Transfer} govern the new predicate `st';  they are respectively defined\footnote{The superscript `fin' in \textsf{(I)} means that $x$ is finite, i.e.\ its number of elements are bounded by a natural number.} as:  
\begin{enumerate}
\item[\textsf{(I)}] $(\forall^{\st~\textup{fin}}x)(\exists y)(\forall z\in x)\varphi(z,y)\di (\exists y)(\forall^{\st}x)\varphi(x,y)$, for internal $\varphi$ with any (possibly nonstandard) parameters.  
\item[\textsf{(S)}] $(\forall^{\st} x)(\exists^{\st}y)(\forall^{\st}z)\big((z\in x\wedge \varphi(z))\asa z\in y\big)$, for any $\varphi$.
\item[\textsf{(T)}] $(\forall^{\st}t)\big[(\forall^{\st}x)\varphi(x, t)\di (\forall x)\varphi(x, t)\big]$, where $\varphi(x,t)$ is internal, and only has free variables $t, x$.  
\end{enumerate}
The system \textsf{IST} is (the internal system) \textsf{ZFC} extended with the aforementioned external axioms;  
the former is a conservative extension of \textsf{ZFC} for the internal language, as proved in \cite{wownelly}.  

\smallskip

In \cite{brie}, the authors study G\"odel's system $\textsf{T}$ extended with special cases of the external axioms of \textsf{IST}.  
In particular, they consider the systems $\H$ and $\P$, introduced in the next section, which are conservative extensions of the (internal) logical systems \textsf{E-HA}$^{\omega}$ and $\textsf{E-PA}^{\omega}$, respectively \emph{Heyting and Peano arithmetic in all finite types and the axiom of extensionality}.       
We refer to \cite{kohlenbach3}*{\S3.3} for the exact definitions of the (mainstream in mathematical logic) systems \textsf{E-HA}$^{\omega}$ and $\textsf{E-PA}^{\omega}$.  
Furthermore, \textsf{E-PA}$^{\omega*}$ and $\textsf{E-HA}^{\omega*}$ are the definitional extensions of \textsf{E-PA}$^{\omega}$ and $\textsf{E-HA}^{\omega}$ with types for finite sequences, as in \cite{brie}*{\S2}.  For the former systems, we require some notation.  
\begin{nota}[Finite sequences]\label{skim}\rm
The systems $\textsf{E-PA}^{\omega*}$ and $\textsf{E-HA}^{\omega*}$ have a dedicated type for `finite sequences of objects of type $\rho$', namely $\rho^{*}$.  Since the usual coding of pairs of numbers goes through in both, we shall not always distinguish between $0$ and $0^{*}$.  
Similarly, we do not always distinguish between `$s^{\rho}$' and `$\langle s^{\rho}\rangle$', where the former is `the object $s$ of type $\rho$', and the latter is `the sequence of type $\rho^{*}$ with only element $s^{\rho}$'.  The empty sequence for the type $\rho^{*}$ is denoted by `$\langle \rangle_{\rho}$', usually with the typing omitted.  Furthermore, we denote by `$|s|=n$' the length of the finite sequence $s^{\rho^{*}}=\langle s_{0}^{\rho},s_{1}^{\rho},\dots,s_{n-1}^{\rho}\rangle$, where $|\langle\rangle|=0$, i.e.\ the empty sequence has length zero.
For sequences $s^{\rho^{*}}, t^{\rho^{*}}$, we denote by `$s*t$' the concatenation of $s$ and $t$, i.e.\ $(s*t)(i)=s(i)$ for $i<|s|$ and $(s*t)(j)=t(|s|-j)$ for $|s|\leq j< |s|+|t|$. For a sequence $s^{\rho^{*}}$, we define $\overline{s}N:=\langle s(0), s(1), \dots,  s(N)\rangle $ for $N^{0}<|s|$.  
For a sequence $\alpha^{0\di \rho}$, we also write $\overline{\alpha}N=\langle \alpha(0), \alpha(1),\dots, \alpha(N)\rangle$ for \emph{any} $N^{0}$.  By way of shorthand, $q^{\rho}\in Q^{\rho^{*}}$ abbreviates $(\exists i<|Q|)(Q(i)=_{\rho}q)$.  Finally, we shall use $\underline{x}, \underline{y},\underline{t}, \dots$ as short for tuples $x_{0}^{\sigma_{0}}, \dots x_{k}^{\sigma_{k}}$ of possibly different type $\sigma_{i}$.          
\end{nota}    

\subsubsection{The classical system $\P$}\label{PIPI}
In this section, we introduce the system $\P$, a conservative extension of $\textsf{E-PA}^{\omega}$ with fragments of Nelson's $\IST$.  

\smallskip

To this end, we first introduce the base system $\textsf{E-PA}_{\st}^{\omega*}$.  
We use the same definition as \cite{brie}*{Def.~6.1}, where \textsf{E-PA}$^{\omega*}$ is the definitional extension of \textsf{E-PA}$^{\omega}$ with types for finite sequences as in \cite{brie}*{\S2}.  
The set $\T^{*}$ is defined as the collection of all the terms in the language of $\textsf{E-PA}^{\omega*}$.    
\bdefi\label{debs}
The system $ \textsf{E-PA}^{\omega*}_{\st} $ is defined as $ \textsf{E-PA}^{\omega{*}} + \T^{*}_{\st} + \textsf{IA}^{\st}$, where $\T^{*}_{\st}$
consists of the following axiom schemas.
\begin{enumerate}
\item The schema\footnote{The language of $\textsf{E-PA}_{\st}^{\omega*}$ contains a symbol $\st_{\sigma}$ for each finite type $\sigma$, but the subscript is essentially always omitted.  Hence $\T^{*}_{\st}$ is an \emph{axiom schema} and not an axiom.\label{omit}} $\st(x)\wedge x=y\di\st(y)$,
\item The schema providing for each closed\footnote{A term is called \emph{closed} in \cite{brie} (and in this paper) if all variables are bound via lambda abstraction.  Thus, if $\underline{x}, \underline{y}$ are the only variables occurring in the term $t$, the term $(\lambda \underline{x})(\lambda\underline{y})t(\underline{x}, \underline{y})$ is closed while $(\lambda \underline{x})t(\underline{x}, \underline{y})$ is not.  The second axiom in Definition \ref{debs} thus expresses that $\st_{\tau}\big((\lambda \underline{x})(\lambda\underline{y})t(\underline{x}, \underline{y})\big)$ if $(\lambda \underline{x})(\lambda\underline{y})t(\underline{x}, \underline{y})$ is of type $\tau$.  We usually omit lambda abstraction for brevity.\label{kootsie}} term $t\in \T^{*}$ the axiom $\st(t)$.
\item The schema $\st(f)\wedge \st(x)\di \st(f(x))$.
\end{enumerate}
The external induction axiom \textsf{IA}$^{\st}$ is as follows.  
\be\tag{\textsf{IA}$^{\st}$}
\Phi(0)\wedge(\forall^{\st}n^{0})(\Phi(n) \di\Phi(n+1))\di(\forall^{\st}n^{0})\Phi(n).     
\ee
\edefi
Secondly, we introduce some essential fragments of $\IST$ studied in \cite{brie}.  
\bdefi[External axioms of $\P$]~
\begin{enumerate}
\item$\HAC_{\INT}$: for any internal formula $\varphi$, we have
\be\label{HACINT}
(\forall^{\st}x^{\rho})(\exists^{\st}y^{\tau})\varphi(x, y)\di \big(\exists^{\st}F^{\rho\di \tau^{*}}\big)(\forall^{\st}x^{\rho})(\exists y^{\tau}\in F(x))\varphi(x,y),
\ee
\item $\textsf{I}$: for any internal formula $\varphi$, we have
\[
(\forall^{\st} x^{\sigma^{*}})(\exists y^{\tau} )(\forall z^{\sigma}\in x)\varphi(z,y)\di (\exists y^{\tau})(\forall^{\st} x^{\sigma})\varphi(x,y), 
\]
\item The system $\P$ is $\textsf{E-PA}_{\st}^{\omega*}+\textsf{I}+\HAC_{\INT}$.
\end{enumerate}
\end{defi}
Note that \textsf{I} and $\HAC_{\INT}$ are fragments of Nelson's axioms \emph{Idealisation} and \emph{Standard part}.  
By definition, $F$ in \eqref{HACINT} only provides a \emph{finite sequence} of witnesses to $(\exists^{\st}y)$, explaining its name \emph{Herbrandized Axiom of Choice}.   

\smallskip

The system $\P$ is connected to $\textsf{E-PA}^{\omega}$ by the following theorem.    
Here, the superscript `$S_{\st}$' is the syntactic translation defined in \cite{brie}*{Def.\ 7.1}, and also listed starting with \eqref{dombu} in the proof of Corollary \ref{consresultcor}.    
\begin{thm}\label{consresult}
Let $\Phi(\tup a)$ be a formula in the language of \textup{\textsf{E-PA}}$^{\omega*}_{\st}$ and suppose $\Phi(\tup a)^\Sh\equiv\forallst \tup x \, \existsst \tup y \, \varphi(\tup x, \tup y, \tup a)$. If $\Delta_{\intern}$ is a collection of internal formulas and
\be\label{antecedn}
\P + \Delta_{\intern} \vdash \Phi(\tup a), 
\ee
then one can extract from the proof a sequence of closed\footnote{Recall the definition of closed terms from \cite{brie} as sketched in Footnote \ref{kootsie}.\label{kootsie2}} terms $t$ in $\mathcal{T}^{*}$ such that
\be\label{consequalty}
\textup{\textsf{E-PA}}^{\omega*} + \Delta_{\intern} \vdash\  \forall \tup x \, \exists \tup y\in \tup t(\tup x)\ \varphi(\tup x,\tup y, \tup a).
\ee
\end{thm}
\begin{proof}
Immediate by \cite{brie}*{Theorem 7.7}.  
\end{proof}
The proofs of the soundness theorems in \cite{brie}*{\S5-7} provide an algorithm $\mathcal{A}$ to obtain the term $t$ from the theorem.  In particular, these terms 
can be `read off' from the nonstandard proofs.    

\smallskip

In light of Section \ref{kintro}, the following corollary (which is not present in \cite{brie}) is essential to our results.  Indeed, the following corollary expresses that we may obtain effective results as in \eqref{effewachten} from any theorem of Nonstandard Analysis which has the same form as in \eqref{bog}.  In Sections \ref{main} and \ref{RMSTUD}, we show that the scope of this corollary includes the Big Five systems of Reverse Mathematics (see Section \ref{RM}).  
\begin{cor}\label{consresultcor}
If $\Delta_{\intern}$ is a collection of internal formulas and $\psi$ is internal, and
\be\label{bog}
\P + \Delta_{\intern} \vdash (\forall^{\st}\underline{x})(\exists^{\st}\underline{y})\psi(\underline{x},\underline{y}, \underline{a}), 
\ee
then one can extract from the proof a sequence of closed$^{\ref{kootsie2}}$ terms $t$ in $\mathcal{T}^{*}$ such that
\be\label{effewachten}
\textup{\textsf{E-PA}}^{\omega*} +\QFAC^{1,0}+ \Delta_{\intern} \vdash (\forall \underline{x})(\exists \underline{y}\in t(\underline{x}))\psi(\underline{x},\underline{y},\underline{a}).
\ee
\end{cor}
\begin{proof}
Clearly, if for internal $\psi$ and $\Phi(\underline{a})\equiv (\forall^{\st}\underline{x})(\exists^{\st}\underline{y})\psi(x, y, a)$, we have $[\Phi(\underline{a})]^{S_{\st}}\equiv \Phi(\underline{a})$, then the corollary follows immediately from the theorem.  
A tedious but straightforward verification using the clauses (i)-(v) in \cite{brie}*{Def.\ 7.1} establishes that indeed $\Phi(\underline{a})^{S_{\st}}\equiv \Phi(\underline{a})$.  
For completeness, we now list these five inductive clauses and perform this verification.  

\smallskip

Hence, suppose $\Phi(\underline{a})$ and $\Psi(\underline{b})$  in the language of $\P$ have the interpretations
\be\label{dombu}
\Phi(\underline{a})^{S_{\st}}\equiv (\forall^{\st}\underline{x})(\exists^{\st}\underline{y})\varphi(\underline{x},\underline{y},\underline{a}) \textup{ and } \Psi(\underline{b})^{S_{\st}}\equiv (\forall^{\st}\underline{u})(\exists^{\st}\underline{v})\psi(\underline{u},\underline{v},\underline{b}),
\ee
 for internal $\psi, \varphi$.  These formulas then behave as follows by \cite{brie}*{Def.\ 7.1}:
\begin{enumerate}[(i)]
\item $\psi^{S_{\st}}:=\psi$ for atomic internal $\psi$.  
\item$ \big(\st(z)\big)^{S_{\st}}:=(\exists^{\st}x)(z=x)$.
\item $(\neg \Phi)^{S_{\st}}:=(\forall^{\st} \underline{Y})(\exists^{\st}\underline{x})(\forall \underline{y}\in \underline{Y}[\underline{x}])\neg\varphi(\underline{x},\underline{y},\underline{a})$.  
\item$(\Phi\vee \Psi)^{S_{\st}}:=(\forall^{\st}\underline{x},\underline{u})(\exists^{\st}\underline{y}, \underline{v})[\varphi(\underline{x},\underline{y},\underline{a})\vee \psi(\underline{u},\underline{v},\underline{b})]$
\item $\big( (\forall z)\Phi \big)^{S_{\st}}:=(\forall^{\st}\underline{x})(\exists^{\st}\underline{y})(\forall z)(\exists \underline{y}'\in \underline{y})\varphi(\underline{x},\underline{y}',z)$
\end{enumerate}
Hence, fix $\Phi_{0}(\underline{a})\equiv(\forall^{\st}\underline{x})(\exists^{\st}\underline{y})\psi_{0}(\underline{x},\underline{y}, \underline{a})$ with internal $\psi_{0}$, and note that $\phi^{S_{\st}}\equiv\phi$ for any internal formula.  
We have $[\st(\underline{y})]^{S_{\st}}\equiv (\exists^{\st} \underline{w})(\underline{w}=\underline{y})$ and also 
\[
[\neg\st(\underline{y})]^{S_{\st}}\equiv (\forall^{\st} \underline{W} ) (\exists^{\st}\underline{x})(\forall \underline{w}\in \underline{W}[\underline{x}])\neg(\underline{w}=\underline{y})\equiv (\forall^{\st}\underline{w})(\underline{w}\ne \underline{y}).  
\]    
Hence, $[\neg\st(\underline{y})\vee\neg \psi_{0}(\underline{x}, \underline{y}, \underline{a})]^{S_{\st}}$ is just $(\forall^{\st}\underline{w})[(\underline{w}\ne \underline{y}) \vee \neg \psi_{0}(\underline{x}, \underline{y}, \underline{a})]$, and 
\[
\big[(\forall \underline{y})[\neg\st(\underline{y})\vee \neg\psi_{0}(\underline{x}, \underline{y}, \underline{a})]\big]^{S_{\st}}\equiv
 (\forall^{\st}\underline{w})(\exists^{\st}\underline{v})(\forall \underline{y})(\exists \underline{v}'\in \underline{v})[\underline{w}\ne\underline{y}\vee \neg\psi_{0}(\underline{x}, \underline{y}, \underline{a})].
\]
which is just $(\forall^{\st}\underline{w})(\forall \underline{y})[(\underline{w}\ne \underline{y}) \vee \neg\psi_{0}(\underline{x}, \underline{y}, \underline{a})]$.  Furthermore, we have
\begin{align*}
\big[(\exists^{\st}y)\psi_{0}(\underline{x}, \tup y, \tup a)\big]^{S_{\st}}&\equiv\big[\neg(\forall \underline{y})[\neg\st(\underline{y})\vee\neg \psi_{0}(\underline{x}, \underline{y}, \underline{a})]\big]^{S_{\st}}\\
&\equiv(\forall^{\st} \underline{V})(\exists^{\st}\underline{w})(\forall \underline{v}\in \underline{V}[\underline{w}])\neg[(\forall \underline{y})[(\underline{w}\ne \underline{y}) \vee \neg\psi_{0}(\underline{x}, \underline{y}, \underline{a})]].\\
&\equiv (\exists^{\st}\underline{w})(\exists \underline{y})[(\underline{w}= \underline{y}) \wedge \psi_{0}(\underline{x}, \underline{y}, \underline{a})]]\equiv (\exists^{\st}\underline{w})\psi_{0}(\underline{x}, \underline{w}, \underline{a}).
\end{align*}
Hence, we have proved so far that $(\exists^{\st}\underline{y})\psi_{0}(\underline{x}, \underline{y}, \underline{a})$ is invariant under $S_{\st}$.  By the previous, we also obtain:  
\[
\big[\neg \st(\underline{x})\vee (\exists^{\st}y)\psi_{0}(\underline{x}, \tup y, \tup a)\big]^{S_{\st}}\equiv  (\forall^{\st}\underline{w}')(\exists^{\st} \underline{w})[(\underline{w}'\ne \underline{x}) \vee \psi_{0}(\tup x, \tup w, \tup a)].
\]
Our final computation now yields the desired result: 
\begin{align*}
\big[(\forall^{\st} \underline{x})(\exists^{\st}y)\psi_{0}(\underline{x}, \tup y, \tup a)\big]^{S_{\st}}
&\equiv\big[(\forall \underline{x})(\neg \st(\underline{x})\vee (\exists^{\st}y)\psi_{0}(\underline{x}, \tup y, \tup a))\big]^{S_{\st}}\\
&\equiv(\forall^{\st}\underline{w}')(\exists^{\st} \underline{w})(\forall \underline{x})(\exists \underline{w}''\in \underline{w})[(\underline{w}'\ne \underline{x}) \vee \psi_{0}(\tup x, \tup w'', \tup a)].\\
&\equiv(\forall^{\st}\underline{w}')(\exists^{\st} \underline{w})(\exists \underline{w}''\in \underline{w}) \psi_{0}(\tup w', \tup w'', \tup a).
\end{align*}
The last step is obtained by taking $\underline{x}=\underline{w}'$.  Hence, we may conclude that the normal form $(\forall^{\st} \underline{x})(\exists^{\st}y)\psi_{0}(\underline{x}, \tup y, \tup a)$ is invariant under $S_{\st}$, and we are done.    
\end{proof}
For the rest of this paper, the notion `normal form' shall refer to a formula as in \eqref{bog}, i.e.\ of the form $(\forall^{\st}x)(\exists^{\st}y)\varphi(x,y)$ for $\varphi$ internal.  

\smallskip

Finally, the previous theorems do not really depend on the presence of full Peano arithmetic.  
We shall study the following subsystems.   
\bdefi[Weak systems]~
\begin{enumerate}
\item Let \textsf{E-PRA}$^{\omega}$ be the system defined in \cite{kohlenbach2}*{\S2} and let \textsf{E-PRA}$^{\omega*}$ 
be its definitional extension with types for finite sequences as in \cite{brie}*{\S2}. 
\item $(\QFAC^{\rho, \tau})$ For every quantifier-free internal formula $\varphi(x,y)$, we have
\be\label{keuze}
(\forall x^{\rho})(\exists y^{\tau})\varphi(x,y) \di (\exists F^{\rho\di \tau})(\forall x^{\rho})\varphi(x,F(x))
\ee
\item The system $\RCAo$ is $\textsf{E-PRA}^{\omega}+\QFAC^{1,0}$.  
\end{enumerate}
\edefi
The system $\RCAo$ is the `base theory of higher-order Reverse Mathematics' as introduced in \cite{kohlenbach2}*{\S2}.  
We permit ourselves a slight abuse of notation by also referring to the system $\textsf{E-PRA}^{\omega*}+\QFAC^{1,0}$ as $\RCAo$.
\begin{cor}\label{consresultcor2}
The previous theorem and corollary go through for $\P$ and $\textsf{\textup{E-PA}}^{\omega*}$ replaced by $\P_{0}\equiv \textsf{\textup{E-PRA}}^{\omega*}+\T_{\st}^{*} +\HAC_{\INT} +\textsf{\textup{I}}+\QFAC^{1,0}$ and $\RCAo$.  
\end{cor}
\begin{proof}
The proof of \cite{brie}*{Theorem 7.7} goes through for any fragment of \textsf{E-PA}$^{\omega{*}}$ which includes \textsf{EFA}, sometimes also called $\textsf{I}\Delta_{0}+\textsf{EXP}$.  
In particular, the exponential function is (all what is) required to `easily' manipulate finite sequences.    
\end{proof}
We now discuss the \emph{Standard Part} principle $\Omega$\textsf{-CA}, a very practical consequence of the axiom $\HAC_{\INT}$.  
Intuitively speaking, $\Omega$\textsf{-CA} expresses that we can obtain 
the standard part (in casu $G$) of \emph{$\Omega$-invariant} nonstandard objects (in casu $F(\cdot,M)$).   
Note that we write `$N\in \Omega$' as short for $\neg\st(N^{0})$.
\bdefi[$\Omega$-invariance]\label{homega} Let $F^{(\sigma\times  0)\di 0}$ be standard and fix $M^{0}\in \Omega$.  
Then $F(\cdot,M)$ is {\bf $\Omega$-invariant} if   
\be\label{homegainv}
(\forall^{\st} x^{\sigma})(\forall N^{0}\in \Omega)\big[F(x ,M)=_{0}F(x,N) \big].  
\ee
\edefi
\begin{princ}[$\Omega$\textsf{-CA}]\rm Let $F^{(\sigma\times 0)\di 0}$ be standard and fix $M^{0}\in \Omega$.
For every $\Omega$-invariant $F(\cdot,M)$, there is a standard $G^{\sigma\di 0}$ such that
\be\label{homegaca}
(\forall^{\st} x^{\sigma})(\forall N^{0}\in \Omega)\big[G(x)=_{0}F(x,N) \big].  
\ee
\end{princ}
The axiom $\Omega$\textsf{-CA} provides the standard part of a nonstandard object, if the latter is \emph{independent of the choice of nonstandard number} used in its definition.  
\begin{thm}\label{drifh}
The system $\P_{0}$ proves $\Omega\textup{\textsf{-CA}}$.  
\end{thm}
\begin{proof}
We sketch the derivation of $\Omega$\textsf{-CA} from $\HAC_{\INT}$.  
Let $F(\cdot,M^{0})$ be $\Omega$-invariant, i.e.\ we have 
\be\label{dorkillllll}
(\forall^{\st} x^{\sigma})(\forall N^{0},M^{0}\in \Omega)\big[F(x ,M)=_{0}F(x,N) \big].  
\ee
We immediately obtain (any nonstandard $k^{0}$ will do) that 
\[
(\forall^{\st} x^{\sigma})(\exists k^{0})(\forall N^{0},M^{0}\geq k)\big[F(x ,M)=_{0}F(x,N) \big].  
\]
By the induction axioms present in $\P_{0}$, there is a least such $k$ for every standard $x^{\sigma}$.  
By our assumption \eqref{dorkillllll}, such \emph{least} number $k^{0}$ must be standard, yielding:
\[
(\forall^{\st} x^{\sigma})(\exists^{\st}k^{0})(\forall N^{0},M^{0}\geq k)\big[F(x ,M)=_{0}F(x,N) \big], 
\]
which we could also have obtained via underspill (see \cite{brie}*{Prop.\ 5.11}).
Now apply $\HAC_{\INT}$ to obtain standard $\Phi^{\sigma\di 0}$ such that
\[
(\forall^{\st} x^{\sigma})(\exists k^{0}\in \Phi(x))(\forall N^{0},M^{0}\geq k)\big[F(x ,M)=_{0}F(x,N) \big].
\]
Next, define $\Psi(x):= \max_{i<|\Phi(x)|}\Phi(x)(i)$ and note that 
\[
(\forall^{\st} x^{\sigma})(\forall N^{0},M^{0}\geq \Psi(x))\big[F(x ,M)=_{0}F(x,N) \big].
\]
Finally, put $G(x):=F(x,\Psi(x))$ and note that $\Omega$\textsf{-CA} follows.
\end{proof}
Finally, Dinis, Ferreira, and Gaspar present systems similar to $\P$ and $\H$ in \cites{fega,dinispinis}.
As discussed in \cite{samdinis}, these systems are less suitable for our purposes.  

\subsubsection{The constructive system $\H$}
In this section, we define the system $\H$, the constructive counterpart of $\P$. 
The system $\textsf{H}$ was first introduced in \cite{brie}*{\S5.2}, and constitutes a conservative extension of Heyting arithmetic $\textup{\textsf{E-HA}}^{\omega} $ by \cite{brie}*{Cor.\ 5.6}.
We now study the system $\H$ in more detail.  

\smallskip

Similar to Definition \ref{debs}, we define $ \textsf{E-HA}^{\omega*}_{\st} $ as $ \textsf{E-HA}^{\omega{*}} + \T^{*}_{\st} + \textsf{IA}^{\st}$, where $\textsf{E-HA}^{\omega*}$ is just $\textsf{E-PA}^{\omega*}$ without the law of excluded middle.  
Furthermore, we define
\[
\H\equiv \textup{\textsf{E-HA}}^{\omega*}_{\st}+\HAC + {\I}+\NCR+\textsf{HIP}_{\forall^{\st}}+\textsf{HGMP}^{\st},
\]
where $\HAC$ is $\HAC_{\INT}$ without any restriction on the formula, and where the remaining axioms are defined in the following definition.
\bdefi[Three axioms of $\H$]\label{flah}~
\begin{enumerate}\rm
\item $\textsf{HIP}_{\forall^{\st}}$
\[
[(\forall^{\st}x)\phi(x)\di (\exists^{\st}y)\Psi(y)]\di (\exists^{\st}y')[(\forall^{\st}x)\phi(x)\di (\exists y\in y')\Psi(y)],
\]
where $\Psi(y)$ is any formula and $\phi(x)$ is an internal formula of \textsf{E-HA}$^{\omega*}$. 
\item $\textsf{HGMP}^{\st}$
\[
[(\forall^{\st}x)\phi(x)\di \psi] \di (\exists^{\st}x')[(\forall x\in x')\phi(x)\di \psi] 
\]
where $\phi(x)$ and $\psi$ are internal formulas in the language of \textsf{E-HA}$^{\omega*}$.
\item \textsf{NCR}
\[
(\forall y^{\tau})(\exists^{\st} x^{\rho} )\Phi(x, y) \di (\exists^{\st} x^{\rho^{*}})(\forall y^{\tau})(\exists x'\in x )\Phi(x', y),
\]
where $\Phi$ is any formula of \textsf{E-HA}$^{\omega*}$
\end{enumerate}
\edefi
Intuitively speaking, the first two axioms of Definition \ref{flah} allow us to perform a number of \emph{non-constructive operations} (namely \emph{Markov's principle} and \emph{independence of premises}) 
on the standard objects of the system $\H$, provided we introduce a `Herbrandisation' as in the consequent of $\HAC$, i.e.\ a finite list of possible witnesses rather than one single witness. 
Furthermore, while $\H$ includes idealisation \textsf{I}, one often uses the latter's \emph{classical contraposition}, explaining why \textsf{NCR} is useful (and even essential) in the context of intuitionistic logic.  We discuss the constructive nature of the axioms from Definition \ref{flah} in more detail in Remark \ref{dingdong}.       

\smallskip

Surprisingly, the axioms from Definition \ref{flah} are exactly what is needed to convert nonstandard definitions (of continuity, integrability, convergence, et cetera) into the normal form $(\forall^{\st}x)(\exists^{\st}y)\varphi(x, y)$ for internal $\varphi$, as is clear from e.g.\ Corollary \ref{cordejedi}.  
The latter normal form plays an equally important role in the constructive case as in the classical case by the following theorem.  
\begin{thm}\label{consresult2}
If $\Delta_{\intern}$ is a collection of internal formulas, $\varphi$ is internal, and
\be\label{antecedn3}
\textup{\textsf{H}} + \Delta_{\intern} \vdash \forallst \tup x \, \existsst \tup y \, \varphi(\tup x, \tup y, \tup a), 
\ee
then one can extract from the proof a sequence of closed terms $t$ in $\mathcal{T}^{*}$ such that
\be\label{consequalty3}
\textup{\textsf{E-HA}}^{\omega*} + \Delta_{\intern} \vdash\  \forall \tup x \, \exists \tup y\in \tup t(\tup x)\ \varphi(\tup x,\tup y, \tup a).
\ee
\end{thm}
\begin{proof}
Immediate by \cite{brie}*{Theorem 5.9}.  Note that in the latter, just like in the proof of Corollary \ref{consresultcor}, $\forallst \tup x \, \existsst \tup y \, \varphi(\tup x, \tup y, \tup a) $ is proved to be `invariant' under a suitable syntactic translation.  
\end{proof}
The proofs of the soundness theorems in \cite{brie}*{\S5-7} provide an algorithm $\mathcal{B}$ to obtain the term $t$ from the theorem.  
We have formalised $\mathcal{B}$ in \cite{EXCESS} in the proof assistant Agda.
Finally, we point out one very useful principle.  
\begin{thm}\label{doppi}
The systems $\P, \H$, and $\P_{0}$ prove \emph{overspill}, i.e.\
\be\tag{\textsf{OS}}
(\forall^{\st}x^{\rho})\varphi(x)\di (\exists y^{\rho})\big[\neg\st(y)\wedge \varphi(y)  \big],
\ee
for any internal formula $\varphi$.
\end{thm}
\begin{proof}
See \cite{brie}*{Prop.\ 3.3}.  
\end{proof}
In conclusion, we have introduced the systems $\H$, $\P$, which are conservative extensions of Peano and Heyting arithmetic with fragments of Nelson's internal set theory.  
We have observed that central to the conservation results (Corollary~\ref{consresultcor} and Theorem~\ref{consresult}) is the normal form $(\forall^{\st}x)(\exists^{\st}y)\varphi(x, y)$ for internal $\varphi$.  
We finish this section with a conceptual remark.  
\begin{rem}[Higher-order and second-order results]\label{tamcremark}\rm
The results in this paper are clearly \emph{higher-order} in nature, as follows: 
the term $t$ from Example \ref{exam} is third-order and $\underline{x}, \underline{y}$ in \eqref{bog} and \eqref{antecedn3} can involve variables of any finite type.  
As discussed in Section \ref{RM}, Reverse Mathematics (and also Turing's `machine' notion of computability) is (generally) restricted to the language of second-order arithmetic.  
With a clever tweak, one can however obtain \emph{second-order} results from Nonstandard Analysis, as explored in \cite{samtamc}.  In particular, rather than a term $t$ from G\"odel's $\textsf{T}$ as in e.g.\ Corollary \ref{consresultcor}, one obtains a Turing machine index for computing a function that `emulates' this term $t$.  On a technical note, applying the `$\textsf{ECF}$' translation (see \cite{kohlenbach2}*{\S2}) does not work,
 as \emph{Transfer} is converted to the existence of a discontinuous function on $\N^{\N}$ by Corollary \ref{consresultcor} (see Theorem \ref{frood} for this result), and the existence of such functions is converted to `$0=1$' by $\textsf{ECF}$. 
\end{rem}
\subsubsection{Notations}
In this section, we introduce notations relating to $\H$ and $\P$.  

\smallskip

First of all, we mostly use the same notations as in \cite{brie}.  
\begin{rem}[Notations]\label{notawin}\rm
We write $(\forall^{\st}x^{\tau})\Phi(x^{\tau})$ and $(\exists^{\st}x^{\sigma})\Psi(x^{\sigma})$ as short for 
$(\forall x^{\tau})\big[\st(x^{\tau})\di \Phi(x^{\tau})\big]$ and $(\exists^{\st}x^{\sigma})\big[\st(x^{\sigma})\wedge \Psi(x^{\sigma})\big]$.     
We also write $(\forall x^{0}\in \Omega)\Phi(x^{0})$ and $(\exists x^{0}\in \Omega)\Psi(x^{0})$ as short for 
$(\forall x^{0})\big[\neg\st(x^{0})\di \Phi(x^{0})\big]$ and $(\exists x^{0})\big[\neg\st(x^{0})\wedge \Psi(x^{0})\big]$.  Furthermore, $\neg\st(x^{0})$, is abbreviated by `$x^{0}\in \Omega$'.  
Finally, a formula $A$ is `internal' if it does not involve $\st$, and $A^{\st}$ is defined from $A$ by appending `st' to all quantifiers (except bounded number quantifiers).    
\end{rem}
Secondly, we will use the usual notations for natural, rational and real numbers and functions as introduced in \cite{kohlenbach2}*{p.\ 288-289}. (and \cite{simpson2}*{I.8.1} for the former).  
We only list the definition of real number and related notions in $\P$ and related systems.
\begin{defi}[Real numbers and related notions in $\P$]\label{keepintireal}\rm~
\begin{enumerate}
\item A (standard) real number $x$ is a (standard) fast-converging Cauchy sequence $q_{(\cdot)}^{1}$, i.e.\ $(\forall n^{0}, i^{0})(|q_{n}-q_{n+i})|<_{0} \frac{1}{2^{n}})$.  
We use Kohlenbach's `hat function' from \cite{kohlenbach2}*{p.\ 289} to guarantee that every sequence $f^{1}$ is a real.  
\item We write $[x](k):=q_{k}$ for the $k$-th approximation of a real $x^{1}=(q^{1}_{(\cdot)})$.    
\item Two reals $x, y$ represented by $q_{(\cdot)}$ and $r_{(\cdot)}$ are \emph{equal}, denoted $x=_{\R}y$, if $(\forall n)(|q_{n}-r_{n}|\leq \frac{1}{2^{n}})$. Inequality $<_{\R}$ is defined similarly.         
\item We  write $x\approx y$ if $(\forall^{\st} n)(|q_{n}-r_{n}|\leq \frac{1}{2^{n}})$ and $x\gg y$ if $x>y\wedge x\not\approx y$.  
\item Functions $F:\R\di \R$ mapping reals to reals are represented by functionals $\Phi^{1\di 1}$ mapping equal reals to equal reals, i.e. 
\be\tag{\textsf{RE}}\label{furg}
(\forall x, y)(x=_{\R}y\di \Phi(x)=_{\R}\Phi(y)).
\ee
 \item For a space $X$ with metric $|\cdot|_{X}:X\di \R$, we write `$x\approx y$' for `$|x-y|_{X}\approx 0$'.    
\item Sets of objects of type $\rho$ are denoted $X^{\rho\di 0}, Y^{\rho\di }, Z^{\rho\di 0}, \dots$ and are given by their characteristic functions $f^{\rho\di 0}_{X}$, i.e.\ $(\forall x^{\rho})[x\in X\asa f_{X}(x)=_{0}1]$, where $f_{X}^{\rho\di 0}$ is assumed to output zero or one.  
\end{enumerate}
\end{defi}
Note that `$x\approx y$' for points $x, y\in X$ can also be defined if a collection $\O$ of `basic open sets' of $X$ is given (without reference to a metric).  
Indeed, given such a collection $\O$, `$x\approx y$' is just $(\forall^{\st} O\in \O)(x\in O \asa y\in O)$.  

\smallskip

Thirdly, we use the usual extensional notion of equality. 
\begin{rem}[Equality]\label{equ}\rm
All the above systems include equality between natural numbers `$=_{0}$' as a primitive.  Equality `$=_{\tau}$' for type $\tau$-objects $x,y$ is defined as:
\be\label{aparth}
[x=_{\tau}y] \equiv (\forall z_{1}^{\tau_{1}}\dots z_{k}^{\tau_{k}})[xz_{1}\dots z_{k}=_{0}yz_{1}\dots z_{k}]
\ee
if the type $\tau$ is composed as $\tau\equiv(\tau_{1}\di \dots\di \tau_{k}\di 0)$.
In the spirit of Nonstandard Analysis, we define `approximate equality $\approx_{\tau}$' as follows:
\be\label{aparth2}
[x\approx_{\tau}y] \equiv (\forall^{\st} z_{1}^{\tau_{1}}\dots z_{k}^{\tau_{k}})[xz_{1}\dots z_{k}=_{0}yz_{1}\dots z_{k}]
\ee
with the type $\tau$ as above.  All the above systems include the \emph{axiom of extensionality} for all $\varphi^{\rho\di \tau}$ as follows:
\be\label{EXT}\tag{\textsf{E}}  
(\forall  x^{\rho},y^{\rho}) \big[x=_{\rho} y \di \varphi(x)=_{\tau}\varphi(y)   \big].
\ee
However, as noted in \cite{brie}*{p.\ 1973}, the so-called axiom of \emph{standard} extensionality \eqref{EXT}$^{\st}$ is problematic and cannot be included in $\P$ or $\H$.  
Finally, a functional $\Xi^{ 1\di 0}$ is called an \emph{extensionality functional} for $\varphi^{1\di 1}$ if 
\be\label{turki}
(\forall k, f^{1}, g^{1})\big[ \overline{f}\Xi(f,g, k)=_{0}\overline{g}\Xi(f,g,k) \di \overline{\varphi(f)}k=_{0}\overline{\varphi(g)}k \big],  
\ee
i.e.\ $\Xi$ witnesses \eqref{EXT} for $\varphi$.  
As will become clear in Section \ref{X}, standard extensionality is translated by our template $\CI$ into the existence of an extensionality functional, and the latter amounts to no more than an unbounded search.   
\end{rem}

\subsection{Reverse Mathematics 101}\label{RM}
Reverse Mathematics (RM hereafter) is a program in the foundations of mathematics initiated around 1975 by Friedman (\cites{fried,fried2}) and developed extensively by Simpson (\cite{simpson2, simpson1}) and others.  
We refer to \cite{stillebron} for an introduction to RM.  
The aim of RM is to find the axioms necessary to prove a statement of \emph{ordinary} mathematics, i.e.\ dealing with countable or separable objects.   
The classical\footnote{In \emph{Constructive Reverse Mathematics} (\cite{ishi1}), the base theory is based on intuitionistic logic.} base theory $\RCA_{0}$ of `computable\footnote{The system $\RCA_{0}$ consists of induction $I\Sigma_{1}$, and the {\bf r}ecursive {\bf c}omprehension {\bf a}xiom $\Delta_{1}^{0}$-CA.} mathematics' is always assumed.  
Thus, the aim of RM is as follows:  
\begin{quote}
\emph{to find the minimal axioms $A$ such that $\RCA_{0}$ proves $ [A\di T]$ for statements $T$ of ordinary mathematics.}
\end{quote}
Surprisingly, once the minimal axioms $A$ have been found, we almost always also have $\RCA_{0}\vdash [A\asa T]$, i.e.\ not only can we derive the theorem $T$ from the axioms $A$ (the `usual' way of doing mathematics), we can also derive the axiom $A$ from the theorem $T$ (the `reverse' way of doing mathematics).  In light of the latter, the field was baptised `Reverse Mathematics'.    

\smallskip

Perhaps even more surprisingly, in the majority\footnote{Exceptions are classified in the so-called Reverse Mathematics Zoo (\cite{damirzoo}).  These are often based on combinatorics or logic.  
} 
of cases for a statement $T$ of ordinary mathematics, either $T$ is provable in $\RCA_{0}$, or the latter proves $T\asa A_{i}$, where $A_{i}$ is one of the logical systems $\WKL_{0}, \ACA_{0},$ $ \ATR_{0}$ or $\FIVE$.  The latter together with $\RCA_{0}$ form the `Big Five' and the aforementioned observation that most mathematical theorems fall into one of the Big Five categories, is called the \emph{Big Five phenomenon} (\cite{montahue}*{p.~432}).  
Furthermore, each of the Big Five has a natural formulation in terms of (Turing) computability (see e.g.\ \cite{simpson2}*{I.3.4, I.5.4, I.7.5}).
As noted by Simpson in \cite{simpson2}*{I.12}, each of the Big Five also corresponds (sometimes loosely) to a foundational program in mathematics.  

\smallskip

The logical framework for Reverse Mathematics is \emph{second-order arithmetic}, in which only natural numbers and sets thereof are available.  As a result, functions from reals to reals are not available, and have to be represented by so-called \emph{codes} (see \cite{simpson2}*{II.6.1}).  In the latter case, the coding of continuous functions amounts to introducing a modulus of (pointwise) continuity (see \cite{kohlenbach4}*{\S4}).  The nonstandard theorems proved in the system $\P$ do not involve coding for continuous functions (reals are coded as in Definition \ref{keepintireal}); 
however, as will become clear below, a modulus of continuity naturally `falls out of' the nonstandard definition of continuity as in \eqref{NST}.  
Thus, the nonstandard framework seems to `do the coding for us'.

\smallskip

In light of the previous, one of the main results of RM is that \emph{mathematical theorems} fall into \emph{only five} logical categories.  
By contrast, there are lots and lots of (purely logical or non-mathematical) statements which fall outside of these five categories.  Similarly, most \emph{mathematical theorems} from Nonstandard Analysis 
have the normal from required for applying term extraction via Corollary~\ref{consresultcor}, while there are plenty of non-mathematical or purely logical statements which do not.  
In conclusion, the results in this paper are inspired by the \emph{Reverse Mathematics way of thinking} that mathematical theorems (known in the literature) will behave `much nicer' than arbitrary formulas (even of restricted complexity).  In particular, since there is no meta-theorem for the (Big Five and its zoo) classification of RM, one cannot hope to obtain a meta-theorem for the template $\CI$ from Section~\ref{detail}.     

\smallskip

Finally, there is a tradition of Nonstandard Analysis in Reverse Mathematics and related topics (see e.g.\ \cites{tahaar,pimpson, tanaka1, tanaka2, horihata1, yo1, yokoyama2, yokoyama3}), which provides a source of proofs in (pure) Nonstandard Analysis for $\CI$ as defined in Section \ref{detail}.  
\section{Main results I: Continuity, integration, convergence, and differentiability}\label{main}
In this section, we prove our first batch of results, namely we show how to convert nonstandard theorems dealing with the notions mentioned in the section title, 
into effective theorems \emph{no longer involving Nonstandard Analysis} (and vice versa). 

\smallskip

We provide full details in the sections dealing with continuity, Riemann integration, and limits, and then switch to less detailed sketches for the fundamental theorem of calculus and Picard's theorem.  
These case studies allow us to formulate our template $\CI$ in Section~\ref{detail}.  
The results in this section show that the template $\CI$ applies to representative theorems of the base theory of Reverse Mathematics.      
  
\subsection{Riemann integration}\label{frakkk}
In this section, we study the statement $\CRI$: \emph{a uniformly continuous function on the unit interval is Riemann integrable}.  
We first obtain the effective version of $\CRI$ from the nonstandard version inside $\P_{0}$. 
We then obtain the same result in the \emph{constructive} system $\H$.  Finally, we re-obtain the nonstandard version from a special effective version, called the \emph{Hebrandisation}.
\subsubsection{Riemann integration in $\P_{0}$}
First of all, the `usual' nonstandard definitions of continuity and integration are as follows.
\bdefi[Continuity]\label{Kont}
A function $f$ is \emph{nonstandard continuous} on $[0,1]$ if
\be\label{soareyou3}
(\forall^{\st}x\in [0,1])(\forall y\in [0,1])[x\approx y \di f(x)\approx f(y)].
\ee
A function $f$ is \emph{nonstandard uniformly continuous} on $[0,1]$ if
\be\label{soareyou4}
(\forall x, y\in [0,1])[x\approx y \di f(x)\approx f(y)].
\ee
\edefi
\bdefi[Integration]\label{kunko}~
\begin{enumerate}
\item A \emph{partition} of $[0,1]$ is any sequence $\pi=(0, t_{0}, x_{1},t_{1},  \dots,x_{M-1}, t_{M-1}, 1)$.  We write `$\pi \in P([0,1]) $' to denote that $\pi$ is such a partition.
\item For $\pi\in P([0,1])$, $\|\pi\|$ is the \emph{mesh}, i.e.\ the largest distance between two adjacent partition points $x_{i}$ and $x_{i+1}$. 
\item For $\pi\in P([0,1])$ and $f:\R\di \R$, the real $S_{\pi}(f):=\sum_{i=0}^{M-1}f(t_{i}) (x_{i}-x_{i+1}) $ is the \emph{Riemann sum} of $f$ and $\pi$.  
\item A function $f$ is \emph{nonstandard integrable} on $[0,1]$ if
\be\label{soareyou5}
(\forall \pi, \pi' \in P([0,1]))\big[\|\pi\|,\| \pi'\|\approx 0  \di S_{\pi}(f)\approx S_{\pi'}(f)  \big].
\ee
\end{enumerate}
\edefi
Let $\textsf{CRI}_{\ns}$ be the statement $(\forall f:\R\di \R)[\eqref{soareyou4}\di \eqref{soareyou5}]$, and let $\textsf{CRI}_{\ef}(t)$ be \eqref{EST}.  
\begin{thm}\label{varou}
From the proof of $\CRI_{\ns}$ in $\P_{0} $, a term $t$ can be extracted such that $\textup{\textsf{E-PRA}}^{\omega*} $ proves $\CRI_{{\ef}}(t)$.  
\end{thm}
\begin{proof}
We first show that $\CRI_{\ns}$ can be proved in $\P_{0} $.  Given two partitions $\pi=(0, t_{0}, x_{1},t_{1},  \dots,x_{M-1}, t_{M-1}, 1)$ and 
$\pi'=(0, t_{0}', x_{1}',t_{1}',  \dots,x_{M'-1}', t_{M'-1}, 1)$ with infinitesimal mesh, let $x_{i}''$ for $i\leq M''$ be an enumeration\footnote{To make this enumeration effective, work with the approximations $[x_{i}](2^{M})$ and $[x_{i}'](2^{M'})$ and note that the difference is infinitesimal in the below steps.} of all $x_{i}$ and $x_{i}'$ in increasing order.  
Let $t_{i}''$ and $t_{i}'''$ for $i\leq M''$ be the associated $t_{i}$ and $t_{i}''$ with repetitions (corresponding to $x_{i}''$) of the latter to obtain a list of length $M''$.    
Then we have that
\begin{align*}\textstyle
|S_{\pi}(f)-S_{\pi'}(f)| &\textstyle=|\sum_{i=0}^{M-1}f(t_{i}) (x_{i}-x_{i+1}) -\sum_{i=0}^{M'-1}f(t_{i}') (x_{i}'-x_{i+1}') |\\
&\textstyle=|\sum_{i=0}^{M''-1}f(t_{i}'') (x_{i}''-x_{i+1}'') -\sum_{i=0}^{M''-1}f(t_{i}''') (x_{i}''-x_{i+1}'') |\\
&\textstyle=|\sum_{i=0}^{M''-1}(f(t_{i}'')-f(t_{i}'''))\cdot (x_{i}''-x_{i+1}'') \\
&\textstyle\leq \sum_{i=0}^{M''-1}|f(t_{i}'')-f(t_{i}''')| \cdot|x_{i}''-x_{i+1}''| \leq   \sum_{i=0}^{M''-1}\eps_{0} \cdot|x_{i}''-x_{i+1}''|\approx 0,
\end{align*}
where $\eps_{0}:=\max_{i\leq M''}|f(t_{i}'')-f(t_{i}''')|$ is an infinitesimal due to the (uniform nonstandard) continuity of $f$ and the definition of $\pi, \pi'$.  
Hence, we obtain $\CRI_{\ns}$, from which we now derive $\CRI_{\ef}(t)$.  To this end, we show that $\CRI_{\ns}$ can be brought into the normal form for applying Corollary \ref{consresultcor}.  
First of all, we resolve `$\approx$' in the antecedent of $\CRI_{\ns}$ as follows:
\be\label{first}\textstyle
(\forall x, y\in [0,1])[(\forall^{\st}N)|x- y|\leq \frac{1}{N} \di (\forall ^{\st}k)|f(x)- f(y)|\leq \frac{1}{k}].
\ee
Bringing all standard quantifiers in \eqref{first} outside the square brackets, we obtain:
\be\label{first2}\textstyle
(\forall^{\st}k)(\forall x, y\in [0,1])(\exists^{\st}N)[|x- y|\leq \frac{1}{N} \di |f(x)- f(y)|\leq \frac{1}{k}].
\ee
Since the formula in square brackets is internal in \eqref{first2}, we may apply (the contraposition of) idealisation \textsf{I} and obtain 
\be\label{first3}\textstyle
(\forall^{\st}k)(\exists^{\st}x^{0^{*}})(\forall x, y\in [0,1])(\exists N\in x)[|x- y|\leq \frac{1}{N} \di |f(x)- f(y)|\leq \frac{1}{k}].
\ee
By defining $N^{0}:=\max_{i<|x|}x(i)$, \eqref{first3} becomes the following:
\be\label{first333333}\textstyle
(\forall^{\st}k)(\exists^{\st}N')(\forall x, y\in [0,1])(\exists N\leq N')[|x- y|\leq \frac{1}{N} \di |f(x)- f(y)|\leq \frac{1}{k}], 
\ee
which immediately yields, due to the `monotone nature' of the formula, that
\be\label{first4}\textstyle
(\forall^{\st}k)(\exists^{\st}N)\big[(\forall x, y\in [0,1])[|x- y|\leq \frac{1}{N} \di |f(x)- f(y)|\leq \frac{1}{k}]\big].
\ee
Since the formula in (big) square brackets is internal, we may apply $\HAC_{\INT}$ and obtain standard $\Phi^{0\di 0^{*}}$ such that $N\in \Phi(k)$ in \eqref{first4}.  
By defining $g(k)$ to be the maximum of all components of $\Phi(k)$, i.e.\ $g(k):=\max_{i<|\Phi(k)|}\Phi(k)(i)$, we obtain:
\be\label{first5}\textstyle
(\exists^{\st}g)(\forall^{\st}k)\big[(\forall x, y\in [0,1])[|x- y|\leq \frac{1}{g(k)} \di |f(x)- f(y)|\leq \frac{1}{k}]\big], 
\ee  
and let $A(g, k, f)$ be the (internal) formula in big square brackets.  
Similarly, the consequent of $\CRI_{\ns}$ yields:
\be\textstyle\label{second}
(\forall^{\st}k')(\exists^{\st}N')\big[(\forall \pi, \pi' \in P([0,1]))(\|\pi\|,\| \pi'\|\leq \frac{1}{N'}  \di |S_{\pi}(f)- S_{\pi}(f)|\leq \frac{1}{k'} )\big],
\ee
where $B(k', N', f)$ is the (internal) formula in square brackets. Then $\CRI_{\ns}$ implies 
\be\label{second2}
(\forall^{\st}  g, k')(\forall f)(\exists^{\st} N', k)[A(g, k, f)\di B(k', N', f)].
\ee
Applying (the contraposition of) idealisation \textsf{I} to \eqref{second2}, we obtain 
\be\label{second22}
(\forall^{\st}  g, k')(\exists^{\st}x^{0^{*}})(\forall f)(\exists N', k\in x)[A(g, k, f)\di B(k', N', f)], 
\ee  
which immediately yields, by defining $l^{0}:=\max_{i<|x|}x(i)$, that
\be\label{second22334}
(\forall^{\st}  g, k')(\exists^{\st}l)(\forall f)(\exists N', k\leq l)[A(g, k, f)\di B(k', N', f)], 
\ee  
Furthermore, by the monotone behaviour of $B(k', \cdot ,f)$, we have that:
\be\label{second23}
(\forall^{\st}  g, k')(\exists^{\st}N')(\forall f)(\exists k)[A(g, k, f)\di B(k', N', f)].
\ee  
Applying Corollary \ref{consresultcor} to \eqref{second23}, the system $\textup{\textsf{E-PRA}}^{\omega*}$ proves
\be\label{tochie}
(\forall   g, k')(\exists N'\in  t( g, k'))(\forall f)(\exists k)[A(g, k, f)\di B(k', N', f)],
\ee
for some term $t$ from the original language.  
Define $u(g, k')$ to be the maximum of the components of $t( g, k')$, i.e.\ $u(g, k'):=\max_{i<|t(g,k')|}t(g, k')(i)$, yielding:       
\be\label{tochie2}
(\forall  f, g, k')(\exists k)[A(g, k, f)\di B(k', u(g, k'), f)], 
\ee
again by the special structure of $B$.  Bringing all quantifiers inside again, we obtain   
\be\label{crux}
(\forall  f, g)[(\forall k)A(g, k, f)\di (\forall k')B(k', u( g, k'), f)], 
\ee
which is exactly $\CRI_{\ef}(u)$ by the definitions of $A$ and $B$.  
\end{proof}
Note that the actual computation in $\CRI_{\ef}(t)$ only takes place on the modulus $g$. 
Thanks to the previous theorem, we can define (inside \textsf{E-PRA}$^{\omega*}$) \emph{the} Riemann integral functional\footnote{In particular, $[I(f,0, x)](k)$ is defined as $\sum_{i=0}^{i(x)} f(\frac{i}{2^{k}}) \frac{1}{2^{k}} $ where $i(x)$ 
is the least partition point larger than $x$.  To guarantee that $I(f,0, x)$ converges `fast enough', one considers $[I(f,0, x)](t(g,k))$ where $g$ is a modulus of uniform continuity of $f$ and $t$ is the term from Theorem~\ref{varou}.} $I(f, 0,x)$ 
which takes as input a modulus of uniform continuity $g$ for $f$ on $[0,1]$ and outputs $\int_{0}^{x}f(t) dt$. 
The aforementioned results are not really shocking, but what \emph{is} surprising is the following observation.
\begin{enumerate}
\item[a)] The straightforward derivation of the effective $\CRI_{\ef}$ from the \emph{nonstandard} $\CRI_{\ns}$, especially given the claims from Section \ref{foef} regarding the non-constructive nature of Nonstandard Analysis.  
\item[b)] The \emph{uniformity} of the derivation from a), i.e.\ a similar derivation should work for many other pairs of nonstandard and effective theorems.      
\end{enumerate}
With regard to b), we mention the following two examples.
\begin{exa}\rm
Two other theorems which can be treated exactly as in Theorem~\ref{varou} are the \emph{Weierstra\ss~approximation theorem} (see \cite{simpson2}*{IV.2.4}) and the statement that \emph{every uniformly continuous function on $[0,1]$ has a supremum}. 
We discuss the latter in more detail:  first of all, a nonstandard uniformly continuous function $f$ satisfies 
\be\label{frdfg}
(\forall x\in [0,1], M\in \Omega)[f(x)\lessapprox \sup(f, M)]
\ee
where $\sup(f, M)=\|f\|_{M}:=\max_{i\leq 2^{M}}[f(\frac{i}{2^{M}})](2^{M})$ and $[x](k)$ is the $k$-th approximation of the real $x$.  
Secondly, the effective version, which can be derived from the nonstandard version in exactly the same way as in Theorem~\ref{varou}, is as follows.
\begin{thm}[$\textsf{SUP}_{\ef}(t)$]
For every $f:\R\di \R$ with modulus of uniform continuity $g$ on $[0,1]$, we have $  (\forall x\in [0,1], N\geq  t(g, n))\big[ f(x)\leq \sup(f, N)+\frac{1}{n} \big]$.
\end{thm}
As for the Riemann integral, we can now define the functional $\|f\|=\sup_{x\in [0,1]}f(x)$ where the modulus of uniform continuity is implicit.  
We will study \eqref{frdfg} in the context of (nonstandard) compact spaces in Section \ref{compaq}.
\end{exa}
Finally, we discuss the following remark regarding the use of idealisation \textsf{I}, $\HAC_{\INT}$, and term extraction as in Corollary \ref{consresultcor}.  
\begin{rem}[Using $\HAC_{\INT}$ and $\textsf{I}$]\label{simply}\rm
By definition, $\HAC_{\INT}$ produces a functional of type $\sigma\di \tau^{*}$ which outputs a \emph{finite sequence} of witnesses.  
Now, in the proof of Theorem~\ref{varou}, $\HAC_{\INT}$ is applied to \eqref{first4} to obtain $\Phi^{0\di 0^{*}}$, and from the latter, the function $g^{1}$ is defined as follows: $g(x):= \max_{i<|\Phi(x)|}\Phi(x)(i)$.  In particular, $g$ satisfies \eqref{first5}, and provides a \emph{witnessing functional}, due to the `monotone' nature of the internal formula in \eqref{first4}.  In general, $\HAC_{\INT}$ provides a \emph{witnessing functional} assuming (i) $\tau=0$ in $\HAC_{\INT}$ and (ii) the formula $\varphi$ from $\HAC_{\INT}$ is `sufficiently monotone' as in: 
$(\forall x^{\sigma},n^{0},m^{0})\big([n\leq_{0}m \wedge\varphi(n,x)] \di \varphi(m,x)\big)$.    

\smallskip

A similar observation applies to idealisation \textsf{I};  indeed, consider \eqref{first2} and note that the internal formula in the latter is monotone as above.  Taking the maximum of $x$ from \eqref{first3} as $N:=\max_{i<|x|}x(i)$, one can drop the quantifier `$(\exists N\in x)$' in \eqref{first3} to obtain \eqref{first333333} and \eqref{first4}.  

\smallskip

A similar observation applies to terms obtained by \emph{term extraction};  indeed, consider \eqref{tochie} and note that the internal formula in the latter is monotone as above.  Taking the maximum of $t$ from \eqref{tochie} as $u(g, k'):=\max_{i<|t(g,k')|}t(g, k')(i)$, one can drop the quantifier `$(\exists N'\in t(g,k'))$' in \eqref{tochie} to obtain \eqref{tochie2} and \eqref{crux}.  

\smallskip

To save space in proofs, we will sometimes skip the (obvious) step involving the maximum of the finite sequences, when applying $\HAC_{\INT}$, $\textsf{I}$, and term extraction. 
\end{rem}
\subsubsection{Riemann integration in $\H$}
In this section, we show that Theorem \ref{varou} also goes through for $\H$ and Heyting arithmetic. 
We point out that the convention from Remark \ref{simply} also applies to the axioms of $\H$ from Definition \ref{flah} and Theorem \ref{consresult2}.  In particular, we will sometimes skip the obvious step 
involving the maximum, as discussed in the former remark.     

\smallskip

Of course, it is a theorem of constructive mathematics (see e.g.\ \cite{bish1}*{p.47})) that a uniformly continuous function (with a modulus) is Riemann integrable on compact intervals.  
\emph{What is surpring} is that the proof of Theorem~\ref{varou} still goes through in a constructive setting, as it seems we used a number of \emph{non-constructive} logical laws in the previous proof, like \emph{independence of premises} to bring the quantifier `$(\exists^{\st} N')$' to the front as in \eqref{second2}. 
As it turns out, this is not problematic, as we show now.  
\begin{cor}\label{cordejedi}
Theorem \ref{varou} also goes through constructively, i.e.\ we can prove $\CRI_{\ns}$ in $\textup{\textsf{H}} $ and a term $t$ can be extracted such that $\textup{\textsf{E-HA}}^{\omega*} $ proves $\CRI_{{\ef}}(t)$.  
\end{cor}
\begin{proof}
First of all, it is clear that the proof of $\CRI_{\ns}$ in the theorem also goes through in $\textup{\textsf{H}}$.  
We now show that $\CRI_{\ns}$ can be brought into the normal form \eqref{third5} (which is essentially \eqref{second23}) \emph{inside the constructive system $\H$}.  
Hence, working in $\H$, consider the antecedent of $\CRI_{\ns}$ as in \eqref{first}.  The quantifier $(\forall^{\st}k)$ can be brought to the front as in \eqref{first2} in intuitionistic logic, i.e.\ we obtain 
\be\label{deep}\textstyle
(\forall^{\st}k)(\forall x, y\in [0,1])[(\forall^{\st}N)|x- y|\leq \frac{1}{N} \di |f(x)- f(y)|\leq \frac{1}{k}].
\ee
To bring the quantifier $(\forall^{\st}N)$ to the front, we apply the axiom $\textsf{HGMP}^{\st}$ included in $\H$ to \eqref{deep}, which is sufficiently internal, to obtain   
\be\label{deep2}\textstyle
(\forall^{\st}k)(\forall x, y\in [0,1])(\exists^{\st}N')[(\forall N\leq N')|x- y|\leq \frac{1}{N'} \di |f(x)- f(y)|\leq \frac{1}{k}],
\ee
which immediately yields 
\be\label{deep3}\textstyle
(\forall^{\st}k)(\forall x, y\in [0,1])(\exists^{\st}N)[|x- y|\leq \frac{1}{N} \di |f(x)- f(y)|\leq \frac{1}{k}].
\ee
The system $\H$ also includes $\textsf{NCR}$, which implies \emph{the contraposition of idealisation}~$\textsf{I}$.  
Hence, applying $\textsf{NCR}$ to \eqref{deep3}, we obtain
\be\label{deep4}\textstyle
(\forall^{\st}k)(\exists^{\st}N')(\forall x, y\in [0,1])(\exists N'\leq N)[|x- y|\leq \frac{1}{N'} \di |f(x)- f(y)|\leq \frac{1}{k}],
\ee
which again immediately yields
\be\label{deep5}\textstyle
(\forall^{\st}k)(\exists^{\st}N)(\forall x, y\in [0,1])[|x- y|\leq \frac{1}{N} \di |f(x)- f(y)|\leq \frac{1}{k}].
\ee
Since $\HAC_{\INT}$ is also included in $\H$, we obtain \eqref{first5}, and \eqref{second} is proved similarly inside $\H$.  
So far, we have shown that $\H$ proves 
\be\label{third2}
(\forall f)[(\exists^{\st}g)(\forall^{\st}k)A(g, k, f)\di (\forall^{\st}k')(\exists^{\st}N')B(k', N', f)], 
\ee
which immediately yields (due to intuitionistic logic) that
\be\label{third21}
(\forall f)(\forall^{\st} g, k')[(\forall^{\st}k)A(g, k, f)\di (\exists^{\st}N')B(k', N', f)], 
\ee
To bring the quantifier $(\exists^{\st}N')$ to the front, we apply the axiom $\textsf{HIP}_{\forall^{\st}}$ included in $\H$ to \eqref{third21}, which is sufficiently internal, to obtain 
\be\label{third3}
(\forall f)(\forall^{\st} g, k')(\exists^{\st}N)[(\forall^{\st}k)A(g, k, f)\di (\exists N'\leq N)B(k', N', f)], 
\ee
which again yields, due to the monotone behaviour of $B$, that
\be\label{third4}
(\forall f)(\forall^{\st} g, k')(\exists^{\st}N)[(\forall^{\st}k)A(g, k, f)\di B(k', N, f)].
\ee
To bring the quantifier $(\forall^{\st}k)$ to the front, we apply the axiom $\textsf{HGMP}^{\st}$ included in $\H$ to \eqref{third4}, which is sufficiently internal, to obtain 
\be\label{third5}
(\forall^{\st} g, k')(\forall f)(\exists^{\st}N, k)[(\forall k''\leq k)A(g, k'', f)\di B(k', N, f)], 
\ee
which is essentially \eqref{second2}.  Similar to the way the latter gives rise to \eqref{second23}, apply $\NCR$ (which supplies the classical contraposition of idealisation $\textsf{I}$) to \eqref{third5} to obtain  
\be\label{third6}
(\forall^{\st} g, k')(\exists^{\st}N)(\forall f)(\exists k)[(\forall k''\leq k)A(g, k'', f)\di B(k', N, f)], 
\ee
Now apply Theorem \ref{consresult2} to \eqref{third6} to obtain $\CRI_{\ef}(t)$ in the same way as in the proof of the theorem.  
\end{proof}
We repeat that it is rather surprising that the system $\H$ includes exactly the `non-constructive' axioms (listed in Definition \ref{flah}) required to 
bring nonstandard definitions into the associated normal form \emph{in a constructive setting}.  We discuss this `non-constructive' status in more detail in the following remark.  
\begin{rem}\label{dingdong}\rm
First of all, note that the axioms from Definition~\ref{flah} only apply to a \emph{part} of the universe of objects, namely \emph{the standard ones}.  This partiality explains why these axioms can be `constructive' at all, in the sense that $\H$ and $\textsf{E-HA}^{\omega*}$ prove the same internal sentences.  

\smallskip

Secondly, Nelson states in \cite{wownelly}*{p.\ 1166} that in \textsf{IST}, every specific object of conventional mathematics is a standard set.  
In other words, the universe of standard objects may be viewed as an attempt at isolating the `actual objects of mathematics' from the formalism in which they are studied. 

\smallskip

Thirdly, Bishop discusses in \cite{nukino}*{p.\ 56} the concept of \emph{numerical implication}, an alternative constructive notion of implication based on G\"odel's Dialectica interpretation; he notes that \emph{in practice} the usual definition of implication amounts to numerical implication.   
Furthermore, Bishop conjectures that numerical implication can be derived constructively, while his derivation in \cite{nukino} uses non-constructive principles like \emph{independence of premises} and \emph{Markov's principle}.  

\smallskip

In light of these three observations, the first two `non-constructive' axioms from Definition \ref{flah} are nothing more than a formalisation of the claims made by Nelson and Bishop regarding mathematical practice.    
\end{rem}
\subsubsection{Herbrandisation}
In this section, we introduce the notion of \emph{Hebrandisation}.  
Intuitively speaking, the latter is a `more constructive' version of $\CRI_{\ef}(t)$ which implies $\CRI_{\ns}$ in $\P_{0}$.  
Indeed, the results so far obtained in this section suggest that \emph{if one shakes Nonstandard Analysis in the right way, constructive mathematics will fall out}.  
It is a natural question, especially in the light of RM, if there is a reversal here. In particular, is there a way to make Nonstandard Analysis `fall out of' some kind of constructive mathematics?  
We shall provide a positive answer to this question in Corollary \ref{somaar}.  To prove the latter, we first need to consider a slightly modified proof of Theorem \ref{varou}.  
\begin{rem}[Herbrandisation]\label{herbrand}\rm
Consider the proof of Theorem \ref{varou}, in particular the step from \eqref{second22} to \eqref{second23}.  
Instead of `forgetting' the information regarding $k$ in \eqref{second22}, 
we apply Corollary \ref{consresultcor} directly to the formula \eqref{second22} to obtain a term $t$ such that $\textup{\textsf{E-PRA}}^{\omega*}$ proves (where $A, B$ are as in the proof of Theorem \ref{varou}):
\[
(\forall  g, k')(\exists l\in t(g, k'))(\forall f)(\exists N', k\leq l)[A(g, k, f)\di B(k', N', f)].  
\]  
Now define $s(g, k')$ as the maximum of all entries of $t(g, k')$ and note that 
\be\label{fsecond}
(\forall  g,  f,k')\big[(\forall k\leq s(g, k'))A(g, k, f)\di B(k', s(g, k'), f)\big].  
\ee
It is insightful to write out \eqref{fsecond} in full as follows:
\begin{align}\textstyle
\textstyle~(\forall f, ~&g ,k')\Big[(\forall k\leq s(g,k'))(\forall \textstyle x, y \in [0,1])(|x-y|<\frac{1}{g(k)} \di |f(x)-f(y)|\leq\frac{1}{k})\notag\\
&\label{FEST}\textstyle\di  (\forall \pi, \pi' \in P([0,1]))\big(\|\pi\|,\| \pi'\|< \frac{1}{s(g,k')}  \di |S_{\pi}(f)- S_{\pi}(f)|\leq \frac{1}{k'} \big)  \Big].
\end{align}
We refer to \eqref{FEST} as the \emph{Herbrandisation of $\CRI$}, denoted by $\CRI_{\her}(s)$.  
The consequent and antecedent of the latter are connected in such a way that pushing $(\forall k')$ into the consequent requires dropping the bound $s(g, k')$ in the antecedent;  this however breaks the aforementioned connection.      
\end{rem}
We now show that a proof of the Herbrandisation \eqref{FEST} can be converted into a proof of the nonstandard version $\CRI_{\ns}$.  
Combined with Theorem \ref{varou} and Remark~\ref{herbrand}, we observe that $\CRI_{\ns}$ and $\CRI_{\her}(t)$ have the same computational content in that 
a proof of one theorem can be converted into a proof of the other one.  
\begin{cor}\label{somaar} Let $t$ be a term in the internal language. 
A proof inside $\textup{\textsf{E-PRA}}^{\omega*}$ of the Herbrandisation $\CRI_{\her}(t)$, can be converted into a proof inside $\P_{0}$ of $\CRI_{\ns}$.  
\end{cor}
\begin{proof}
First of all, recall that any term $t$ of the internal language is standard in $\P_{0}$ due to Definition~\ref{debs}.  Hence, if $\textup{\textsf{E-PRA}}^{\omega*}$ proves $\CRI_{\her}(t)$, then $\P_{0}$ proves $\CRI_{\her}(t)\wedge \st(t)$. 
Now fix a function $f$, nonstandard continuous on $[0,1]$, and obtain \eqref{first5} as in the proof of the theorem.  For $g$ as in \eqref{first5}, we obtain the antecedent of $\CRI_{\her}(t)$ for any standard $k'$.  Hence, we also have the consequent of $\CRI_{\her}(t)$ for any standard $k'$, which immediately yields the 
nonstandard Riemann integrability of $f$ on $[0,1]$.  Note that \emph{for both the antecedent and consequent cases} it is essential that $t$ maps standard inputs to standard outputs.    
\end{proof}
In the following sections, we shall discuss the Herbrandisations of a number of nonstandard theorems (but omit others in the interest of space).  
In general, the \emph{Herbrandisation} of a theorem $T$ of (pure) Nonstandard Analysis is defined as follows. First of all, if $\P_{0}$ proves $T$ and also that the latter has the \emph{equivalent} normal form $(\forall^{\st}x)(\exists^{\st} y)\varphi(x, y)$, then the Herbrandisation of $T$ is $(\forall x)(\exists y\in t(x))\varphi(x, y)$ for a term $t$ (obtained via Corollary \ref{consresultcor}).  

\smallskip

Secondly, if $\P_{0}$ does not prove $T$, but does prove that the latter has the \emph{equivalent} normal form $(\forall^{\st}x)(\exists^{\st} y)\varphi(x, y)$, then the Herbrandisation of $T$ is $(\exists \Phi)(\forall x)(\exists y\in \Phi(x))\varphi(x, y)$.  Note that in this way $\P_{0}$ proves $(\exists^{\st} \Phi)(\forall x)(\exists y\in \Phi(x))\varphi(x, y)\di T$.

\smallskip 

Finally, we discuss various interpretations of the notion of Herbrandisation.  
\begin{rem}[Use principle]\label{iamauser}\rm
The \emph{Use principle} is a basic result in Computability theory (see e.g.\ \cite{zweer}*{Theorem~1.9, p.\ 50}) 
which states that for an oracle Turing machine which halts, the output only depends on a finite subset of the oracle.  
In other words, such a Turing machine only `uses' a finite subset of its oracle.  

\smallskip

Analogously, the consequent of $\CRI_{\ef}(t)$ as in \eqref{EST} also only `uses' finitely many instances of the antecedent in the following specific way:      
let $(\forall k)A(k, f, g)$ and $(\forall n)B(n, f, g)$ be the antecedent and consequent of \eqref{EST} (with `$(\forall n)$' brought into the consequent).  Then for fixed $n_{0}, f_{0}, g_{0}$, to guarantee that $B(n_{0}, f_{0}, g_{0})$, we do 
not require $(\forall k)A(k, f, g)$, but only $(\forall k\leq k_{0})A(k, f, g)$ for some fixed $k_{0}$ (which of course depends on $n_{0}$).  
The Herbrandisation $\CRI_{\her}(t)$ does nothing more than make this dependence explicit, i.e.\ we know `how much' of the antecedent is `used' by the consequent (in an exact effective way).    
\end{rem}
The `quantitative' nature of the Herbrandisation leads us to the following remark.  
\begin{rem}[Hard versus soft analysis]\rm
Hardy makes a distinction between `hard' and `soft' analysis in \cite{hardynicht}*{p.\ 64}.  
Intuitively speaking, soft (resp.\ hard) analysis deals with qualitative (resp.\ quantitative) information and continuous/infinite (resp.\ discrete/finite) objects.  
Tao has on numerous occasions discussed the connection between so-called hard and soft analysis (\cite{taote, tao2}), and how Nonstandard Analysis connects the two.   

\smallskip

The Herbrandisation $\CRI_{\her}(t)$ is quantitative or `hard' in nature, as it states that to compute the Riemann integral of $f$ up to precision $\frac{1}{k'}$ via $S_{\pi}(f)$, we should use a partition $\pi$ finer than $s(g, k')$ and $f$ should only have `jumps' (or discontinuities) in its graph less than $s(g, k')$, where $g$ witnesses this partial continuity. 

\smallskip

By combining Theorem \ref{varou} and Corollary \ref{somaar}, we can convert a proof of $\CRI_{\ns}$ into a proof $\CRI_{\her}(t)$ \emph{and vice versa}.  
The forward conversion can be done via the template $\CI$ in Section \ref{detail};  the reverse conversion (and the associated algorithm) can be `read off' from the proof of Corollary \ref{somaar}.    
Thus, we observe the first steps towards an `algorithmic two-way street' between soft analysis (in the guise of $\CRI_{\ns}$) and hard analysis (in the guise of $\CRI_{\her}(t)$).  
\end{rem}
\begin{rem}[Real real analysis]\rm
In light of the previous remarks, we can view $\CRI_{\ef}(t)$ as a `global' statement (as it requires $(\forall k)A(k,f,g)$ to provide any information), 
while $\CRI_{\her}(t)$ is a `pointwise' statement (as it already provides information if $(\forall k\leq s(g,k'))A(k,f,g)$).  In other words, $\CRI_{\her}(t)$ even provides information 
about certain \emph{discontinuous} $f:\R\di \R$.  This need not be surprising as the following function $f_{0}$ has many `jumps' (or discontinuities), but is nonstandard continuous.  
\[
f_{0}(x):=
\begin{cases}\textstyle
0 &  (\exists i \leq M)\big(\frac{2i}{2M}\leq [x](2^{M}) < \frac{2i+1}{2M}\big)\\
\frac{1}{M} & \text{otherwise}
\end{cases}\qquad \qquad (M\in \Omega).
\]
In particular, $\CRI_{\ns}$ also provides information about functions which are discontinuous in a certain sense.        
in this way, the nonstandard version $\CRI_{\ns}$ and the Herbrandisation \eqref{FEST} play the following important foundational role regarding idealising assumptions in physics:  

\smallskip

We cannot be sure that $\R\di \R$-functions originating from measurements in physics are continuous (nonstandard or $\eps$-$\delta$), 
due to finite measurement precision and conceptual limitations (the Planck scale).  
But how can we then apply the usual theorems from calculus, e.g.\ involving Riemann integration?  In general, how can the thoroughly idealised results of mathematics apply to physics, given that we cannot verify if the (infinitary) conditions 
of mathematical theorems are met due to finite measurement precision and conceptual limitations?
  
\smallskip

The answer to these grand questions seems quite simple in light of the connection between $\CRI_{\ns}$ and its Herbrandisation:  we can just prove results from `pure' Nonstandard Analysis like $\CRI_{\ns}$ and obtain the associated Herbrandisation, which     
also provides approximations to the Riemann integral, should the function $f$ be discontinuous (or just unknown to be continuous beyond a certain precision).  In other words, it seems that theorems from pure Nonstandard Analysis are `robust' in the sense that their associated Herbrandisations still provide approximate/partial results if the conditions of the original theorem are only partially/approximately met.  In our opinion, this robustness partially explains `that other unreasonable effectiveness', 
namely why mathematics is so effective, as in useful, in the natural sciences, as discussed by Wigner in \cite{wignman}.             
Ironically\footnote{For the reader to properly appreciate the irony here, \cite{kluut}*{p.\ 513} should be consulted.}, this robustness also constitutes a formalisation of the following observation made by Bishop.
\begin{quote}
[\dots ] whenever you have a theorem: $B\di A$, then you suspect that you have a theorem: $B$ is approximately true $\di$  $A$ is approximately true. (\cite{kluut}*{p.\ 513})
\end{quote}
\end{rem}
In conclusion, the \emph{nonstandard} version $\CRI_{\ns}$ gives rise to the \emph{effective} version $\CRI_{\ef}$ in an elegant and algorithmic fashion.  
Conversely, to re-obtain the former, the \emph{Herbrandisation} $\CRI_{\her}$ is needed, which provides \emph{approximate} results if the conditions in the consequent are only partially met.  

\subsection{Uniform limit theorem}
In this section, we study the \emph{uniform limit theorem} $\textsf{ULC}$ from \cite{munkies}*{Theorem 21.6}, which states that if a sequence of continuous functions \emph{uniformly} converges to another function, the latter is also continuous.  
We will obtain the effective version of $\ULC$ from the nonstandard version using $\P_{0}$.  
We adopt the usual definition of nonstandard convergence, namely as follows.      
\bdefi[Nonstandard convergence]\label{nsconv}~
\begin{enumerate}
\item A sequence $x_{(\cdot)}^{0\di 1}$ \emph{nonstandard converges} to $x^{1}$ if $(\forall N\in \Omega)(x\approx x_{N})$.  
\item A sequence $f_{(\cdot)}^{0\di (1\di 1)}$ \emph{nonstandard uniformly converges} to $f^{1\di 1}$ on $[0,1]$ if $(\forall   x\in [0,1], N\in \Omega)(f_{N}(x)\approx f(x))$.
\end{enumerate}
\edefi
Hence, we have the following nonstandard and effective version of $\ULC$.  
\begin{thm}[$\ULC_{\ns}$] For all $f_{n}, f:\R\di \R$, if $f_{n}$ is nonstandard continuous on $[0,1]$ for all standard $n$, and $f_{n}$ nonstandard uniformly converges to $f$ on $[0,1]$, then $f$ is also nonstandard continuous on $[0,1]$.   
\end{thm}
\begin{thm}[$\ULC_{\ef}(t)$] For all $f_{(\cdot)},g_{(\cdot)},h$, and $f$, if $f_{n}$ is continuous on $[0,1]$ with modulus $g_{n}$ for all $n$, 
and $f_{n}$ uniformly converges to $f$ on $[0,1]$ with modulus $h$, then $f$ is continuous on $[0,1]$ with modulus $t(g_{(\cdot)}, h)$.   
\end{thm}
We have the following theorem.
\begin{thm}\label{floggen}
From the proof of $\ULC_{\ns}$ in $\P_{0}$, a term $t$ can be extracted such that $\textup{\textsf{E-PRA}}^{\omega*} $ proves $\ULC_{{\ef}}(t)$.  
\end{thm}
\begin{proof}
We first show that $\ULC_{\ns}$ can be proved in $\textup{\textsf{E-PRA}}^{\omega*}_{\st} $.  To this end, let $f$ and $f_{n}$ be as in the antecedent of $\ULC_{\ns}$.  
Now fix standard $x_{0}\in [0,1]$ and any $y_{0}\in [0,1]$ such that $x_{0}\approx y_{0}$.  By assumption, we have $(\forall^{\st}n)(f_{n}(x_{0})\approx f_{n}(y_{0}))$, implying  $(\forall^{\st}n, k)(|f_{n}(x_{0})- f_{n}(y_{0})|\leq \frac{1}{k})$.  
By overspill, there is $N_{0}\in \Omega$ such that $(\forall  n, k\leq N_{0})(|f_{n}(x_{0})- f_{n}(y_{0})|\leq \frac{1}{k})$.  
Again by assumption, we have $f(x_{0})\approx f_{N_{0}}(x_{0})\approx f_{N_{0}}(y_{0})\approx f(y_{0})$, which implies the nonstandard continuity of $f$.       

\smallskip

Secondly, we show that $\ULC_{\ns}$ can be brought into the right normal form for applying Corollary \ref{consresultcor}.  
Making explicit all standard quantifiers in the first conjunct in the antecedent of $\ULC_{\ns}$, we obtain: 
\[\textstyle
(\forall^{\st}x^{1}\in [0,1], n^{0})(\forall y^{1}\in [0,1])[(\forall^{\st}N)|x-y|\leq \frac{1}{N} \di (\forall^{\st}k)|f_{n}(x)- f_{n}(y)|\leq \frac{1}{k}].
\]
Bringing the standard quantifiers to the front as much possible, this becomes
\[\textstyle
(\forall^{\st}x^{1}\in [0,1], n^{0}, k^{0})(\forall y^{1}\in [0,1])(\exists^{\st}N)[|x-y|\leq \frac{1}{N} \di |f_{n}(x)- f_{n}(y)|\leq \frac{1}{k}].
\]
Since the formula in square brackets is internal, we may apply (the contraposition of) idealisation \textsf{I} and obtain 
\[\textstyle
(\forall^{\st}x^{1}\in [0,1], n^{0}, k^{0})(\exists^{\st} m^{0})(\forall y^{1}\in [0,1])(\exists N\leq m)[|x-y|\leq \frac{1}{N} \di |f_{n}(x)- f_{n}(y)|\leq \frac{1}{k}], 
\]
which immediately yields 
\[\textstyle
(\forall^{\st}x^{1}\in [0,1], n^{0}, k^{0})(\exists^{\st} N^{0})(\forall y^{1}\in [0,1])[|x-y|\leq \frac{1}{N} \di |f_{n}(x)- f_{n}(y)|\leq \frac{1}{k}].
\]
Now apply $\HAC_{\INT}$ to obtain a standard functional $\Psi$ such that $(\exists N\in \Psi(x, n, k))$.  Define $g_{n}(x, k)$ as $\max_{i<|\Psi(x, n, k)|}\Psi(x, n, k)(i)$ and we obtain: 
\[\textstyle
(\exists^{\st} g_{(\cdot)})(\forall^{\st}x^{1}\in [0,1], n^{0}, k^{0})\big[(\forall y^{1}\in [0,1])[|x- y|\leq\frac{1}{g_{n}(x,k)} \di | f_{n}(x)- f_{n}(y)| \leq \frac{1}{k}]\big].
\]
For brevity let $A(\cdot)$ be the formula in big square brackets.  Similarly, the second conjunct of the antecedent of $\ULC_{\ns}$ yields
\[\textstyle
(\exists^{\st} h)(\forall^{\st}k)\big[(\forall x\in [0,1])(\forall N\geq h(x, k))(| f_{N}(x)- f(x)|\leq \frac{1}{k})\big],    
\]
where $B(\cdot)$ is the formula in big square brackets.  Lastly, the consequent of $\ULC_{\ns}$, which is just the nonstandard continuity of $f$, implies 
\[\textstyle
(\forall^{st} k, x \in [0,1])(\exists^{\st}N)\big[(\forall  y \in [0,1])(|x-y|<\frac{1}{N} \di |f(x)-f(y)|<\frac{1}{k})\big],
\]
where $E(\cdot)$ is the formula in big square brackets.  Hence $\ULC_{\ns}$ implies that 
\begin{align*}
(\forall f, f_{n})\Big[\big[(\exists^{\st} g_{(\cdot)})(\forall^{\st}x^{1}\in [0,1], &n^{0}, k^{0})A(f_{n}, g_{n}, x, n, k) \wedge (\exists^{st} h)(\forall^{\st}k')B(f, f_{n}, h, k')  \big]\\
&\di (\forall^{\st} k'', x' \in [0,1])(\exists^{\st}N)E(f, N , k'', x')\Big].
\end{align*}
Bringing all standard quantifiers to the front, we obtain
\begin{align}
(\forall^{\st} g_{(\cdot)}, h&,k'', x'\in [0,1])\underline{(\forall f, f_{n})(\exists^{\st}x \in [0,1], n^{0}, k, k' ,  N^{0})}\label{fruk}\\
&\Big[\big[A(f_{n}, g_{n}, x, n, k) \wedge  B(f, f_{n}, h, k')  \big]\di E(f, N , k'', x')\Big],\notag
\end{align}
where the formula in big(gest) square brackets is internal.  Applying idealisation to the underlined quantifier alternation in \eqref{fruk}, we obtain 
\begin{align}
(\forall^{\st} g_{(\cdot)}, h&,k'', x'\in [0,1])(\exists^{\st}z){(\forall f, f_{n})(\exists (x \in [0,1], n^{0}, k, k' ,  N^{0}) \in z)}\label{fruk2}\\
&\Big[\big[A(f_{n}, g_{n}, x, n, k) \wedge  B(f, f_{n}, h, k')  \big]\di E(f, N , k'', x')\Big],\notag
\end{align}
As \eqref{fruk2} was proved in $\P_{0}$, Corollary \ref{consresultcor} tells us that $\textup{\textsf{E-PRA}}^{\omega*}$ proves 
\begin{align}
(\forall g_{(\cdot)}, h&,k'', x'\in [0,1])(\exists z\in t(g_{(\cdot)}, h, k'', x')){(\forall f, f_{n})(\exists (x \in [0,1], n^{0}, k, k' ,  N^{0}) \in z)}\notag\\
&\Big[\big[A(f_{n}, g_{n}, x, n, k) \wedge  B(f, f_{n}, h, k')  \big]\di E(f, N , k'', x')\Big],\label{frakker}
\end{align}
where $t$ is a term in the language of $\textup{\textsf{E-PRA}}^{\omega*}$.  
The formula \eqref{frakker} yields
\begin{align}
(\forall  g_{(\cdot)}, & h,k'', x'\in [0,1])(\exists  N^{0}\in s( g_{(\cdot)}, h, k'', x'))(\forall f, f_{n})(\exists x \in [0,1], n^{0}, k, k') \notag\\
&\Big[\big[A(f_{n}, g_{n}, x, n, k) \wedge B(f, f_{n}, h, k')  \big]\di E(f, N , k'', x')\Big],\label{frak}
\end{align}
where $s$ is obtained by ignoring all components of $t$ except those containing the potential witnesses to $N$.  Note that the special structure of $E$ plays a role.    
Now define $u( g_{(\cdot)}, h, k'', x')))$ to be the maximum of all components of $s(g_{(\cdot)}, h, k'', x')))$ and push all existential quantifiers in \eqref{frak} back inside the square brackets:  
\begin{align}
(\forall f, f_{n}, g_{(\cdot)},  h) &\Big[\big[(\forall x \in [0,1], n^{0}, k)A(f_{n}, g_{n}, x, n, k)\label{pierlala} \\
&\wedge (\forall k')B(f, f_{n}, h, k')  \big]\di (\forall k'', x'\in [0,1])E(f, u(g_{(\cdot)}, h, k'', x') , k'', x')\Big],\notag
\end{align}
which is exactly $\ULC_{\ef}(u)$, and we are done.
\end{proof}
As it turns out, there is a constructive version of the uniform limit theorem in Constructive Analysis (see \cite{bridge1}*{Prop.\ 1.12, p.\ 86}), and there is a somewhat similar version of the uniform limit theorem in computability theory (see \cite{yasugi1}*{Prop. 3.2}).  
\begin{cor}\label{cordejedi2}
The theorem also goes through constructively, i.e.\ we can prove $\ULC_{\ns}$ in $\textup{\textsf{H}} $ and a term $t$ can be extracted such that $\textup{\textsf{E-HA}}^{\omega*} $ proves $\ULC_{{\ef}}(t)$.  
\end{cor}
\begin{proof}
Similar to the proof of Corollary \ref{cordejedi}. 
\end{proof}
Following Remark \ref{herbrand}, the Herbrandisation $\ULC_{\her}(s)$ is defined as \eqref{frakker}.  
Note that once the correct Herbrandisation is found, re-obtaining the original nonstandard version is easy.  In particular, 
the proof of Corollary \ref{somaar} amounts to nothing more than the observation that the term $t$ in the Herbrandisation $\CRI_{\her}(t)$ is standard in $\P_{0}$, which yields the nonstandard definitions from the usual ones.    
\begin{cor}\label{somaar2} Let $t$ be a term in the internal language. 
A proof inside $\textup{\textsf{E-PRA}}^{\omega}$ of $\ULC_{\her}(t)$, can be converted into a proof inside $\P_{0}$ of $\ULC_{\ns}$. 
\end{cor}
\begin{proof}
Similar to the proof of Corollary \ref{somaar}.  In particular, note that the term $t$ is standard in $\P_{0}$ and observe that this yields $\ULC_{\ns}$ from $\ULC_{\her}(t)$.  
\end{proof}
The Herbrandisation $\ULC_{\her}(t)$ tells us `how much' continuity and convergence of $f_{n}$ we need to obtain `how much' partial pointwise continuity of $f$.   
In particular, the term $t$ provides a finite collection of points $x\in [0,1]$, precisions $k, k'$ and indices $n$ for which $f_{n}$ should satisfy the corresponding definitions of continuity and convergence (to $f$), to guarantee 
pointwise continuity of $f$ around $x'\in [0,1]$ up to precision $k''$ in an interval also determined by $t$.  

\smallskip

The uniform limit theorem merely serves as an example of the class of theorems that can be treated in a similar way as Theorem \ref{floggen}.  
Another more challenging example is \cite{rudin}*{Theorem 7.17, p.\ 152} as follows.
\begin{thm}\label{fryg}
Suppose ${f_n}$ is a sequence of functions, differentiable on  $[a, b]$, and ${f_n(x_0)}$ converges for some point $x_0$ on $[a, b]$. If  $f'_n$ converges uniformly on $[a, b]$, then ${f_n}$ converges uniformly to a function $f$, and  $f'(x) = \lim_{n\to \infty} f'_n(x)$ for $ x \in [a, b]$.
\end{thm}
Note that we can avoid talking about the existence of $f$ in the nonstandard version of Theorem \ref{fryg} by instead making use of $f_{N}$ for any nonstandard $N$.  

\smallskip

In conclusion, the two above case studies (Theorems \ref{varou} and \ref{floggen}) suggest that the `constructive' version of a theorem can be derived from the `nonstandard' version in a rather algorithmic fashion.   
The associated template $\CI$ shall be formulated in Section \ref{detail}.  Before that, we treat two more case studies, but in less detail than the above ones.  These two extra case studies are also interesting in their own right.    

\subsection{Fundamental theorem of calculus}
In this section, we treat the fundamental theorem of calculus (denoted $\FTC$;  see \cite{rudin}*{\S6}).  As is well-known, $\FTC$ states that differentiability and integration cancel each other out.  
We will formulate various nonstandard versions of $\FTC$ with proofs in $\P_{0}$, and obtain the associated effective versions.  
Our study of $\FTC$ is interesting in its own right, as it leads to a non-trivial extension of the scope of Corollary \ref{consresultcor}, due to Theorem \ref{necess}.  

\smallskip

We use the following definition of nonstandard differentiability.  
\bdefi[Nonstandard differentiability]\label{settleourdiff}
A function $f:\R\di \R$ is \emph{nonstandard differentiable} on $[0,1]$ if for all standard $k^{0}$:
\[
(\forall \eps, \eps'\approx 0)\textstyle(\forall x \in [-\frac{1}{k},1+\frac{1}{k}])\big[\eps, \eps'\ne 0\di \Delta_{\eps}(f(x))\approx \Delta_{\eps'}(f(x))\big],
\]
where $\Delta_{\eps}(f(x)):=\frac{f(x+\eps)-f(x)}{\eps}$.
\edefi  
We now study the part of $\FTC$ which states that $F'(x)=f(x)$ for continuous $f$ and the primitive $F(x):=\int_{a}^{x}f(x)dx$;   
the first nonstandard version of $\FTC$ is as follows, where $S_{\pi}(f, 0,x)$ is the Riemann sum corresponding to $f$ and $\pi$ \emph{limited to the interval $[0,x]$} as in Definition \ref{kunko}. 
\begin{thm}[$\FTC_{\ns}$]
For all $f:\R\di \R$ and all standard $k^{0}$, if $f$ is nonstandard uniform continuous on $[0,1]$, then
\[\textstyle
(\forall x\in [\frac{1}{k},1-\frac{1}{k}])(\forall N \in \Omega )(\forall \pi\in P([0,1]))\big[\|\pi\|\cdot N\approx 0\di \Delta_{1/N}\big(S_{\pi}(f, 0, x)\big)\approx f(x) \big]. \notag
\]
\end{thm}
Note that $\|\pi\|$ as in the consequent $\FTC_{\ns}$ is not only infinitesimal, but also `infinitesimal compared to $1/N$'. 
The associated effective version $\FTC_{\ef}(t)$ is:
\begin{thm}[$\FTC_{\ef}(t)$] 
For all $f:\R\di \R$ with modulus of uniform continuity $g$ on $[0,1]$, and any $k$, we have for all $x\in [\frac{1}{k},1-\frac{1}{k}]$ and all $l$ that:
\[
 \textstyle 
(\forall  N \geq t(g, k, l)(1) ,  \pi\in P([0,1]))\Big[ \big|\|\pi\|\cdot N \big|\leq \frac{1}{t( g, k, l)(2)} \di  \big|\Delta_{\frac{1}{N}}\big(S_{\pi}(f, 0, x)\big)- f(x)\big|\leq \frac{1}{l} \Big].
\]
\end{thm}
In contrast to $\CRI_{\ns}$ and $\ULC_{\ns}$ from the previous sections, the consequent of $\FTC_{\ns}$ does not have an obvious normal form.  
In this light, we prove the following theorem which significantly extends the scope of Corollary \ref{consresultcor}.     
\begin{thm}[$\P_{0}$]\label{necess}
For internal $\varphi$, the formula $(\forall N\in \Omega)(\forall^{\st}x)(\exists^{\st}y)\varphi(x,y, N)$ has an equivalent normal form \(as given by \eqref{overps}\).  
\end{thm}
\begin{proof}
Let $\varphi$ be internal and consider $(\forall N\in \Omega)(\forall^{\st}x)(\exists^{\st}y)\varphi(x,y, N)$, which yields
\[
(\forall^{\st}x)(\forall N) \big[(\forall^{\st} k^{0})(N\geq k)\di   (\exists^{\st}y)\varphi(x,y, N)\big].
\]
Pushing the standard quantifiers outside, we obtain
\[
(\forall^{\st}x)(\forall N)(\exists^{\st}k, y) \big[N\geq k\di  \varphi(x,y, N)\big],
\]
and applying idealisation \textsf{I}, a normal form emerges:
\[
(\forall^{\st}x)(\exists^{\st}w)(\forall N)(\exists k, y\in w) \big[N\geq k\di  \varphi(x,y, N)\big].
\]
Let $l^{0}$ be the maximum of all entries in $w$ pertaining to $k$, and let $v$ be the subsequence of $w$ with all entries pertaining to $y$.  With these definitions:  
\be\label{overps}
(\forall^{\st}x)(\exists^{\st}l^{0}, v)(\forall N\big[N\geq l\di (\exists y\in v) \varphi(x,y, N)\big], 
\ee
which expresses that underspill may be applied to formulas like in the theorem.  
\end{proof}
We are now ready to prove the following theorem.
\begin{thm}\label{floggen4}
From the proof of $\FTC_{\ns}$ in $\P_{0} $, a term $t$ can be extracted such that $\textup{\textsf{E-PRA}}^{\omega*} $ proves $\FTC_{{\ef}}(t)$.  
\end{thm}
\begin{proof}
A straightforward adaptation of the proof of \cite{aloneatlast3}*{Theorem 29} yields a proof of $\FTC_{\ns}$ in $\P_{0}$.     
In the consequent of $\FTC_{\ns}$, resolve both occurrences of `$\approx$', push the standard quantifiers outside to obtain a formula as in Theorem \ref{necess}.  
The latter theorem then implies that the consequent of $\FTC_{\ns}$ yields the following normal form:
\begin{align}\textstyle
(\forall^{\st}l)(\exists^{\st}M, K)\big(\forall x\in &\textstyle\big[\frac{1}{k},1-\frac{1}{k}\big]\big)(\forall N \geq M )(\forall \pi\in P([0,1]))\label{farouke}\\
&\textstyle\Big[\big|\|\pi\|\cdot N\big|\leq \frac{1}{K}  \di\big|\Delta_{\frac{1}{N}}\big(S_{\pi}(f, 0, x)\big)- f(x)\big|\leq\frac{1}{l} \Big], \notag
\end{align}
and the rest of the proof is similar to that of Theorem \ref{varou} or \ref{floggen}.  
\end{proof}
There are obvious corollaries to the theorem similar to Corollary \ref{cordejedi} and~\ref{somaar}.
In particular, we may obtain a constructive version and a Herbrandisation of $\FTC_{\ns}$.  This Herbrandisation is nothing more than $\FTC_{\ef}(t)$ with $(\forall k')$ in the antecedent replaced 
by $(\forall k'\leq t(g,k,l)(3))$ for a modified term $t$.  In particular, this Herbrandisation tells us that for $f$ continuous up to precision $1/ t(g,k,l)(3)$ witnessed by $g$, differentiation with quotient $\frac{1}{N}\leq \frac{1}{t(g,k,l)(1)}$ cancels out, up to precision $1/l$, Riemann integration for partitions finer than ${1}/({t( g, k, l)(2)\times N})$.        

\smallskip

Secondly, we consider  a version of $\FTC_{\ns}$ involving the functional $I(f,0, x)$ from Section \ref{frakkk}.  
Note that $\FTC'_{\ns}$ could also be made `hybrid' as follows: introduce the usual $\eps$-$\delta$-definition of uniform continuity with a standard modulus $g$ (in the antecedent), and use $[I(f, 0,x)](g(k))$ in the consequent.  
With such a modulus, $f$ is automatically nonstandard continuous, i.e.\ one could use techniques from both Constructive and Nonstandard Analysis.  
\begin{thm}[$\FTC'_{\ns}$]
For $f:\R\di \R$ and $k^{0}$, if $f$ is nonstandard uniformly continuous on $[0,1]$, we have $\big(\forall N\in \Omega, x\in \big[\frac{1}{k}, 1-\frac{1}{k}\big]\big)\big(\Delta_{\frac{1}{N}}(I(f,0, x))\approx f(x)\big)$.
\end{thm}
Similar to Theorem \ref{floggen4}, one obtains the following effective version.  
\begin{thm}[$\FTC_{\ef}'(t)$]
For $f:\R\di\R$ with modulus of uniform continuity $g$ on $[0,1]$, and any $k$, $(\forall l)(\forall  N\geq t(g , k, l), x\in [\frac{1}{k}, 1-\frac{1}{k}])\big[|\Delta_{\frac{1}{N}}(I(f,0, x))- f(x)|\leq \frac{1}{l}\big]$.
\end{thm}
The previous example shows that we can `mix and match' nonstandard definitions with effective definitions. 

\smallskip

Finally, the \emph{second part} of $\FTC$ as formulated in \cite{rudin}*{Theorem 6.21} can be treated similarly:  
the nonstandard version is as follows. 
\begin{thm}[$\FTC_{\ns2}$]
For all $f:\R\di \R$ and standard $k$, if $f$ is nonstandard differentiable on $[\frac1k, 1-\frac1k]$, then
$(\forall \eps\ne 0)\big[\eps\approx 0 \di I(\Delta_{\eps}(f(\cdot)),0, 1))\approx f(1)-f(0)]$.
\end{thm}
The following effective version of $\FTC_{\ns2}$ is obtained similar to Theorem \ref{floggen4}.  
\begin{thm}[$\FTC_{\ef2}(t)$]
For $f:\R\di\R$ with modulus of uniform continuity on $[0,1]$, and any $k$, we have for all $l$ that 
\be\label{kokoko}
(\forall \eps)\textstyle\big(\forall x \in \big[\frac{1}{k},1-\frac{1}{k}\big])[0< |\eps|\leq \frac{1}{t(g, l)}\di \big| \Delta_{\eps}(I(f, 0,x))- (f(1)-f(0))\big|\leq\frac{1}l\big].
\ee
\end{thm}
Note that our notion of nonstandard differentiability from Definition \ref{settleourdiff} gives rise to Bishop's definition \cite{bridge1}*{Definition 5.1, p.~44}.  This is both true for the nonstandard version (using $\Omega$\textsf{-CA}) and the effective version as in $\FTC_{\ef2}(t)$, i.e.\ a modulus of (uniform) differentiability as in \eqref{kokoko} naturally emerges.  
\subsection{Picard's theorem}\label{chu}
In this section, we treat the \emph{Picard existence theorem} ($\PICA$ for short; see e.g.\ \cite{simpson2}*{IV.8.4}).  
As suggested by its name, $\PICA$ states the \emph{existence} of a (unique) solution to a certain differential equation of the form $y'=f(x,y)$.  
We will formulate various nonstandard versions of $\PICA$ with proofs in $\P_{0}$, and obtain the associated effective versions.  We make use of the usual definitions of Lipschitz continuity and boundedness as follows.    
\be\label{gimmelip}
 (\forall x, y, z\in [-1,1])\big(|f(x, y)-f(x, z)|\leq |y-z| \wedge |f(x,y)|\leq 1 \big). 
\ee
First of all, we study the `usual' nonstandard version of $\PICA$. 
To this end, consider the sequence $\phi_{n}$ in \cite{simpson2}*{IV.8.4, Equation (16)}; we could also use $\phi$ from \cite{aloneatlast3}*{Equation (22)}.
Intuitively speaking, $\phi_{n}(x)$ is such that $\phi_{M}(x)\approx \int_{0}^{x}f(t, \phi_{M}(t))dt$ for nonstandard $M$, which gives rise to the following theorem.  
\begin{thm}[$\PICA_{\ns}$]
For $f:\R^{2}\di \R$ nonstandard uniformly continuous on $[-1, 1]^{2}$, and as in \eqref{gimmelip}, we have $(\forall K, L\in \Omega, x\in [0,1])(\phi_{K}(x)\approx \phi_{L}(x))$ and
\be\label{toguh}
(\forall x\in [-1,1])(\forall N, M\in \Omega)\big[\textstyle\frac{N}{M}\approx 0\di \Delta_{\frac1N}\phi_{M}(x)\approx f(x, \phi_{M}(x))\big].
\ee
\end{thm}
The associated effective version is as follows.
\begin{thm}[$\PICA_{\ef}(t)$]
For any $f:(\R\times \R)\di \R$ with modulus of uniform continuity $g$ on $[-1,1]\times [-1,1]$, and such that \eqref{gimmelip}, we have
\[
 (\forall k)(\forall x\in [-1,1], K,L\geq t(g)(k)(3))\big[ |\phi_{K}(x)- \phi_{L}(x)|\leq \textstyle\frac{1}{k}\big], 
\]
\[
 (\forall k)(\forall x\in [-1,1], N,M\geq t(g)(k)(1))\big[\textstyle\frac{N}{M}\leq \frac{1}{t( g)(k)(2)}\di |\Delta_{\frac1N}\phi_{M}(x)- f(x, \phi_{M}(x))|\leq \textstyle\frac{1}{k}\big].
 \]
\end{thm}
Note that $\PICA_{\ef}(t)$ tells us how fast the sequence $\phi_{n}$ converges to the unique solution of $y'=f(x,y)$, \emph{and} how fast the derivative of this solution converges to $f$.  
\begin{thm}\label{floggen11111113}
From the proof of $\PICA_{\ns}$ in $\P_{0}$, a term $t$ can be extracted such that $\textup{\textsf{E-PRA}}^{\omega*} $ proves $\PICA_{{\ef}}(t)$.  
\end{thm}
\begin{proof}
To prove $\PICA_{\ns}$, the proof of Peano's existence theorem  in \cite{aloneatlast3}*{Theorem 31} or \cite{loeb1}*{Theorem 14.1, p.\ 64} is readily adapted to the case at hand.  
Alternatively, to prove Picard's theorem in $\RCA_{0}$ in \cite{simpson2}*{IV.8.4}, Simpson constructs a sequence $\phi_{n}$ which converges to a solution of $y'=f(x, y)$.  
It is immediate from the details of the RM-proof that $\phi_{N}(x)\approx \phi_{M}(x)$ for any nonstandard $N,M$ and any $x$ in the relevant interval if $f$ is as in $\PICA_{\ns}$.     

\smallskip

The second part of the theorem is straightforward in light of the proofs in the previous sections.  
In particular, one easily obtains a normal form for the consequent of $\PICA_{\ns}$ using Theorem \ref{necess}.  
\end{proof}
Secondly, we consider a weaker nonstandard version of $\PICA$, which will yield less computational information, as expected.
\begin{thm}[$\PICA'_{\ns}$]
For $f:\R^{2}\di \R$, nonstandard uniformly continuous on $[-1, 1]^{2}$ and \eqref{gimmelip}, we have $(\exists \phi)(\forall x\in [-1,1])(\forall N\in \Omega)\big[\Delta_{\frac1N}\phi(x)\approx f(x, \phi(x))\big]$.  
\end{thm}
The associated effective version of $\PICA_{\ns}$ is as follows.
\begin{thm}[$\PICA_{\ef}'(t)$]
For any $f:(\R\times \R)\di \R$ with modulus of uniform continuity $g$ on $[-1,1]\times [-1,1]$ such that \eqref{gimmelip}, we have
\begin{align}
 (\forall k)(\exists \phi)(\forall x\in [-1,1], N\geq t(g)(k))\big[|\Delta_{\frac1N}\phi(x)- f(x, \phi(x))|\leq \textstyle\frac{1}{k}\big]
\end{align}
\end{thm}
In contrast to $\PICA_{\ef}(t)$, the effective version $\PICA_{\ef}'(t)$ only tells us the \emph{existence} of an approximate solution $\phi$ to $y'=f(x,y)$, while it \emph{computes} how precise the derivative $\phi'$ needs to be approximated by $\Delta_{\frac1N}\phi(x)$.  
We have the following theorem.  
\begin{thm}\label{floggen3}
From the proof of $\PICA_{\ns}'$ in $\P_{0}$, a term $t$ can be extracted such that $\textup{\textsf{E-PRA}}^{\omega*} $ proves $\PICA_{{\ef}}'(t)$.  
\end{thm}
\begin{proof}
By Theorem \ref{floggen11111113}, it is immediate that $\P_{0}$ proves $\PICA_{\ns}'$.  
 The rest of the proof is now straightforward in light of the previous proofs, if we can provide the right normal form for the consequent of $\PICA_{\ns}$.  
Now, the latter consequent implies  
\[\textstyle
( \exists \phi:\R\di \R)(\forall^{\st}k)(\forall N\in \Omega)(\forall x\in [-1,1])\big[|\Delta_{\frac1N}\phi(x)- f(x, \phi(x))|\leq\frac1k\big],
\]
and applying underspill immediately yields 
\[\textstyle
( \exists \phi:\R\di \R)(\forall^{\st}k)(\exists^{\st}K)(\forall N\geq K)(\forall x\in [-1,1])\big[|\Delta_{\frac1N}\phi(x)- f(x, \phi(x))|\leq\frac1k\big],
\]
and finally we have the following normal form:
\be\label{drawiu}\textstyle
(\forall^{\st}k)(\exists^{\st}K)( \exists \phi)(\forall N\geq K)(\forall x\in [-1,1])\big[|\Delta_{\frac1N}\phi(x)- f(x, \phi(x))|\leq\frac1k\big],
\ee
and $\PICA_{\ef}'(t)$ now follows in a straightforward manner.  
\end{proof}
Applying $\Omega\textsf{-CA}$ to $\PICA_{\ns}$, we note that $\phi$ from $\PICA'_{\ns}$ may be assumed to be standard.
Starting from this stronger version in Theorem \ref{floggen3}, the normal form \eqref{drawiu} would involve `$( \exists^{\st} \phi)$', and the associated 
effective version would provide a witnessing functional for $(\exists \phi)$ (involving a finite sequence of solutions).  Similarly, the term $t$ in \eqref{EST} only depends on $g$ \emph{and not on $f$}, as the former is standard in the normal form \eqref{second23} of $\CRI_{\ns}$, but no such assumption is made for the latter.  
The previous is just the observation that term extraction as in Corollary \ref{consresultcor} only provides information about \emph{standard} quantifiers, while ignoring the internal quantifiers.    
This observation leads us to the following remark.  

\begin{rem}[Standard quantifiers]\rm
If we wish to exclude objects from being inputs or outputs of the term $t$ in Corollary \ref{consresultcor}, we can just drop the `st' in the relevant quantifiers in the nonstandard version \emph{assuming this omission is at all possible\footnote{For instance, we can replace nonstandard continuity in $\CRI_{\ns}$ by \eqref{first5}, but we cannot drop the `\st' relating to $g$ in \eqref{first5}, as the standardness of $g$ is required to guarantee that $f$ in \eqref{first5} is still \emph{nonstandard} continuous.}}.  
In other words, the standard quantifiers quantify over the `computationally relevant' objects.       
A similar setup is defined by Berger in \cite{uhberger}, namely an extension of Heyting arithmetic which sports besides the usual quantifiers, new quantifiers for which the quantified object is not computationally used in the proof at hand.  
Apparently, this can improve the efficiency of the associated algorithms.  
A similar interpretation of the standard quantifiers is discussed in \cites{brie,benno2, SB}.  
\end{rem}
As expected, the proof of both nonstandard versions of $\PICA$ can be carried out in $\H$ and we could obtain a version of Corollaries \ref{cordejedi} and \ref{somaar}.  
As it turns out, Picard's theorem has a constructive proof, to be found in \cite{diedif}*{\S4}.  

\smallskip

Finally, we have not used the (even constructively true) fact that the solution to Picard's theorem is \emph{unique}.  The following addition could be made to either nonstandard version of $\PICA$.
\begin{align*}\textstyle
(\forall \psi)\big[(\forall N\in \Omega) (\forall x\in [-1, 1])[\Delta_{\frac{1}{N}}&\psi(x)\approx f(x, \psi(x))] \\
&\textstyle \di (\forall x\in [-1,1])(\forall M\in \Omega)(\psi(x)\approx \phi_{M}(x))     \big].
\end{align*}
As a result, the effective version of $\PICA$ would be extended with a `modulus of unicity':  
given a functional witnessing $h$ how fast $\Delta_{\frac{1}{n}}\psi(x)\di f(x, \psi(x))$ as $n\di \infty$, there is a term $s( \cdot, h)$ witnessing how fast $\phi_{m}(x)\di \psi(x)$ as $m\di \infty$.  

\subsection{The template $\CI$}\label{detail}
In this section, we formulate the template $\CI$ sketched in Section \ref{kintro} based on the above case studies.  
We emphasize that some aspects of $\CI$ are inherently vague.    
Recall the algorithms $\mathcal{A}$ and $\mathcal{B}$ introduced in Section \ref{P}.  
\begin{tempie}[$\CI$]~\rm
The starting point for $\CI$ is a mathematical theorem $T$ formulated in the language of $\textsf{E-PA}^{\omega*}$.  
\begin{enumerate}[(i)]
\item Replace in $T$ all definitions (convergence, continuity, et cetera) by their well-known counterparts from Nonstandard Analysis.  For the resulting theorem $T^{*}$, look up the proof (e.g.\ in 
\cites{loeb1,stroyan, nsawork2}) and formulate it inside $\P_{0}$ (or $\P$) and $\H$ if possible.   
Otherwise add to the conditions of $T^{*}$ (external) axioms from $\IST$ to guarantee this provability.        \label{famouslastname}
\item Bring all nonstandard definitions in $T^{*}$ into the \emph{normal form} $(\forall^{\st}x)(\exists^{\st}y)\varphi(x, y)$.  This operation usually requires $\textsf{I}$ for $\P_{0}$, and usually requires the new axioms from Definition \ref{flah} for $\H$.    
If necessary, drop `st' in leading existential quantifiers of positively occurring formulas.  \label{famousfirsdtname}
\item Starting with the most deeply nested implication, bring
\be\label{duggg}
(\forall^{\st}x_{0})(\exists^{\st}y_{0})\varphi_{0}(x_{0}, y_{0})\di (\forall^{\st}x_{1})(\exists^{\st}y_{1})\varphi_{1}(x_{1}, y_{1}),
\ee
into the normal form $(\forall^{\st}x)(\exists^{\st}y)\varphi(x, y)$.  
In particular, apply $\HAC_{\INT}$ to the antecedent of \eqref{duggg} and bring to the front all standard quantifiers.  The last step is trivial in $\P_{0}$ and requires the axioms from Definition \ref{flah} for $\H$.
\item Apply Corollary \ref{consresultcor} (if applicable Theorem \ref{consresult2}) to the proof of (the normal form of) $T^{*}$.  The algorithm $\mathcal{A}$ (if applicable $\mathcal{B}$) provides a term $t$.    
\item Output the term(s) $t$ and the proof(s) of the effective version.  
\end{enumerate}    
\end{tempie}
The theorems in the above case studies all had proofs inside $\H$ or $\P_{0}$, i.e.\ the final sentence in step \eqref{famouslastname} does not apply.  
In Section \ref{RMSTUD}, we shall study theorems for which we \emph{do} have to add external axioms of $\IST$ to the conditions of the theorem.  
Similarly, the final sentence of step \eqref{famousfirsdtname} does not apply to the above first three case studies, but was relevant to Picard's theorem. 

\smallskip

Finally, there is a tradition of Nonstandard Analysis in RM and related topics (see e.g.\ \cites{pimpson,tahaar, tanaka1, tanaka2, horihata1, yo1, yokoyama2, yokoyama3}), which provides a source of proofs in (pure) Nonstandard Analysis for $\CI$.  To automate the process of applying $\CI$, we have implemented the term extraction algorithm from Corollary \ref{consresult2} in Agda (see \cite{EXCESS}).  

\section{Main results II: Reverse Mathematics 241}\label{RMSTUD}
In this section, we prove the second batch of results involving representative theorems from the four strongest Big Five systems of RM. 
Thus, we establish the vast scope of the template $\CI$ from Section \ref{detail}, as discussed in the introduction.  

\smallskip

In particular, we shall prove certain equivalences between nonstandard theorems and fragments of Nelson's axioms \emph{Standard Part} and \emph{Transfer} (see Section~\ref{P}).  
By running the (usually very simple) proofs of these nonstandard equivalences through $\CI$, we obtain explicit\footnote{An implication $(\exists \Phi)A(\Phi)\di (\exists \Psi)B(\Psi)$ is \emph{explicit} if there is a term $t$ in the language such that additionally $(\forall \Phi)[A(\Phi)\di B(t(\Phi))]$, i.e.\ $\Psi$ can be explicitly defined in terms of $\Phi$.\label{kurf}} equivalences for the Big Five beyond the base theory.  Hence, it seems the RM of Nonstandard Analysis is `simpler' than classical RM, \emph{and} provides more (relative) computational information.  Furthermore, the corresponding Herbrandisations provide `pointwise' versions of the aforementioned explicit equivalences, where the use of the input is more carefully tracked.    
Thus, as suggested by the section title, nonstandard equivalences provide \emph{two} kinds of RM-equivalences, a `global' one directly inspired by RM, and a `pointwise' one which contains more computational information, namely the Herbrandisation.  
  
\smallskip

Finally, our study of compactness in Section \ref{compaq} is particularly interesting as nonstandard compactness gives rise to different normal forms which 
yield quite different effective results:  respectively, the (re)discovery of \emph{totally boundedness} (the preferred notion of compactness in constructive and computable analysis), and the (re)discovery of explicit$^{\ref{kurf}}$ RM-equivalences at the level of $\WKL_{0}$.  Similarly, our (quite simple) notion of \emph{nonstandard-separating set} will give rise to the `constructive' notion of separating set, for the Stone-Weierstra\ss~theorem in Section \ref{stoner}.

\subsection{Monotone convergence theorem}\label{X}
In this section, we study the \emph{monotone convergence theorem} $\MCT$, i.e.\ the statement that \emph{every bounded increasing sequence of reals is convergent}, which is equivalent to arithmetical comprehension $\ACA_{0}$ by \cite{simpson2}*{III.2.2}.  In particular, this study provides a first example that the template $\CI$ applies to the third Big Five category, while more general results may be found in Section \ref{algea}.  
With some effort, the reader may verify that our results do not require full primitive recursion, but that the exponential function suffices, as studied in \cite{kohlenbach2}*{\S3}, \cite{kohlenbach1}, and \cite{bennosam}.   
\subsubsection{From nonstandard $\MCT$ to explicit $\MCT$}
In this section, we prove an equivalence between a nonstandard version of $\MCT$ and a fragment of \emph{Transfer}.
From this nonstandard equivalence, we extract an explicit$^{\ref{kurf}}$ RM equivalence involving $\MCT$ and arithmetical comprehension.  

\smallskip

Firstly, the nonstandard version of $\MCT$ (involving nonstandard convergence) is:
\be\label{MCTSTAR}\tag{\MCT$_{\textsf{ns}}$}
(\forall^{\st} c_{(\cdot)}^{0\di 1})\big[(\forall n^{0})(c_{n}\leq c_{n+1}\leq 1)\di (\forall N,M\in \Omega)[c_{M}\approx c_{N}]    \big].
\ee
We could have introduced a `st' in the antecedent of \ref{MCTSTAR}, but Theorem \ref{sef} would not change.  
Note that by applying $\Omega$\textsf{-CA} from Theorem \ref{drifh} to the consequent, we immediately obtain the (standard) limit of $c_{n}$.  
The effective version \MCT$_{\textsf{ef}}(t)$ is:
\be\label{MCTSTAR22}\textstyle
(\forall c_{(\cdot)}^{0\di 1},k^{0})\big[(\forall n^{0})(c_{n}\leq c_{n+1}\leq 1)\di (\forall N,M\geq t(c_{(\cdot)})(k))[|c_{M}- c_{N}|\leq \frac{1}{k} ]   \big].
\ee
We require two equivalent (\cite{kohlenbach2}*{Prop.\ 3.9}) versions of arithmetical comprehension: 
\be\label{mu}\tag{$\mu^{2}$}
(\exists \mu^{2})\big[(\forall f^{1})( (\exists n)f(n)=0 \di f(\mu(f))=0)    \big],
\ee
\be\label{mukio}\tag{$\exists^{2}$}
(\exists \varphi^{2})\big[(\forall f^{1})( (\exists n)f(n)=0 \asa \varphi(f)=0    \big],
\ee
and also the restriction of Nelson's axiom \emph{Transfer} to $\Pi_{1}^{0}$-formulas as follows:
\be\tag{$\paai$}
(\forall^{\st}f^{1})\big[(\forall^{\st}n^{0})f(n)\ne0\di (\forall m)f(m)\ne0\big].
\ee
Denote by $\textsf{MU}(\mu)$ the formula in square brackets in \eqref{mu}; the associated functional $\mu^{2}$ is also called \emph{Feferman's search operator}.  
We have the following theorem which establishes the explicit equivalence between $(\mu^{2})$ and uniform $\MCT$.  
\begin{thm}\label{sef}
From the proof of $\MCT_{\ns}\asa \paai$ in $\P_{0} $, two terms $s, u$ can be extracted such that $\textup{\textsf{E-PRA}}^{\omega*}$ proves:
\be\label{frood}
(\forall \mu^{2})\big[\textsf{\MU}(\mu)\di \MCT_{\ef}(s(\mu)) \big] \wedge (\forall t^{1\di 1})\big[ \MCT_{\ef}(t)\di  \MU(u(t))  \big].
\ee
\end{thm}
\begin{proof}
To establish $\MCT_{\ns}\di \paai$, fix standard $f^{1}$ such that $(\forall^{\st}n)f(n)=0$ and define the \emph{standard} sequence  $c_{(\cdot)}$ of reals by $c_{k}=0$ if $(\forall i\leq k)f(i)=0$ and $c_{k}=\sum_{i=1}^{k}\frac{1}{2^{i}}$ otherwise.  
Clearly, $c_{(\cdot)}$ is increasing and hence nonstandard convergent as in $\MCT_{\ns}$.  However, if $(\exists m_{0})f(m_{0}+1)=0$ and $m_{0}$ is the least such number, we have $0=c_{m_{0}}\not\approx c_{m_{0}+1}\approx 1$.  
Thus, $\paai$ follows and we now prove the other direction.  Assume $\paai$ and define $\varphi(f, M)$ as $0$ if $(\exists  n\leq M)f(n)=0$ and $1$ otherwise.  By assumption, we have 
\[
(\forall^{\st}f^{1})(\forall N, M\in \Omega)\big[ \varphi(f, M)=\varphi(f, N)=0 \asa (\exists^{\st}n)f(n)=0  ].
\]
Applying $\Omega\textsf{-CA}$ to $\varphi(f, N)$ yields $(\exists^{2})^{\st}$.  The latter is the functional version of $\ACA_{0}$ (relative to \st) and the usual proof of $\MCT$ goes through in $\ACA_{0}$ (see \cite{simpson2}*{I.9.1}).
Hence, we obtain $\MCT^{\st}$ and $\MCT_{\ns}$ follows by applying $\paai$ to the innermost universal formula in the consequent of $\MCT^{\st}$ expressing that $c_{n}$ converges;  indeed, the following formula follows from $\MCT^{\st}$:  
\[\textstyle
(\forall^{\st}k )(\exists^{\st}m)(\forall^{\st} N,M\geq m)[|c_{M}- c_{N}|\leq \frac{1}{k}\big],
\]
and apply $\paai$ to $(\forall^{\st}N, M\geq m)(\cdots)$.  We now prove the theorem for the implication $\paai \di \MCT_{\ns}$ and leave the other one to the reader.    
The former implication is readily converted to:
\begin{align}
(\forall^{\st}f^{1})&(\exists^{\st}n)\big[(\exists m)f(m)=0\di (\exists i\leq n)f(i)=0]\label{noniesimpel}\\
&\di (\forall^{\st} c_{(\cdot)}^{0\di 1}, k)(\exists^{\st}m)\big[(\forall n^{0})(c_{n}\leq c_{n+1}\leq 1)\di (\forall N,M\geq m)[|c_{M}- c_{N}|\textstyle\leq \frac1k]    \big].\notag
\end{align}
Let $A$ (resp.\ $B$) be the first (resp.\ second) formula in square brackets in \eqref{noniesimpel}.  Applying $\HAC_{\INT}$ to the antecedent of the latter (and performing the usual step involving the maximum), we obtain $(\exists^{\st} \mu^{2})(\forall^{\st}f^{1})A(f, \mu(f))$.  
Hence, \eqref{noniesimpel} becomes
\[
(\forall^{\st} c_{(\cdot)}^{0\di 1}, k^{0}, \mu^{2})(\exists^{\st}m, f)[A(f, \mu(f))\di B(c_{(\cdot)}, k, m)], 
\]
and Corollary \ref{consresultcor} yields a term $t$ such that $\textsf{E-PRA}^{\omega*}$ proves
\be\label{hermct}
(\forall  c_{(\cdot)}^{0\di 1}, k^{0}, \mu^{2})(\exists m, f\in t(\mu, c_{(\cdot)}, k))[A(f, \mu(f))\di B(c_{(\cdot)}, k, m)], 
\ee
and define $s(\mu, c_{(\cdot)}, k)$ to be the maximum of all entries for $m$ in $t(\mu, c_{(\cdot)}, k)$.
We immediately obtain, using classical logic, that
\[
(\forall   \mu^{2})[(\forall f^{1})A(f, \mu(f))\di (\forall c_{(\cdot)}^{0\di 1}, k^{0})B(c_{(\cdot)}, k, s(\mu, c_{(\cdot)}, k))], 
\]
which is exactly as required by the theorem.  
\end{proof}
With slight effort, the proof of $\MCT_{\ns}\asa\paai$ goes through in $\H$, and hence the theorem is constructive.  
We again stress that the results in \eqref{frood} are not (necessarily) surprising in and of themselves.  
\emph{What is surprising} is that we can `algorithmically' derive these effective results from the \emph{quite simple} proof of $\paai\asa \MCT_{\ns}$, in which \emph{no efforts towards effective results are made}.  

\subsubsection{Herbrandisation}\label{dafkens}
By the above results, $(\mu^{2})$ (resp.\ $(\exists \Phi^{1\di 1})\MCT_{\ef}(\Phi)$) clearly is the Herbrandisation of $\paai$ (resp.\ $\MCT_{\ns}$).  
In this section, we study the Herbrandisation of $\paai\asa \MCT_{\ns}$, which is different from the Herbrandisations of the components.  
The Herbrandisation $\MTE_{\her}(s,t)$ of the aforementioned equivalence is the conjunction of \eqref{hermct} and the following:
\be\label{zion}
(\forall  f^{1})(\exists m^{0},c_{(\cdot)}^{0\di 1}, k^{0}, g^{1}\in  s(f)))[B(c_{(\cdot)}, k, g(k))\di A(f, m)) ], 
\ee
where $A, B$ are as in \eqref{hermct}.  Intuitively speaking, \eqref{frood} effectively converts Feferman's search operator into a uniform version of $\MCT$ (and vice versa), while the Herbrandisation $\MCT_{\her}(s,t)$ effectively converts \emph{a finite sequence of instances} of Feferman's search operator into \emph{one instance} of $\MCT$ (and vice versa).   
Thus, the Herbrandisation of $\MCT_{\ns}$ is the \emph{pointwise} version of the effective result \eqref{frood}.  
\begin{cor}\label{herken}
Let $s,t$ be terms in the internal language. 
A proof inside $\textup{\textsf{E-PRA}}^{\omega*}$ of $\MTE_{\her}(s,t)$, can be converted into a proof inside $\P_{0}$ of $\paai\asa \MCT_{\ns}$. 
\end{cor}
\begin{proof}
For the first conjunct of $\MTE_{\her}(s,t)$, i.e.\ \eqref{hermct}, the term $t$ is standard in $\P_{0}$, yielding that the latter proves
\[
(\forall^{\st}  c_{(\cdot)}^{0\di 1}, k^{0}, \mu^{2})(\exists^{\st} m, f)[A(f, \mu(f))\di B(c_{(\cdot)}, k, m)].  
\]
Bringing all quantifiers inside again, we obtain
\be\label{tokkieninja}
(\exists^{\st}\mu^{2})(\forall^{\st}f)A(f, \mu(f))\di (\forall^{\st}  c_{(\cdot)}^{0\di 1}, k^{0})(\exists^{\st}m^{0})B(c_{(\cdot)}, k, m).  
\ee
The consequent of \eqref{tokkieninja} clearly implies $\MCT_{\ns}$, while the antecedent is \emph{equivalent} to $(\forall^{\st}f)(\exists^{\st}n)A(f, n)$ thanks to $\HAC_{\INT}$.  
Hence, the antecedent of \eqref{tokkieninja} is clearly $\paai$.  One treats \eqref{zion} in exactly the same way.    
\end{proof}
By the previous corollary, the equivalence $\paai\asa \MCT_{\ns}$ contains the same computational information (up to algorithmic extraction) as its Herbrandisation $\MCT_{\her}(s,t)$.  
The latter tells us which instances of the search operator $(\mu^{2})$ are needed to witness the convergence of a given sequence $c_{(\cdot)}$ up to a given precision $1/k$, and vice versa.    
Hence, we can establish a bridge between soft and hard analysis \emph{even if the theorem at hand is non-constructive} thanks to the nonstandard version of the RM-equivalence.  

\smallskip

In conclusion, it turns out we may extract \emph{two} kind of RM-equivalences from $\MCT_{\ns}$: the `global' equivalence \eqref{frood} and the `pointwise' Herbrandisation, and the latter contains much more computational information.  In other words, the RM of Nonstandard Analysis gives rise to at least two kinds of higher-order RM, as suggested by the title of Section \ref{RMSTUD}.    
\subsubsection{Some corollaries}
In this section, we discuss some corollaries to Theorem~\ref{frood}.  

\smallskip
Firstly, since $\MCT$ states the \emph{existence} of the limit $c$ of $c_{(\cdot)}$, we shall also obtain such a version from $\MCT_{\ns}$ too.  
Let $\MCT'_{\ef}(t)$ be the following variation of \eqref{MCTSTAR22}:
\[\textstyle
(\forall c_{(\cdot)}^{0\di 1},k^{0})\big[(\forall n^{0})(c_{n}\leq c_{n+1}\leq 1)\di (\forall N\geq t(c_{(\cdot)})(1)(k))(|c_{N}- t(c_{(\cdot)})(2)|\leq \frac{1}{k} )   \big].
\]
\begin{cor}\label{loppp}
From the proof of $\MCT_{\ns}\asa \paai$ in $\P_{0} $, two terms $s, u$ can be extracted such that $\textup{\textsf{E-PRA}}^{\omega*}$ proves:
\be\label{frood2}
(\forall \mu^{2})\big[\textsf{\MU}(\mu)\di \MCT_{\ef}'(s(\mu)) \big] \wedge (\forall t^{1\di 1})\big[ \MCT_{\ef}'(t)\di  \MU(u(t))  \big].
\ee
\end{cor}
\begin{proof}
Using $\paai$ and $\Omega\textsf{-CA}$, the following variant of \eqref{MCTSTAR22} can be obtained:
\[\textstyle
(\forall^{\st} c_{(\cdot)}^{0\di 1})(\exists^{\st}c^{1})(\forall^{\st} k)(\exists^{\st}m)\big[(\forall n^{0})(c_{n}\leq c_{n+1}\leq 1)\di (\forall N\geq m)(|c_{N}- c|\textstyle\leq \frac1k)\big] .
\]
Now apply $\HAC_{\INT}$ to remove the $(\exists^{\st}m)$ quantifier, and obtain via $\paai$ that:
\[\textstyle
(\forall^{\st} c_{(\cdot)}^{0\di 1})(\exists^{\st}c^{1}, h^{1})(\forall k)\big[(\forall n^{0})(c_{n}\leq c_{n+1}\leq 1)\di (\forall N\geq h(c_{(\cdot)}, k))(|c_{N}- c|\textstyle\leq \frac1k)\big].
\]
The rest of the proof is now similar to that of the theorem.  Note that Corollary~\ref{consresultcor} only provides a finite list of witnesses to $(\exists c)$ and $(\mu^{2})$ has to be used to select the right one.  
On the other hand, out of the finite list of witnesses to $(\exists h)$, we just take the (pointwise) maximum of all entries, which suffices in light of the `monotone' behaviour of the associated variable.    
\end{proof}
Secondly, we now discuss an alternative proof using \emph{standard extensionality} of the previous corollary.  Recall from Remark \ref{equ} that  the axiom of standard extensionality $\eqref{EXT}^{\st}$ cannot be conservatively added to the systems $\P$ and $\H$.   

\smallskip

The following corollary is important, as it shows that we \emph{can} use the axiom standard extensionality for proving normal forms, as this axiom is translated by $\CI$ to (almost) a triviality.
Thus, let $\MCT_{\ef}''(t)$ be $\MCT_{\ef}'(t)$ with the extra assumption that $t$ is \emph{standard extensional}, i.e.\ we additionally have:
\be\label{dergggg}
(\forall^{\st}c_{(\cdot)}^{0\di 1}, s_{(\cdot)}^{0\di 1})\big[ c_{(\cdot)}\approx_{0\di 1}s_{(\cdot)}\di t(c_{(\cdot)})\approx_{0\times 1} t(s_{(\cdot)}) \big].
\ee
Note that $\paai$ implies standard extensionality \eqref{dergggg} for standard $t$ from the axiom of extensionality \eqref{EXT}.  
\begin{cor}\label{dergggg2}
From the proof of $(\exists^{\st}\varphi)\MCT_{\ef}''(\varphi)\di \paai$ in $\P_{0} $, a term $ u$ can be extracted such that $\textup{\textsf{E-PRA}}^{\omega*}$ proves:
\be\label{frood3333}
 (\forall \varphi )\big[ \MCT_{\ef}'(\varphi)\di  \MU(u(\varphi,\Xi))  \big],
\ee
where $\Xi$ is an extensionality functional for $\varphi$.  
\end{cor}
\begin{proof}
First of all, assume $(\exists^{\st}\varphi)\MCT_{\ef}''(\varphi)$ and suppose $\paai$ is false;  let $f$ and $c_{(\cdot)}$ be as in the proof of Theorem \ref{sef} and note that $c_{(\cdot)}\approx_{0\di 1}0^{0\di 1}$, 
while $0=\varphi(0^{0\di 1})(2)\not\approx \varphi(c_{(\cdot)})(2)=1$.  This contradicts standard extensionality for $\varphi$ and we obtain $\paai$.    

\smallskip

Secondly, we bring \eqref{frood3333} in normal form.  In particular, resolving `$\approx_{\tau}$' in its various incarnations in \eqref{dergggg} yields for all standard $c_{(\cdot)}^{0\di 1}, s_{(\cdot)}^{0\di 1}, k^{0}$
\[\textstyle
[(\exists^{\st}N^{0},M^{0})\big[ |c_{N}- s_{N}|\leq \frac{1}{M}]\di t(c_{(\cdot)})(1)= t(s_{(\cdot)})(1)\wedge \overline{t(c_{(\cdot)})(2)}k= \overline{t(s_{(\cdot)})(2)}(k) \big].
\]
This formula now easily yields a normal for $(\exists^{\st}\varphi)\MCT_{\ef}''(\varphi)\di \paai$ in light of the proof of the theorem.  Applying $\CI$ now finishes the proof.   
\end{proof}
The previous proof is not much shorter than that of Corollary \ref{loppp}, but it is conceptually interesting that we may use standard extensionality in obtaining normal forms, although we cannot add it as an axiom to our base theory, as discussed in Remark \ref{equ}.
Furthermore, Kohlenbach studies theorems regarding discontinuous functions of type $\R\di \R$ and related domains in \cite{kohlenbach2}*{\S3}.  These theorems naturally involve functionals of type $1\di 1$ and are thus particularly amenable to the above `standard extensionality' trick.  

\smallskip

It should not come as a surprise that the results obtained in this section for $\MCT$ can be generalised to other theorems from RM equivalent to $\ACA_{0}$.   
For instance, in light of the proof of \cite{simpson2}*{III.2.8}, it is straightforward to obtain Theorem \ref{sef} and its corollaries for the \emph{Ascoli-Arzela theorem}.  
For theorems not dealing with analysis, we shall even obtain a template in Section \ref{algea}.  
Moreover, in Sections \ref{compaq} and \ref{dinkitoes} (resp.\ Section \ref{suskeenwiske}), we shall obtain explicit equivalences for the \emph{second} (resp.\ fourth and fifth) Big Five system $\WKL_{0}$ (resp.\ $\ATR_{0}$ and $\FIVE$).  

\subsubsection{Connection to Kohlenbach's proof mining}\label{swisch}  
We discuss the connection between our results and Kohlenbach's \emph{proof mining} program (\cite{kohlenbach3}).  

\smallskip

First of all, we consider the following example.  
\begin{exa}[The scope of the proof mining of classical mathematics]\label{pop}\rm
We consider a common `counterexample' from the proof mining of classical mathematics from \cite{kohlenbach3}*{\S2.2}.  To this end, let $T$ be Kleene's well-known primitive recursive predicate as in e.g.\ \cite{zweer}*{Theorem 3.3}, where $T(x, y, z)$ intuitively expresses that the Turing machine with index $x$ and input $y$ halts with output $z$.  
Then the following formula has a trivial proof using the law of excluded middle:
\be\label{klieken}
(\forall x^{0})(\exists y^{0})(\forall z^{0}) \big( T(x, x, y) \vee \neg T(x, x, z) \big).  
\ee
As discussed in \cite{kohlenbach3}*{\S2.2}, an upper bound for $y$ in \eqref{klieken} gives rise to a solution to the (non-computable) Halting problem.  In other words, \eqref{klieken} is a simple example of a \emph{classically provable} formula for which proof mining \emph{cannot provide computable information} about the existential quantifier.  

\end{exa}
Secondly, we show that \eqref{klieken} can be treated in our nonstandard framework.  We proceed as follows:  observe that neither \eqref{klieken} nor \eqref{klieken}$^{\st}$ is a normal form, i.e.\ we cannot apply term extraction to the fact that $\P$ proves the former two formulas.  
Now, one \emph{can} prove the following formula:
\be\label{klieken2}
(\forall^{\st} x^{0})(\exists^{\st} y^{0})(\forall z^{0}) \big( T(x, x, y) \vee \neg T(x, x, z) \big),   
\ee
inside $\P_{0}+\paai$, but applying term extraction (as in the previous theorems in this section) yields a bound on $y$ \emph{in terms of the Turing jump functional}.  
Thus, no contradiction with Corollary \ref{consresultcor} ensues.  Furthermore, \eqref{klieken2} cannot\footnote{If $\P$ proves \eqref{klieken2}, then apply Corollary \ref{consresultcor} and obtain a term which provides an upper bound to the existential quantifier in \eqref{klieken}.  This term solves the Halting problem, a contradiction.} be proved in $\P$, and allowing (standard) oracles in \eqref{klieken2}, the latter would even imply $\paai$.    

\smallskip

Thirdly, in light of Example \ref{pop}, the usual `$\eps$-$\delta$' definition of convergence (with its two quantifier alternations) cannot be studied directly in the proof mining of classical mathematics.   
To this end, the following (classically equivalent) definition of convergence, called `metastable' by Tao (\cite{taote}*{p.\ 79}), is studied:
\be\textstyle\label{herforem}
(\forall k^{0}, g^{1})(\exists N^{0})\big[(\forall i^{0}, j^{0}\in [N, N+g(N)])(|x_{i}-x_{j}|<\frac1k)\big].
\ee
We now show that a certain nonstandard version of \eqref{herforem} has little computational content, similar to \eqref{klieken}.  
To this end, let $A(k , g, N, x_{(\cdot)})$ be the formula in square brackets in \eqref{herforem} and define $\textsf{META}_{\ns}$ as the statement 
that for all standard sequences $x_{(\cdot)}$ we have $(\forall^{\st} k^{0}, g^{1})(\exists^{\st}N^{0})A(k , g, N, x_{(\cdot)})\di (\exists^{\st}K^{0})(\forall n^{0})(|x_{n}|\leq K)$. 
We have the following theorem, which sprouts from the section on metastability in Diener's \emph{Habilitationsschrift} (\cite{hadie}).
\begin{thm}
From the proof $\P\vdash \META_{\ns}\asa \paai$, a term $t$ can be extracted such that for any $\Phi^{3}, h^{2}$, if for all sequences $x_{(\cdot)}$ we have
\[
(\forall k^{0}, g^{1})(\exists N^{0}\leq h(k, g))A(k , g, N, x_{(\cdot)})\di (\forall n^{0})(|x_{n}|\leq \Phi(h)),
\]
then $\MU(t(\Phi))$, i.e.\ $t$ computes the Turing jump from any functional $\Phi$ providing upper bounds for metastable sequences.  
\end{thm}
\begin{proof} 
The implication $\paai\di \META_{\ns}$ is straightforward in light of the equivalence between the usual definition of convergence and \eqref{herforem} relative to `\st'; $\paai$ is then 
used to make the sequence bounded everywhere.  For the reverse implication, assume $\META_{\ns}$ and suppose $\paai$ is false; let $f^{1}$ be standard such that $(\exists n^{0})(f(n)\ne0)\wedge (\forall^{\st}m^{0})(f(m)=0)$ and define $x_{n}:=\sum_{i=0}^{n}if(n)$.  This sequence is standard and metastable as in the antecedent of $\META_{\ns}$.  By the latter, $x_{(\cdot)}$ is bounded by a standard number, but this yields a contradiction as $x_{n_{0}}$ is nonstandard if $f(n_{0})\ne0$.  
\end{proof}
By the previous theorem, even obtaining an upper bound (let alone a rate of convergence) from a metastable sequence is not possible in a computational way.   

\smallskip

Fourth, part of the point of studying $\META_{\ns}$ and \eqref{klieken2} as above is to show that Nonstandard Analysis cannot `magically' solve the problems in the proof mining of classical mathematics.  
However, by grace of its extended language, the limitations of the proof mining of classical mathematics imposed by \eqref{klieken}, namely the restriction to less than two quantifier alternations, \emph{do not} apply to our nonstandard framework.  In other words, Nonstandard Analysis \emph{avoids} these problems by `moving the goalpost' (adopting a richer language).   
Of course, the template $\CI$ is limited to normal forms, but all theorems of `pure' Nonstandard Analysis fall into this category, and such theorems may involve arbitrary many quantifier alternations between \emph{internal quantifiers}.  
Furthermore, as shown in Section~\ref{suskeenwiske}, the template $\CI$ applies equally well to theorems of $\FIVE$, i.e.\ no special modification is necessary to 
study relatively strong RM-systems.  

\smallskip

Finally, a number of people have pressed the author for a direct comparison between Kohlenbach's approach to proof mining and the template $\CI$.  
We conjecture that the `one-size-fits-all' provided by the template $\CI$ is very general, and will produce terms of \emph{acceptable}\footnote{One heuristic in proof ming is that \emph{short proofs yield terms of low complexity}, as each step in the proof determines a sub-term of the eventually extracted term.  Now, Nonstandard Analysis is known for its short proofs, while the term $t$ from Corollary \ref{consresultcor} only depends on the use of \emph{external} axioms.  In particular, internal axioms like the usual induction axiom of $\textsf{E-HA}^{\omega*}$ or $\textsf{E-PA}^{\omega*}$, do not influence the complexity of the term $t$ (see \cite{brie}*{p.\ 1981, item 2}).} complexity while requiring little expertise beyond knowledge of Nonstandard Analysis.  By contrast, Kohlenbach's approach is more tailored to certain classes of proofs, and therefore produces terms of \emph{much} better complexity; as a downside, considerable expertise is required to wield the latter.  In conclusion, the template $\CI$ and Kohlenbach's proof mining are \emph{complementary}, in that the former provides a `first approximation' which can be performed by non-experts, while the latter fleshes out the exact (or even optimal) complexity, but requires considerably more expertise to pull off.

\subsection{Compactness}\label{compaq}
In this section, we study various theorems regarding compactness.  We first introduce nonstandard compactness in Section \ref{cinq} and discuss its various normal forms.  
We then study a number of theorems regarding nonstandard compactness in Sections \ref{cont1} to \ref{lapier}, from which we obtain effective versions.  
We obtain explicit equivalences at the level of $\WKL_{0}$ and $\ACA_{0}$.
In other words, we show that the template $\CI$ applies to the second and third Big Five category.  
In the interest of space, Herbrandisations are only briefly discussed in Section \ref{snarkie2}.
\subsubsection{Nonstandard compactness and normal forms} \label{cinq}
In this section, we introduce nonstandard compactness and discuss its various normal forms.      

\smallskip

First of all, the nonstandard characterisation of compactness is known as \emph{Robinson's theorem} (see \cite{loeb1}*{Theorem 2.2, p.\ 120} or \cite{robinson1}*{Theorem 4.1.13, p.\ 93}).  
In particular, a set $X$ is \emph{nonstandard compact} if $(\forall x\in X)(\exists^{\st}y\in X)(x\approx y)$, recalling Definition \ref{keepintireal}.  
As it turns out, there are various interesting but \emph{different} normal forms which can be derived from nonstandard compactness, as sketched now.  

\smallskip

Secondly, as an example of an \emph{equivalent} normal form, consider the nonstandard compactness of Cantor space as follows:
\be\label{STP}\tag{\textsf{STP}}
(\forall \alpha^{1}\leq_{1}1)(\exists^{\st} \beta^{1}\leq_{1}1)(\alpha\approx_{1} \beta),
\ee
which is equivalent to the following formula by \cite{samGH}*{Theorem 3.2}:
\begin{align}\label{fanns}
(\forall T^{1}\leq_{1}1)\big[(\forall^{\st}n)(\exists \beta^{0})&(|\beta|=n \wedge \beta\in T ) \di (\exists^{\st}\alpha^{1}\leq_{1}1)(\forall^{\st}n^{0})(\overline{\alpha}n\in T)   \big]
\end{align}
where `$T\leq_{1}1$' denotes that $T$ is a binary tree.  
Clearly, \eqref{fanns} is a nonstandard version of weak K\"onig's lemma, and the latter is a compactness principle as discussed in \cite{simpson2}*{IV}. 
Furthermore, \eqref{fanns} is equivalent to the following normal form:
\begin{align}\label{frukkklk}
(\forall^{\st}g^{2})(\exists^{\st}w^{1^{*}})(\forall T^{1}\leq_{1}1)(\exists ( \alpha^{1}\leq_{1}1,  &~k^{0}) \in w)\big[(\overline{\alpha}g(\alpha)\not\in T)\\
&\di(\forall \beta\leq_{1}1)(\exists i\leq k)(\overline{\beta}i\not\in T) \big] \notag
\end{align}
as shown in \cite{samGH}*{\S3}.   Moreover, $\STP$ is equivalent to the normal form (see \cite{dagsam, dagsamII})
\be\label{hro}
(\forall^{\st} G^{2})(\exists^{\st}w^{1^{*}})(\forall f\leq_{1}1)(\exists i<|w|)\big(  f\in \big[\overline{w(i)}G(w(i))\big]  \big), 
\ee    
which expresses that the uncountable cover $\cup_{f\in 2^{\N}}[\overline{f}G(f)]$ of Cantor space has a finite sub-cover $\cup_{i\leq k}[\overline{f_{i}}G(f_{i})]$ provided by the finite sequence $w=\langle f_{0}, \dots, f_{k}\rangle$.
Similarly, the nonstandard compactness of the unit interval $(\forall x\in [0,1])(\exists^{\st}y\in [0,1])(x\approx y)$ 
is \emph{equivalent} to a nonstandard version of Heine-Borel compactness, namely \eqref{HB}.  As studied in Section \ref{dinki}, the latter normal form give rise to the explicit versions of 
certain equivalences from the RM of $\WKL_{0}$.  

\smallskip

Thirdly, we shall also obtain an interesting normal form in Theorem \ref{noco} which is a \emph{weakening} of nonstandard compactness, as follows: 
\be\label{marhi}\textstyle
(\exists^{\st}h)(\forall^{\st} k)(\forall  x\in X)(\exists i< |h( k)|)(|x-h(k)(i)|_{X}\leq \frac{1}{k}).
\ee
Intuitively speaking, the weaker normal form \eqref{marhi} no longer states that every 
object in the space has a \emph{standard} object infinitely close by, but only the existence of a `discrete grid' with infinitesimal mesh on the space.  
This grid can be obtained by fixing standard $h$ as in \eqref{marhi} and applying overspill.   
As studied in Sections~\ref{cont1} and \ref{cont2}, the constructive/computable notion of compactness, called \emph{totally boundedness}, is obtained by term extraction from the weaker normal form \eqref{marhi}.     
This term extraction result is particularly elegant, as `dividing a compact space in infinitesimal pieces' features prominently in the intuitive infinitesimal calculus used to date in physics and engineering.

\smallskip

Finally, similar to the equivalence between $\STP$ and \eqref{fanns}, it can be shown that the nonstandard compactness of $X$ is equivalent to the following formula:  
\begin{align}\textstyle
(\forall z\in Z)\big[(\forall^{\st}n)(\exists x\in X)(|x-z|_{Z}<\frac{1}{n}) \di   (\exists^{\st}y\in X)(|y-z|_{Z}\approx 0)],\label{kal}
\end{align}
where $X\subseteq Z$, and the latter is a metric space with metric $|\cdot|_{Z}$.  Note that \eqref{kal} expresses sequential compactness, 
as the antecedent gives rise to a sequence in $X$ converging to $z$; furthermore, it is straightforward to formulate \eqref{kal} for general topological spaces (without a metric), i.e.\ we could treat rather general spaces using \eqref{kal} and \eqref{kalkuttttt}.  A normal form of \eqref{kal} is similar to the one for \eqref{fanns} given by \eqref{frukkklk}, 
namely as follows (and is obtained as in the proof of \cite{samGH}*{Cor.\ 3.4}):
\begin{align}\textstyle
(\forall^{\st}g)(\exists^{\st}w\in X^{*})(\exists^{\st} n^{0})
\Big[(\forall z\in Z)\big[ (\forall y\in  w)&\textstyle (|y-z|_{Z}>_{\R} \frac{1}{g(y)}) \label{kalkuttttt}\\
&\textstyle\di (\forall x\in X)(|x-z|_{Z}>_{\R}\frac{1}{n})   \big]\Big],\notag
\end{align}
where, as in Notation \ref{skim}, we write $z \in X^{*}$ for $z=(z_{0}, \dots, z_{k})$ for a sequence of length $|z|=k$ with $z_{i}\in X$ for $i<|z|$.  
The normal form \eqref{kalkuttttt} constitutes an `external' (as in `from the outside') characterisation of nonstandard compactness:  
a point $z\in Z$ is (standardly) bounded away from $X$ if $g$ bounds all points in the grid $w\in X^{*}$ away from $z$.  
As it happens, we study the associated \emph{anti-Specker property} in the context of higher-order RM in \cite{samnetspilot}.

\subsubsection{Compactness and continuity I}\label{cont1}
In this section, we study the statement $\textsf{CSU}$ that \emph{a continuous function on a compact metric space has a supremum}.  
As we will observe, applying $\CI$ to a suitable nonstandard version of $\CSU$, nonstandard compactness is converted to \emph{totally boundedness}, 
the preferred notion of compactness in constructive and computable mathematics.  

\smallskip

As discussed in Section \ref{cinq}, nonstandard compactness can be weakened to an interesting normal form, which is provided by the following theorem.  
\begin{thm}[$\P_{0}$]\label{noco}
If $X$ is a nonstandard compact metric space, i.e.\ $(\forall x\in X)(\exists^{\st} y\in X)(x\approx y)$, then $X$ is `effectively compact' as follows:
\be\label{straffenbeukje}\textstyle
(\exists^{\st}h)(\forall^{\st} k)(\forall  x\in X)(\exists i< |h( k)|)(|x-h(k)(i)|_{X}\leq \frac{1}{k}),
\ee
which also implies the following
\be\textstyle\label{straffenhendrik}
(\exists^{\st}\psi_{0})(\forall M\in \Omega)(\forall x\in X)(\exists i< |\psi_{0}(M)|)(x\approx \psi_{0}(M)(i)).
\ee
\end{thm}
\begin{proof}
From $(\forall x\in X)(\exists^{\st} y\in X)(x\approx y)$, we obtain 
\be\label{flikko}\textstyle
(\forall ^{\st}k)(\forall x\in X)(\exists^{\st} y\in X)(|x- y|_{X}\leq\frac1k)
\ee
and hence $(\forall ^{\st}k)(\exists^{\st}y'\in X^{*})(\forall x\in X)(\exists y\in y')(|x- y|\leq\frac1k)$ by idealisation~$\textsf{I}$.  
Apply $\HAC_{\INT}$ and let $\Phi$ be the resulting standard functional.  Define $\Xi(k):=\Phi(k)(0)*\Phi(k)(1)*\dots *\Phi(k)(|\Phi(k)|-1)$ and $\psi_{0}(k):= \Xi(1)*\Xi (2)*\dots*\Xi(k)$.  
We obtain that 
\be\label{maki}\textstyle
(\forall ^{\st}k)(\forall x\in X)(\exists i<|\psi_{0}(k)|)\big[|x- \psi_{0}(k)(i)|_{X}\leq\frac1k\big],  
\ee
and \eqref{straffenbeukje} is immediate.
The latter implies \eqref{straffenhendrik} as follows:  fix $x\in X$ in \eqref{maki} and apply overspill to $(\forall^{\st} k)(\exists i< |\psi_{0}( k)|)(|x-\psi_{0}(k)(i)|_{X}\leq \frac{1}{k})$.  
Hence, we obtain $(\exists i< |\psi_{0}( k)|)(|x-\psi_{0}(k)(i)|_{X}\leq \frac{1}{k})$ for $k\leq K_{0}\in \Omega$, which immediately yields $(\forall x\in X)(\exists i< |\psi_{0}(K_{0})|)(x\approx \psi_{0}(k)(i))$.  Redefining $\psi_{0}$ as $\psi'_{0}$ slightly, we can guarantee that $(\forall K\in \Omega)(\forall x\in X)(\exists i< |\psi'_{0}( K)|)(x\approx \psi_{0}(K)(i))$.  Indeed, define $\psi'_{0}(M)$ as $\psi_{0}(K_{1})$, where $K_{1}$ is the largest $K_{0}\leq M$ such that $(\forall k\leq K_{0})(\exists i< |\psi_{0}( k)|)\big(\big[|x-\psi_{0}(k)(i)|_{X}\big](M)\leq \frac{1}{k}\big)$.   
\end{proof}
Note that  \eqref{straffenbeukje} and \eqref{straffenhendrik} constitute a weakening of nonstandard compactness, as there is no longer a \emph{standard} object infinitesimally close in the former.  
This seems to be the price we have to pay in exchange for the `effective content' present in the former.  

\smallskip

The bigger picture regarding \eqref{straffenhendrik} is as follows: essential to both nonstandard Riemann integration and the nonstandard definition of the supremum as in \eqref{frdfg} (and the whole of the infinitesimal calculus), is the fact that we can `divide the unit interval in pieces of infinitesimal length'.     
Intuitively speaking, the functional $\psi_{0}$ from Theorem \ref{noco} allows us to perform a similar `subdivision' of any nonstandard compact space, as $\psi_{0}(i)$ for $i\leq M\in \Omega$ comes infinitely close to any point of the nonstandard compact space.  

\smallskip

In light of the previous, let $\psi_{0}$ be as in Theorem \ref{noco}, and consider the following:
\[\textstyle
\sup_{X}(f, \psi_{0},M):=\max_{i\leq M}[f(\psi_{0}(M)(i))](M),
\] 
to be compared to \eqref{frdfg}.  The nonstandard version of $\CSU$, is then as follows.  
In light of \cite{simpson2}*{I.10.3}, the \emph{pointwise} case cannot be proved without the use of $\WKL_{0}$, explaining why we opted for uniform continuity.    
\begin{thm}[$\CSU_{\ns}$]
Let the metric space $X$ be nonstandard compact and let $f:X\di \R$ be nonstandard uniformly continuous.  Then we have
\begin{align}
\textstyle(\forall x\in X, M\in \Omega)[f(x)\lessapprox \sup_{X}(f, \psi_{0},M)] \notag
\end{align}
\end{thm}
Note that we still obtain the consequent of $\CSU_{\ns}$ if we replace nonstandard compactness in the latter by \eqref{straffenhendrik} itself.  
This observation turns out to be essential to the proof of Theorem~\ref{floggen31}.  The effective version of $\CSU$ is:  
\begin{thm}[$\CSU_{\ef}(t)$]
For all $ f: X\di \R$ with modulus of uniform continuity $g$, and any $h:0\di X^{*}$ such that:
\be
  \textstyle   (\forall k,  x\in X)(\exists i< |h( k)|)(|x-h(k)(i)|_{X}\leq \frac{1}{k})       \big] \label{amaai}
\ee
we have $   (\forall n,x'\in [0,1], N\geq  t(g, h, n)(1))\big[ f(x')\leq \sup_{X}(f, t(g,h,n)(2),N)+\frac{1}{n} \big].$

\end{thm}
Note that a `supremum functional' is readily defined from $\CSU_{\ef}(t)$.  Also note that the term $t$ only depends on how $X$ is represented by $h$ as in \eqref{amaai}.    
\begin{thm}\label{floggen31}
From the proof of $\CSU_{\ns}$ in $\P_{0}$, a term $t$ can be extracted such that $\textup{\textsf{E-PRA}}^{\omega*} $ proves $\CSU_{{\ef}}(t)$.  
\end{thm}
\begin{proof}
The proof of $\CSU_{\ns}$ is straightforward in light of \eqref{straffenhendrik}.  Indeed, the latter and nonstandard uniform continuity guarantee that all $f(x)$ are at most infinitesimally larger than the supremum from $\CSU_{\ns}$.      
Note again that we did not use nonstandard compactness, but that \eqref{straffenhendrik} suffices.  

\smallskip

For the second part of the theorem, $(\forall x\in X)(\exists^{\st}y\in X)(x\approx y)$ implies 
\[\textstyle
(\forall ^{\st}k)(\exists^{\st}y'\in X^{*})(\forall x\in X)(\exists y\in y')(|x- y|\leq\frac1k)
\]
as in the proof of Theorem \ref{noco}.  Using $\HAC_{\INT}$, it is now immediate that 
\be\label{frik}\textstyle
(\exists^{\st}h)(\forall^{\st} k)(\forall  x\in X)(\exists i< |h( k)|)(|x-h(k)(i)|_{X}\leq \frac{1}{k}),  
\ee
which implies \eqref{straffenhendrik} in the same way as in Theorem \ref{noco}.  As noted in this proof and just below $\CSU_{\ns}$, we can replace nonstandard compactness in $\CSU_{\ns}$ by \eqref{frik} and obtain that for $f$ which is nonstandard uniformly continuous on $X$, we have
\begin{align}\textstyle
  (\exists^{\st}h)(\forall^{\st} k)(\forall  x\in X)&  \textstyle(\exists i< |h( k)|)(|x-h(k)(i)|_{X}\leq \frac{1}{k})     \big]\label{derfing}\\
&\textstyle\di (\forall x\in X, M\in \Omega)[f(x)\lessapprox \sup_{X}(f, \psi_{0},M)],\notag
\end{align}
as $\psi_{0}$ can also be obtained from \eqref{frik} by Theorem \ref{noco}.  
Moreover, in light of the proof of Theorem~\ref{noco}, there is a term $s$ such that for all standard $h$ and $f$ which is nonstandard uniformly continuous on $X$, we have
\begin{align}\textstyle
    (\forall^{\st} k)(\forall  x\in X)(\exists i< |h( k)|)&\textstyle(|x-h(k)(i)|_{X}\leq \frac{1}{k})     \big]\label{derfing2}\\
&\textstyle\di (\forall x\in X, M\in \Omega)[f(x)\lessapprox \sup_{X}(f, s(h),M)].\notag
\end{align}
The rest of the proof is now a straightforward application of $\CI$.
\end{proof}
It is important to note that the proof hinges on the fact that nonstandard compactness can be replaced in $\CSU_{\ns}$ by \eqref{frik} to obtain \eqref{derfing} and \eqref{derfing2}.  
We discuss this observation in the following remark.
\begin{rem}[Multiple normal forms]\label{impo}\rm
In the theorems in the previous sections, the conversion of a nonstandard definition to the `effective version' (of continuity, Riemann integration, et cetera) was such that 
\emph{the latter still implied the former}.  As an example, consider nonstandard continuity as in \eqref{soareyou4} which gives rise to the normal form \eqref{first4} and the effective version \eqref{first5}, and \emph{the latter two still imply the former}.  
However, as noted right below Theorem \ref{noco}, \eqref{straffenhendrik}
does not seem to imply nonstandard compactness.  \emph{Nonetheless}, \eqref{frik} does imply \eqref{straffenhendrik}, and the latter suffices for the consequent of $\CSU_{\ns}$, i.e.\ we still obtain \eqref{derfing} and \eqref{derfing2}.  By contrast, it seems the results in \cite{samGH}*{\S4.1} and Section \ref{snarkie2} cannot be obtained with anything weaker than the nonstandard compactness of Cantor space as in \eqref{fanns}.  
\end{rem}
Furthermore, as above, it is straightforward to prove a version of the previous theorem inside $\H$, i.e.\ the associated results also hold constructively\footnote{Note that we \emph{could} apply $\NCR$ to the definition of nonstandard compactness $(\forall x\in X)(\exists^{\st} y\in X)(x\approx y)$ directly \emph{without} removing `$\approx$', but our conservation results will not apply due to the presence of the latter.  Hence, we also apply $\NCR$ to \eqref{flikko} in the constructive case.}.  
On a related note, consider the condition \eqref{amaai} which originates naturally from nonstandard compactness.  The non-uniform version of \eqref{amaai} is:  
\be\label{frier}\textstyle
(\forall  k)(\exists (x_{1}, \dots x_{n})\in X^{*})(\forall  x\in X)(\exists i< n)(|x-x_{i}|_{X}\leq \frac{1}{k}),
\ee
which is nothing more than the notion of \emph{totally boundedness}, the preferred notion of compactness in constructive mathematics (see \cite{bridge1}*{\S4} or \cite{bridges1}*{\S2.2} for a discussion) and computable analysis (see \cite{brapress}*{\S4}).  
Thus, $\CSU_{\ef}(t)$ is essentially \cite{bridge1}*{Cor.\ 4.3, p.~94} with the BHK-interpretation of constructive mathematics (see e.g.\ \cite{bridges1}*{p.\ 8}) made explicit.  
Furthermore, our effective version \eqref{straffenbeukje} corresponds to the notion of \emph{effectively totally bounded} (and the associated compactness) from Computability theory (see e.g.\ \cite{yasugi1}*{Def.\ 3.1} or \cite{omoridesu}*{Def.\ 2.6}).
As it turns out, both the effective \eqref{amaai} and non-effective version \eqref{frier} are studied in computable analysis, as is clear from \cite{brapress}*{Theorem 4.10}.  
  
\smallskip

In conclusion, the constructive and computable notion of compactness \emph{totally boundedness} is implicit in nonstandard compactness, 
and we rediscover the former from the latter by studying $\CSU$ using the template $\CI$.  
However, since we used the weakening \eqref{straffenbeukje} of nonstandard compactness, a Hebrandisation of $\CSU_{\ns}$ cannot be obtained using totally boundedness (see also Section \ref{snarkie2}).      
\subsubsection{Compactness and continuity II}\label{cont2}
In this section, we study another theorem regarding nonstandard compactness, namely the statement $\textsf{CCI}$ that \emph{for a continuous function $f$ from a compact metric space $X$ to a metric space $Y$, $f(X)=\{y: (\exists x\in X)(f(x)=y) \}$ is compact} (see \cite{munkies}*{26.5} or \cite{rudin}*{4.14}).  As in Section \ref{cont1}, we obtain totally boundedness from nonstandard compactness by applying $\CI$.    

\smallskip

First of all, with regard to notation, while the set $f(X)$ may not be defined in $\P_{0}$, the formula $y\in f(X)\equiv (\exists x )(f(x)=_{Y}y)$ makes perfect sense for any $f:X\di Y$.  
Also, (uniform) nonstandard continuity on a metric space $X$ is defined as in Definition~\ref{Kont} with `$[0,1]$' replaced by $X$ (Recall the sixth item of Notation \ref{keepintireal}).

\smallskip

Secondly, the nonstandard and effective versions of $\CCI$ are as follows.  
\begin{thm}[$\CCI_{\ns}$]
For nonstandard uniformly continuous $f:X\di Y$\textup{:}
\begin{align}
(\forall x\in X)(\exists^{\st}x'\in X)(x\approx x') \di  (\forall y\in f(X))(\exists^{\st}y'\in f(X))(y\approx y')\label{forlorn} 
\end{align}  
\end{thm}
\begin{thm}[$\CCI_{\ef}(t)$] For any $f:X\di Y$ with modulus of uniform continuity $g$, and any $h:0\di X^{*}$, we have
\begin{align}\textstyle
  \textstyle   (\forall k,  x\in X)(\exists i&\textstyle< |h( k')|)(|x-h(k')(i)|_{X}\leq \frac{1}{k})   \notag   \\
  &\textstyle\di  (\forall k',  y\in f(X))(\exists i< |t(g,h)( l)|)(|y-t(g,h)(k')(i)|_{Y}\leq \frac{1}{k'}) .\notag
\end{align}
\end{thm}
\begin{thm}\label{floggen34}
From the proof of $\CCI_{\ns}$ in $\P_{0} $, a term $t$ can be extracted such that $\textup{\textsf{E-PRA}}^{\omega*} $ proves $\CCI_{{\ef}}(t)$.  
\end{thm}
\begin{proof}
The proof of $\CCI_{\ns}$ is a straightforward combination of nonstandard compactness and continuity (see \cite{loeb1}*{10.11}).  For the second part, we replace nonstandard compactness \emph{in the consequent} of $\CCI_{\ns}$ 
by the following version of \eqref{straffenbeukje}.  
\be\label{sjokko}\textstyle
(\forall^{\st} l)(\exists^{\st}w\in (f(X))^{*})(\forall  y\in f(X))(\exists i< |w|)(|y-w(i)|_{Y}\leq \frac{1}{l}).
\ee
This replacement is trivial by Theorem \ref{noco}.  In the resulting formula, replace the nonstandard compactness \emph{in the antecedent} by the following version of \eqref{sjokko}:
\be\label{sjokko2}\textstyle
(\forall^{\st} k)(\exists^{\st}z\in X^{*})(\forall  x\in X)(\exists i< |z|)(|x-z(i)|_{X}\leq \frac{1}{k}).
\ee
This replacement is not trivial, but nevertheless correct in $\P_{0}$, as nonstandard uniform continuity \eqref{soareyou4} implies `effective' uniform continuity as in \eqref{first5}.  
The version of $\CCI_{\ns}$ obtained by the replacements \eqref{sjokko} and \eqref{sjokko2} now readily gives rise to $\CCI_{\ef}(t)$ using $\CI$.  
\end{proof}
As in the previous section, nonstandard compactness gives rise to totally boundedness in a natural way.  
As perhaps expected, $\CCI_{\ef}(t)$ is a theorem of Constructive Analysis, namely \cite{bridges1}*{Prop.\ 2.2.6}.  

\smallskip

Finally, the Herbrandisation of $\CCI_{\ns}$ can be obtained by using \eqref{kalkuttttt} for the two occurrences of nonstandard compactness in \eqref{forlorn}, rather than weakening the latter as in the previous proof.  A special case will be discussed in Section \ref{snarkie2}.

\subsubsection{Compactness and continuity III}\label{dinki}  
In this section, we study Heine's theorem, i.e.\ the statement that \emph{a continuous function
is uniformly continuous on a compact set}.  To this end, we formulate an equivalent normal form for the nonstandard compactness of the unit interval, which turns out to be quite similar to Heine-Borel compactness.  As a result of applying $\CI$, we obtain an explicit equivalence between Heine's theorem and (the contraposition of) weak K\"onig's lemma.   

\smallskip

First of all, we formulate a normal form of the nonstandard compactness of $[0,1]$.
\begin{princ}[$\HBL_{\ns}$]
For every sequence of open intervals $(c_{n}, d_{n})$, we have 
\begin{align}\label{HB}
(\forall^{\st}x\in [0,1])&(\exists^{\st}n)(c_{n}\ll x \ll d_{n})\\
&\di (\exists^{\st}k)(\forall^{\st}y\in [0, 1])(\exists i\leq k )(c_{i}\ll y \ll d_{i}).  \notag
\end{align}
\end{princ}
\begin{thm}\label{dickawadddd} In $\P_{0}$, we have $\STP\asa (\forall x\in [0,1])(\exists^{\st}y\in [0,1])\asa \HBL_{\ns}$.
This remains valid if we drop the `\st' in the final universal quantifier of $\HBL_{\ns}$.  
\end{thm}      
\begin{proof}
First of all, assume \eqref{HB} and suppose that $[0,1]$ is not nonstandard compact, i.e.\ $(\exists x_{0}\in [0,1])(\forall^{\st}y\in [0,1])(y\not\approx x_{0})$.  Then let $x_{0}$ be such a number and 
let $q_{i}$ be a standard enumeration of all rationals in $[0,1]$.  Define $f(i)$ as $\big| [x_{0}](N)-q_{i} \big|/2$ for fixed nonstandard $N^{0}$.  By assumption, we have $f(i)\gg 0$ for standard $i$.       
Now consider the sequence $(q_{i}-f(i), q_{i}+f(i))$ and note that we have the antecedent of \eqref{HB}, i.e.\ this sequence standardly covers the unit interval.  
By the consequent of \eqref{HB}, there is finite $k_{0}$ such that $\cup_{i\leq k_{0}}(q_{i}-f(i), q_{i}+f(i))$ standardly covers the unit interval.  
Now fix $M\in \Omega$ and define $q_{i_{0}}$ (resp.\ $q_{i_{1}}$) as the largest $q_{i}\leq [x_{0}](M)$ (resp.\ smallest largest $q_{i}\geq [x_{0}](M)$) for $i\leq k_{0}$.  
By the definition of $f$, we have
\[
q_{i_{0}}\ll q_{i_{0}}+f(i_{0})\ll x_{0} \ll q_{i_{1}}-f(i_{1})\ll q_{i_{1}}.  
\]
However, then there are standard rationals between $q_{i_{0}}+f(i_{0})$ and $q_{i_{1}}-f(i_{1})$ which are not covered by $\cup_{i\leq k_{0}}(q_{i}-f(i), q_{i}+f(i))$, a contradiction.  
Hence, we may conclude that $[0,1]$ is nonstandard compact.   

\smallskip

Secondly, assume that $[0,1]$ is nonstandard compact and suppose \eqref{HB} is false, i.e.\ for some sequence we have the antecedent of the latter, but:  
\[
(\forall^{\st}k)(\exists^{\st}y\in [0, 1])(\forall i\leq k )(c_{i}\gtrapprox y \vee y  \gtrapprox d_{i}),  
\]
which immediately implies for any nonstandard $M$ that:
\[\textstyle
(\forall^{\st}k)(\exists^{\st}y\in [0, 1])(\forall i\leq k )([c_{i}](M)\geq_{0} [y](M)-\frac{1}{k} \vee [y](M)  \geq_{0} [d_{i}](M)-\frac1k).  
\]
Now apply $\HAC_{\INT}$ to obtain standard $\Phi$ such that 
\[\textstyle
(\forall^{\st}k)(\exists y\in \Phi(k))(\forall i\leq k )(y\in [0,1]\wedge  [c_{i}](M)\geq_{0} [y](M)-\frac{1}{k} \vee [y](M)  \geq_{0} [d_{i}](M)-\frac1k).  
\] 
Now apply overspill to the previous formula, i.e.\ there is nonstandard $K^{0}$ such that
\[\textstyle
(\forall k\leq K)(\exists y\in \Phi(k))(\forall i\leq k )(y\in [0,1]\wedge  [c_{i}](M)\geq_{0} [y](M)-\frac{1}{k} \vee [y](M)  \geq_{0} [d_{i}](M)-\frac1k).  
\] 
Thus, let $y_{0}$ be such that $(\forall i\leq K )(y_{0}\in [0,1]\wedge  [c_{i}](M)\geq_{0} [y_{0}](M)-\frac{1}{K} \vee [y_{0}](M)  \geq_{0} [d_{i}](M)-\frac1K)$, and let standard $x_{0}\in [0,1]$ be such that $x_{0}\approx y_{0}$.   
Combining these two facts, we observe that $x_{0}\in [0,1]$ is a standard real such that $(\forall^{\st}i)(c_{i}\gtrapprox x_{0}\vee x_{0}\gtrapprox d_{i})$, a contradiction.  
Hence, \eqref{HB} follows and the equivalence from the theorem has been proved.  

\smallskip

Thirdly, we show that \eqref{HB} generalises as in the last sentence of the theorem.  
To this end, assume \eqref{HB} for some sequence of opens $(c_{n}, d_{n})$, let $k_{0}$ be as in the consequent of \eqref{HB}, and suppose there is $y\in [0,1]$ such that 
$(\forall i\leq k_{0} )(c_{i}\gtrapprox y \vee y\gtrapprox    d_{i})$.  
In the previous formula, we cannot have $y\approx c_{i}$ or $y\approx d_{i}$;  indeed, the numbers $c_{i}, d_{i}$ for $i\leq k_{0}$ have a standard part by nonstandard compactness (which follows from \eqref{HB} as proved in the previous paragraph of this proof) and
such $c_{i}, d_{i}$ are standardly covered by $\cup_{i\leq k_{0}}(c_{i}, d_{i})$.  
Hence, we have $(\forall i\leq k_{0} )(c_{i}\gg y \vee y\gg    d_{i})$.
Now let $j_{0}$ be such that for all $j\leq k_{0}$ with $[c_{j}](M)- [y](M)\geq_{0} 0$, the number $[c_{j_{0}}](M)- [y](M)$ is minimal (where $M\in \Omega$ is fixed).  
Similarly, let $j_{1}$ be such that for all $j\leq k_{0}$ with $[d_{j}](M)- [y](M)\leq_{0} 0$, the number $ [y](M)- [d_{j_{1}}](M)$ is minimal.  
Again by standard compactness, we may assume $d_{j_{1}}$ and $c_{j_{0}}$ to be standard reals, and hence covered by $\cup_{i\leq k_{0}}(c_{i}, d_{i})$.  This implies
\[
[c_{j_{0}}](M)\approx c_{j_{0}} \ll y \ll d_{j_{1}} \approx [d_{j_{1}}](M), 
\]
i.e.\ there are standard numbers not covered by $\cup_{i\leq k_{0}}(c_{i}, d_{i})$ between $c_{j_{0}}$ and $d_{j_{1}}$, a contradiction.  Hence, the equivalence from the theorem remains valid if we drop the `\st' in the final universal quantifier of \eqref{HB}.    

\smallskip

Finally, for the equivalence involving $\STP$, every real has a binary representation by \cite{polahirst}*{Cor.\ 4}.  Hence, for every real $x\in [0,1]$, there is $\beta^{1}\leq_{1}1$ such that $x=_{\R}\sum_{i=1}^{\infty}\frac{\beta(i)}{2^{i}}$.  Apply $\STP$ to find standard $\alpha^{1}\leq_{1}1$ such that $\alpha\approx_{1}\beta$ and note that $x\approx y$ for the \emph{standard} real $y:=\sum_{i=1}^{\infty}\frac{\alpha(i)}{2^{i}}$.  Hence, $\STP$ implies the nonstandard compactness of $[0,1]$, and the reverse implication is proved in the same way.  
\end{proof}
Note that \eqref{HB} deals with \emph{arbitrary covers}, not just standard ones.  Although \eqref{HB} involves a number of standard quantifiers, it does have a straightforward normal form, as will become clear in the proof of Theorem \ref{floggen345}.     
The generalisation of the previous theorem to general metric spaces is straightforward (see also \cite{simpson2}*{IV.1.5}), but we do not go into details.    

\smallskip

The nonstandard and effective versions of Heine's theorem are as follows.  
\begin{thm}[$\HEI_{\ns}$]
For any $f:\R\di \R$, if $f$ is nonstandard continuous on $[0,1]$, then it is uniformly nonstandard continuous there.  
\end{thm}
\begin{thm}[$\HEI_{\ef}(t)$]
For $f, g:\R\di \R$, if $g$ is a modulus of continuity for $f$ on $[0,1]$, then $t(g)$ is a modulus of uniform continuity for $f$ on $[0,1]$.
\end{thm}
The effective version of Heine-Borel compactness is as follows:
\begin{align}\tag{$\HBL_{\ef}(h)$}
(\forall g^{2}, c_{(\cdot)}, d_{(\cdot)})\big[(\forall x\in [0,1])(\exists&\textstyle n, k\leq g(x))[c_{n}+\frac{1}{k}< x < d_{n}-\frac1k]\label{cardargo}\\ 
&\textstyle\di    (\forall y\in [0, 1])(\exists i, j\leq h(g )(c_{i}+\frac{1}{j}<y < d_{i}-\frac{1}{j}).\notag
\end{align}
The reader might wonder `how constructive' \ref{cardargo} is;  this shall be discussed in detail in Remark \ref{snark}.  
We first prove the following theorem.
\begin{thm}\label{floggen345}
From the proof of $ \HEI_{\ns}$ in $\P_{0}+\STP$, a term $t$ can be extracted such that $\textup{\textsf{E-PRA}}^{\omega*} $ proves $(\forall h)[\HBL_{\ef}(h)\di \HEI_{{\ef}}(t(h))]$.  
\end{thm}
\begin{proof}
The proof $\P\vdash\STP\di \HEI_{\ns}$ follows immediately from Theorem \ref{dickawadddd}.  
By the latter, the implication also remains valid if we replace nonstandard compactness by Heine-Borel compactness $\HBL_{\ns}$.
We now bring \eqref{HB} in the required normal form.  
First of all, the antecedent of \eqref{HB} yields
\be\label{horum}\textstyle
(\forall^{\st}x\in [0,1])(\exists^{\st}n, k)[c_{n}+\frac{1}{k}< x < d_{n}-\frac1k], 
\ee
to which we may apply $\HAC_{\INT}$ to obtain standard $g$ such that 
\be\label{horum2}\textstyle
(\forall^{\st}x\in [0,1])(\exists n, k\leq _{0}g(x))[c_{n}+\frac{1}{k}< x < d_{n}-\frac1k], 
\ee
Secondly, the consequent of \eqref{HB} (in its general form) yields
\[\textstyle
(\exists^{\st}k)(\forall y\in [0, 1])(\exists^{\st}l)(\exists i\leq k )(c_{i}+\frac1l <y < d_{i}-\frac{1}{l}),
\]  
and applying idealisation \textsf{I}, we obtain
\be\label{drekke}\textstyle
(\exists^{\st}k)(\exists^{\st}n)\big[(\forall y\in [0, 1])(\exists l\leq n)(\exists i\leq k )(c_{i}+\frac1l <y < d_{i}-\frac{1}{l})\big],
\ee
which also has the right syntactial form.  Finally, let $A(x, n,k)$ be the formula in square brackets in \eqref{horum} and let $B(k,n)$ be the formula in square brackets in \eqref{drekke}.  
Then \eqref{HB} is equivalent to 
\be\label{farfoo}
(\forall^{\st}g)(\exists^{\st}x, k', n')\big[ (\exists n,k\in g(x))A(x, n, k)\di B(k',n')].
\ee
Now prefix \eqref{farfoo} by the usual $(\forall c_{(\cdot)}, d_{(\cdot)})$-quantifier and apply idealisation \textsf{I} one more time to pull the standard extensional quantifier through the former quantifier.  
Thus, we obtain a normal form for \eqref{HB} as follows:
\be\label{essentip}
(\forall^{\st}g)(\exists^{\st}z^{1^{*}})\big[(\forall c_{(\cdot)}, d_{(\cdot)})(\exists x, k', n'\in z)C(g, x, k', n')\big], 
\ee
where $C$ is the formula in square brackets in \eqref{farfoo}.  The normal form of nonstandard pointwise continuity is:
\be\textstyle\label{harny}
(\forall^{\st}  x'\in [0,1], k)(\exists^{\st}N)\big[(\forall y \in [0,1], k)\textstyle(|x'-y|<\frac{1}{N} \di |f(x')-f(y)|\leq\frac{1}{k})\big], 
\ee
which gives rise to the following after applying $\HAC_{\INT}$:
\be\textstyle\label{harny2}
(\exists^{\st}h)(\forall^{\st}  x'\in [0,1], k)(\forall y \in [0,1], k)\textstyle(|x'-y|<\frac{1}{h(x,k)} \di |f(x')-f(y)|\leq\frac{1}{k}). 
\ee
Now let $D(g, z)$ be the formula in square brackets in \eqref{essentip}, let $E(f, k, N)$ be the formula in square brackets in \eqref{deep2}, and let $F(f, x,k, N)$ be the formula in square brackets in \eqref{harny}; $\STP\di \HEI_{\ns}$ implies that for $f:\R\di \R$ and standard $\Xi$, $h^{2}$:
\be\label{flacka}
(\forall^{\st}x', l)F(f, x',l, h(x', l))\wedge (\forall^{\st}g)D(g, \Xi(g)) \di (\forall^{\st}k^{0})(\exists^{\st}N^{0})E(f, k,N)
\ee
The previous formula has almost the right form to apply $\CI$, but we need to make slight modifications still.  First of all, we may trivially drop the `st' on the quantifier involving $x'$ for $F$ in \eqref{flacka};  
this stronger antecedent implies that for fixed $k_{0}$, the cover  $(q-\frac{1}{h(q, k_{0})}, q+\frac{1}{h(q, k_{0})})$ for all rational $q^{0}$ covers \emph{all reals} in the unit interval, not just the standard ones.  

\smallskip

Secondly, we do not need \eqref{farfoo} to obtain a finite sub-cover of the aforementioned cover of $[0,1]$, but it suffices to have the weakening of \eqref{farfoo} by dropping the `st' in the quantifier pertaining to $x$.  
Now repeat the above steps to obtain \eqref{flacka} with these modifications (In particular, drop the `st' predicate for $x$ in \eqref{farfoo} and $x'$ in \eqref{harny2}) in place.  
This amounts to replacing $(\exists x, k', n'\in z)$ in \eqref{essentip} by $(\exists  k', n'\in z)(\exists x)$.  
The theorem now follows from applying $\CI$ to the version of \eqref{flacka} resulting from the aforementioned modifications.  
\end{proof}
In the previous proof, we applied $\CI$ to a slight \emph{modification} of $\HBL_{\ns}\di \HEI_{\ns}$.  
In Section \ref{snarkie2}, we discuss what happens if no such modification is made by considering the Herbrandisation of $\STP$.   
\begin{cor}\label{fralgyy}
From the proof of $ \STP$ in $\P_{0}+\HEI_{\ns}$, a term $s$ can be extracted such that $\textup{\textsf{E-PRA}}^{\omega*} $ proves $(\forall g)[\HEI_{\ef}(g)\di \HBL_{\ef}(s(g))]$.  
\end{cor}
\begin{proof}
In a nutshell, one simply repeats the proof of the theorem, but for: 
\be\label{decv}
[(\forall f:\R\di \R)(\eqref{soareyou3}\di \eqref{soareyou4})]\di (\forall x\in [0,1])(\exists^{\st}y\in [0,1])(x\approx y).
\ee
To prove the implication \eqref{decv}, suppose there is $x_{0}\in [0,1]$ such that $(\forall^{\st}y\in [0,1])(x_{0}\not\approx y)$.  
Fix nonstandard $N$ and define $f_{0}(x):= \frac{1}{|x-x_{0}|+\frac{1}{N}}$.  
Note that $f_{0}(y)\approx \frac{1}{|x_{0}-y|}$ for standard $y$, and therefore $f_{0}$ satisfies \eqref{soareyou3}, but clearly not \eqref{soareyou4}, 
as uniform nonstandard continuity fails for any $z\approx x_{0}$.  
\end{proof}
Now, the contraposition of weak K\"onig's lemma is called the \emph{fan theorem} and denoted $\FAN$.    
Recall that $\STP$ is a nonstandard version of weak K\"onig's lemma in light of \eqref{fanns}.  The contraposition of the latter may be called $\FAN_{\ns}$, while the effective version is defined as follows:
\be\tag{$\FAN_{\ef}(h)$}
(\forall T^{1}\leq_{1}1, g^{2})\big[ (\forall \alpha\leq_{1}1)(\overline{\alpha}g(\alpha)\not\in T)\di (\forall \beta\leq_{1}1)(\overline{\beta}h(g)\not\in T)   \big].
\ee
Hence, the following corollary is proved in the same way as above.  
\begin{cor}\label{firk}
From the proof $\P_{0}\vdash \STP\asa \HEI_{\ns}$, terms $s, t$ can be extracted such that $\textup{\textsf{E-PRA}}^{\omega*} $ proves $(\forall g)[\FAN_{\ef}(g)\di \HEI_{\ef}(s(g))]$ and $(\forall h)[\HEI_{\ef}(h)\di \FAN_{\ef}(s(h))]$.
\end{cor}
An implication $(\exists \Phi)A(\Phi)\di (\exists \Psi)B(\Psi)$ (proved in $\RCAo$) is \emph{explicit} if there is a term $t$ in the language such that additionally $(\forall \Phi)[A(\Phi)\di B(t(\Phi))]$ (is proved in $\RCAo$), i.e.\ $\Psi$ can be explicitly defined in terms of $\Phi$.
Hence, we have established the effective equivalence between Heine-Borel compactness $(\exists h)\HBL_{\ef}(h)$, the fan theorem $(\exists k)\FAN_{\ef}(k)$,
and Heine's theorem $(\exists g)\HEI_{\ef}(g)$.  The non-effective equivalence between these theorems is well-known from RM (see \cite{simpson2}*{IV.2.3}).

\smallskip

In conclusion, we have obtained explicit versions of certain equivalences from the RM of $\WKL_{0}$.  
Intuitively speaking\footnote{More technically, Kohlenbach has pointed out in \cite{kohlenbach2}*{\S3} that theorems classically equivalent to $\WKL_{0}$ can have a uniform version either at the level of $\WKL_{0}$ or at the level of $\ACA_{0}$, depending on the original theorem's constructive status.}, similar results for other such equivalences to $\WKL_{0}$ are immediate, provided the theorems involved have the same syntactic structure as $\FAN$.  As an example of the latter, but also of independent interest, we shall study \emph{Dini's theorem} in Section \ref{dinkitoes}.    
On the other hand, theorems equivalent to $\WKL_{0}$ (classically and constructively) which have a syntactical structure similar to weak K\"onig's lemma, are studied in Section~\ref{lapier}.  

\subsubsection{Compactness and continuity IV}\label{lapier}
In this section, we study the \emph{Weierstra\ss~maximum principle} (\cite{simpson2}*{IV.2.2}) and the \emph{intermediate value theorem} (\cite{simpson2}*{II.6.2}). 
Both theorems are equivalent to $\WKL_{0}$, respectively in classical and constructive RM.  

\smallskip

We shall observe that there are (at least) two possible normal forms for each of these theorems, called the \emph{weak} and the \emph{strong} version, depending on whether the consequent applies to the standard world or to all of the universe.  
By way of example, running the weak (resp.\ strong) version of the intermediate value theorem through $\CI$, we obtain the constructive version from \cite{bridge1}*{p.\ 40} (resp.\ the uniform version from \cite{kohlenbach2}*{\S3} equivalent to $(\exists^{2})$).  
We again stress that these results are not new, but that it is surprising that we rediscover them through $\CI$, i.e.\ via a purely mechanised/syntactical routine \emph{without attention to the meaning of the theorems}.    

\smallskip

First of all, the aforementioned theorems are interesting because they have the same syntactical structure as $\WKL_{0}$, which excludes the (direct) treatment from the previous section involving the (effective version of the) fan theorem.  
In particular, the consequent of the aforementioned theorems (when formulated relative to `st') is of the form $(\exists^{\st}f^{1})(\forall^{\st}x)\varphi(f, x)$, i.e.\ \emph{not} the normal form.  To obtain our cherished normal form and apply $\CI$, there are at least two options.
\begin{enumerate}         
\item Apply \emph{Transfer} to `$(\forall^{\st}x)\varphi(f, x)$' to obtain $(\exists^{\st}f^{1})(\forall x)\varphi(f, x)$ as the consequent of the theorem.  
\item Study the \emph{contraposition} of the nonstandard theorem (which has the syntactical structure of $\FAN$, and hence a normal form by the previous).  
\end{enumerate}
Hence, there are \emph{two} normal forms in this case.  The \emph{strong} nonstandard and effective versions of the intermediate value theorem ($\IVT$ for short) are defined as:      
\begin{thm}[$\IVT_{\ns}$] For standard $f:\R\di\R$ which is nonstandard continuous on $[0,1]$, if $f(0)>0\wedge f(1)<0$ then 
$(\exists^{\st}z\in [0,1])(f(z)=_{\R}0)$.  
\end{thm}
\begin{thm}[$\IVT_{\ef}(t)$]
For $f:\R\di \R$ with modulus of continuity $g$, if $f(0)>0\wedge f(1)<0$ then $f(t(g,f))=_{\R}0$.
\end{thm}
For the \emph{weak} nonstandard version of $\IVT$, let $\IVT_{\ns}'$ be $\IVT_{\ns}$ with `$=_{\R}$' in the consequent replaced by `$\approx$' and the `st' in $(\forall^{\st}f)$ dropped.  
Let $\IVT_{\ef}'(t)$ be $\IVT$ for uniformly continuous functions (with a modulus $g$) and the consequent weakened to $(\forall k^{0})(t(g,k)\in [0,1]\wedge |f(t(g,k))|\leq  \frac{1}{k})$, i.e.\ an approximate intermediate value.  

\smallskip

Kohlenbach has established the equivalence $(\exists t)\IVT_{\ef}(t)\asa (\exists^{2})$ in \cite{kohlenbach2}*{\S3}.   
Furthermore, $\IVT_{\ef}'(t)$ is the constructive version of $\IVT$ by \cite{bridge1}*{Theorem 4.8, p.\ 40}.  
\begin{thm}\label{sef8}
From the proof of $\IVT_{\ns}\asa \paai$ in $\P_{0} $, two terms $s, u$ can be extracted such that $\textup{\textsf{E-PRA}}^{\omega*}$ proves:
\be\label{frood8}
(\forall \mu^{2})\big[\textsf{\MU}(\mu)\di \IVT_{\ef}(s(\mu)) \big] \wedge (\forall t^{(1\di 1)\di 1})\big[ \IVT_{\ef}(t)\di  \MU(u(t))  \big].
\ee
From the proof of $\IVT_{\ns}'$ in $\P_{0} $, a term $t$ can be extraced s.t.\ $\textup{\textsf{E-PRA}}^{\omega*}\vdash\IVT_{\ef}'(t)$.  
\end{thm}
\begin{proof}
The first part of the proof is similar to that of Theorem \ref{sef}.  For the forward implication, the (classical) proof of $\IVT$ in \cite{simpson2}*{II.6.2} goes through in $\P_{0}$ releative to `st'.  
Applying $\paai$ to the consequent of $\IVT^{\st}$, we obtain $\IVT_{\ns}$.  For the reverse implication, assume $\IVT_{\ns}$ and suppose $\paai$ is false, i.e.\ there is standard $h$ such that $(\forall^{\st}n)h(n)=0\wedge (\exists m)h(m)\ne0$. 
Now define the \emph{standard} function $f_{0}$ (which is also nonstandard uniformly continuous) as follows:
\[
[f_{0}(x)](k):=
\begin{cases}
[x-\frac{1}{2}](k) & \textup{otherwise}\\
\frac{1}{2^{m(k)}} & (\exists i\leq k)h(i)\ne0 \wedge [x-\frac{1}{2}](k)<\frac{1}{2^{m(k)}}
\end{cases}
\]
where $m(k)$ is $(\mu j\leq k)(h(j)\ne 0)$, if such there is and zero otherwise.   
Intuitively speaking, $f_{0}$ is just $x-\frac{1}{2}$ but `capped' around $x=\frac{1}{2}$.  
Clearly, there cannot be $x_{0}\in [0,1]$ such that $f_{0}(x_{0})=0$, and $\IVT_{\ns}$ yields a contradiction.  It is now trivial to bring $\IVT_{\ns}$ into a normal form and apply $\CI$.         

\smallskip

For the second part, $\IVT_{\ns}'$ follows in the same way as in the first part of the proof.  We now bring the former into its normal form.  Using the normal form for continuity as in \eqref{EST}, we obtain for all $f$ and standard $g$ that
\begin{align*}
\big[\textstyle(\forall^{\st}k)(\forall  x, y \in [0,1])(|x-y|<\frac{1}{g(k)}&\textstyle \di |f(x)-f(y)|\leq\frac{1}{k}) \wedge f(0)>0\wedge f(1)<0\big]\\
&\di (\exists^{\st}z\in [0,1])(\forall^{\st}l^{0})\textstyle(|f(z)|\leq \frac{1}{l}) \big).  
\end{align*}
The (contraposition of the) previous formula implies that for all $f$
\begin{align}
\big( (\forall^{\st}z\in [0,1])(\exists^{\st}l^{0})&\textstyle(|f(z)|> \frac{1}{l})  \wedge f(0)>0\wedge f(1)<0\big) \label{nokiii}\\ 
&\di \textstyle(\forall^{\st}g)(\exists ^{\st}k)(\exists   x, y \in [0,1])(|x-y|<\frac{1}{g(k)}\wedge |f(x)-f(y)|>\frac{1}{k}).\notag
\end{align}
Strengthening the antecedent of \eqref{nokiii}, we obtain, for all $f$ and standard $h^{2}$,
\begin{align}
\big( (\forall^{\st}z\in [0,1])&\textstyle(|f(z)|> \frac{1}{h(z)})  \wedge f(0)>0\wedge f(1)<0\big)\label{discharge} \\ 
&\di \textstyle(\forall^{\st}g)(\exists ^{\st}k)(\exists   x, y \in [0,1])(|x-y|<\frac{1}{g(k)}\wedge |f(x)-f(y)|>\frac{1}{k}).\notag
\end{align}
Now \eqref{discharge} clearly has a normal form in $\P_{0}$ and $\CI$ provides a term $t$ such that $\textsf{E-PRA}^{\omega*}$ proves for all $f,g, h$ that
\begin{align*}
\big( (\forall z\in t(g, h)(2))&\textstyle(z\in [0,1]\di |f(z)|> \frac{1}{h(z)})  \wedge f(0)>0\wedge f(1)<0\big) \\ 
&\di \textstyle(\exists   x, y \in [0,1])(|x-y|<\frac{1}{g(t(g,h)(1))}\wedge |f(x)-f(y)|>\frac{1}{t(g,h)(1)}).
\end{align*}
Again taking the contraposition, we obtain that for all $f, g, h$ that 
\begin{align}\label{fraak}
\big(\textstyle(\forall&\textstyle   x, y \in [0,1])(|x-y|<\frac{1}{g(t(g,h)(1))}\di\textstyle ~|f(x)-f(y)|\leq \frac{1}{t(g,h)(1)})\\
& \wedge f(0)>0\wedge f(1)<0\big) \di  (\exists z\in t(g,h)(2))\textstyle(z\in [0,1]\wedge |f(z)|\leq  \frac{1}{h(z)}) \notag
\end{align}
which is as required\footnote{Note that we can decide if $|f(z)|\geq \frac{1}{k}$ or $|f(z)|\leq\frac{2}{k}$ for $k^{0}>0$ by \cite{bridge1}*{Cor.~2.17}.} by the theorem if we take $h$ to be the function which always outputs $k^{0}$, and if $g$ is assumed to be a modulus of (uniform) continuity.
\end{proof}
Next, let $\IVT_{\ns}''$ be $\IVT_{\ns}$ with `$(\exists z)$' and `$\approx$' instead of `$(\exists^{\st}z)$' and `$=$' in the consequent.  We shall also refer to the former as a weak version of $\IVT$.      
The (proof of the) following corollary shows that the second part of the theorem can be improved considerably if we consider the meaning/content of the intermediate value theorem when obtaining a normal form for the latter.  
\begin{cor}\label{nomilo}
From the proof of $\IVT_{\ns}''$ in $\H $, a term $t$ can be extraced s.t.\ $\textup{\textsf{E-HA}}^{\omega*}\vdash\IVT_{\ef}'(t)$.
\end{cor}
\begin{proof}
To prove $\IVT_{\ns}''$ in $\H$, fix $M\in \Omega$ and perform the usual `interval halving technique' for $M$ steps using approximations $[~\cdot~](M)$ of all reals involved.  
By continuity, the consequent $(\exists z\in [0,1])(f(z)\approx 0)$ implies $(\forall^{\st}k^{0})(\exists^{\st}w)(|f(w)|\leq\frac{1}{k})$.  
The corollary now follows by applying $\CI$ with the latter consequent. 
\end{proof}
Analogous results for Weierstra\ss' maximum theorem, and indeed for any theorem equivalent to $\WKL_{0}$ with the same syntactical structure as $\WKL_{0}$, are now straightforward.  
As an example, the \emph{strong} nonstandard and effective versions of the Weierstra\ss' maximum theorem are defined as follows:
\begin{thm}[$\textsf{WEIMAX}_{\ns}$]
For standard $f:\R\di \R$ which is nonstandard continuous on $[0,1]$, we have $(\exists^{\st}x\in[0,1])(\forall y\in [0,1])(f(y)\leq f(x))$.
\end{thm}
\begin{thm}[$\textsf{WEIMAX}_{\ef}(t)$]
For any $f:\R\di \R$ with modulus of continuity $g$ on $[0,1]$, we have 
$(\forall y\in [0,1])(f(y)\leq f(t(f,g)))$.  
\end{thm}
A result analgous to \eqref{frood8} for the previous two sentences, is proved as in the proof of the theorem (see also \cite{kohlenbach2}*{\S3}).  The weak versions are obtained by the same weakening as for $\IVT$.   

\smallskip

Finally, note that for the weak effective version of $\IVT_{\ef}'(t)$, if we additionally assume that there is at most one intermediate value, i.e.\
\[
(\forall x, y\in [0,1])(x\ne y\di f(x)\ne 0 \vee f(y)\ne0), 
\]  
then the approximation provided by $t$ converges to the unique intermediate value of $f$.  The same holds for the Weierstra\ss~maximum theorem (and in general), additionally assuming $\WKL$ in the base theory to guarantee that a maximum exists at all.  Unique existence results have been studied in constructive RM (\cite{ishberg}).  

\smallskip

In conclusion, theorems with the same syntactic structure as $\WKL_{0}$ seem to have two possible normal forms, the \emph{strong and weak} nonstandard versions.  
When run through $\CI$, the first one gives rise to an explicit equivalence to $(\mu^{2})$, while the second one gives rise to the `approximate' version from constructive mathematics.   
\subsection{Two versions of Dini's theorem}\label{dinkitoes}
In this section, we investigate \emph{Dini's theorem} which states: \emph{a monotone sequence of continuous functions converging pointwise to a continuous function on a compact space, converges uniformly} (\cite{rudin}*{p.\ 150}).  The classical case leads to the discovery of a `higher-order' RM zoo, which is interesting as follows: while the RM zoo from \cite{damirzoo} contains lots of non-equivalent theorems outside the Big Five, \emph{uniform} RM zoo theorems are either equivalent to $(\exists^{2})$ or in some cases behave like the fan theorem by the results in \cite{samzoo, samzooII}.  In other words, the `higher-order RM zoo' looks rather boring in comparison to the original zoo, but our results in Section \ref{snarkie2} will change that.

\subsubsection{Classical Dini's theorem}\label{snarkie}
In this section, we study {Dini's theorem} in classical mathematics using $\CI$.    
Following \cite{diniberg}*{\S5}, Dini's theorem is equivalent to weak K\"onig's lemma in classical Reverse Mathematics.
Hence, we shall study the nonstandard version of Dini's theorem restricted to $[0,1]$ as follows.  
Recall the definition of convergence from Definition \ref{nsconv}.
\begin{thm}[$\DINI_{\ns}$] For all $f_{n}, f:\R\di \R$, if 
\begin{enumerate}
\item $f, f_{n}$ are nonstandard pointwise continuous for standard $n$ on $[0,1]$,
\item  $f_{n}(x)$ nonstandard converges to $f(x)$ for standard $x\in [0,1]$, and
\item $ (\forall x\in [0,1], n )(f_{n}(x)\geq f_{n+1}(x)\geq f(x))$
\end{enumerate}
then $f_{n}$ \emph{uniformly} nonstandard converges to $f$ in $[0,1]$.  
\end{thm}
\begin{thm}[$\DINI_{\ef}(t)$] For all $f_{(\cdot)}, g_{(\cdot)}, f, g, h$, if 
\begin{enumerate}
\item $f, f_{n}$ are pointwise continuous on $[0,1]$ with moduli $g, g_{n}$ for all $n$,
\item  $f_{n}(x)$ converges to $f(x)$ with modulus $h(x)(\cdot)$ for all $x\in [0,1]$, and
\item $ (\forall x\in [0,1], n )(f_{n}(x)\geq f_{n+1}(x)\geq f(x))$,
\end{enumerate}
then $f_{n}$ \emph{uniformly} converges to $f$ in $[0,1]$ with modulus $t(h, g, g_{(\cdot)})$.  
\end{thm}
We have the following theorem, with a short and elegant nonstandard proof.
\begin{thm}\label{linkebeek}
From $\P_{0}\vdash\STP\asa \DIN_{\ns}$, terms $s,t$ can be extracted such that \textsf{\textup{E-PRA}}$^{\omega*}$ proves $(\forall h)(\HBL_{\ef}(h)\di\DINI_{\ef}(t(h))$ 
and $(\forall g)(\DINI_{\ef}(g)\di\HBL_{\ef}(s(g))$.  
\end{thm}
\begin{proof}
The proof of Dini's theorem in \cite{loeb1}*{Theorem 13.3, p.\ 61} is easily seen to yield a proof of $\DIN_{\ns}$ inside $\P_{0}+\STP$.  
For completeness, we now sketch this proof.  Fix $x_{0}\in [0,1]$ and nonstandard $N_{0}$, and let standard $y_{0}$ be such that $y_{0}\approx x_{0}$ by Theorem~\ref{dickawadddd}.    
Now apply overspill (Theorem \ref{doppi}) to $(\forall^{\st}n)(f_{n}(y_{0})\approx f_{n}(x_{0}))$ to obtain $(\forall n\leq N_{1})(f_{n}(y_{0})\approx f_{n}(x_{0}))$ where $N_{1}$ is nonstandard.  
If $N_{0}\leq N_{1}$ then $f(x_{0})\approx f(y_{0})\approx f_{N_{0}}(y_{0})\approx f_{N_{0}}(x_{0})$.  If $N_{0}>N_{1}$, then by assumption:
\[
f(x_{0})\leq f_{N_{0}}(x_{0})\leq  f_{N_{1}}(x_{0})\approx f_{N_{1}}(y_{0})\approx f(y_{0})\approx f(x_{0}).  
\]
In each case, $f_{N_{0}}(x_{0})\approx f(x_{0})$, implying nonstandard convergence as required.  For the proof of $\DINI_{\ns}
\di \STP$, one could just use Theorem \ref{dickawadddd} and translate the results in \cite{diniberg}*{\S4} to $\P_{0}$.   We now provide a shorter proof inspired by that of Corollary \ref{fralgyy}.  
To this end, assume $\DINI_{\ns}$ and suppose $\STP$ is false, i.e.\ there is $x_{0}\in [0,1]$ such that $(\forall^{\st}y\in [0,1])(x\not\approx y)$.  
Now fix nonstandard $N_{0}$ and define $f(x):= \frac{1}{|x_{0}-x|+{1}/{N_{0}}}$ and $f_{n}(x):= \frac{1}{{1}/{n}+|x-x_{0}|}$.  
By assumption, these functions are nonstandard pointwise continuous for standard $n$, and clearly monotone as in $(\forall x\in [0,1])(\forall n^{0})(f_{n}(x)\leq f_{n+1}(x)\leq f(x))$.    
However, we have $(\forall^{\st}x\in [0,1])(\forall N\in \Omega)(f_{N}(x)\approx f(x))$, while $f_{N_{0}+1}(x_{0})=N_{0}+1 \not\approx N_{0}=f(x_{0})$, contradicting $\DINI_{\ns}$, and the latter is seen to imply $\STP$ by Theorem \ref{dickawadddd}.  

\smallskip

For the remaining part of the theorem, apply $\CI$ to $\HBL_{\ns}\asa\DIN_{\ns}$ while performing the same modifications to the normal forms of $\HBL_{\ns}$ and pointwise continuity as was done at the end of the proof of Theorem \ref{floggen345}.  
\end{proof}
In the previous theorem, the modulus of uniform convergence provided by the term $t$ in $\DIN_{\ef}(t)$ is defined in terms of the functional $h$ as in $\textsf{HBL}(h)$.      
In Remark~\ref{snark}, we discuss the constructive status of the functional $h$ from $\HBL(h)$.  
\subsubsection{Hebrandisation and fan functionals}\label{snarkie2}
In the previous section, an effective equivalence involving Dini's theorem was derived from the associated nonstandard equivalence $\STP\asa \DINI_{\ns}$. 
Since we have already identified $(\mu^{2})$ as the Hebrandisation of $\paai$ in Section \ref{X}, it is a natural question, studied in this section, what the Herbrandisation of $\STP$ could be.  

\smallskip

To this end, we now consider `the' \emph{special fan functional} $\Theta^{2\di (1^{*}\times 0)}$, first introduced in  \cite{samGH}*{\S3} and defined\footnote{Note that the original definition of $\SCF(\Theta)$ in \cite{samGH} is based on \eqref{frukkklk}.  These definitions are equivalent up to a term of G\"odel's $\textsf{T}$, as shown in \cite{dagsam, dagsamII}.} as follows:     
\be\tag{$\SCF(\Theta)$}
(\forall G^{2}, f\leq_{1}1)(\exists i<|\Theta(G)|)\big(  f\in \big[\overline{\Theta(G)(i)}G(\Theta(G)(i)) )\big]  \big), 
\ee
which expresses that the uncountable cover $\cup_{f\in 2^{\N}}[\overline{f}G(f)]$ of Cantor space has a finite sub-cover $\cup_{i\leq k}[\overline{f_{i}}G(f_{i})]$ provided by the finite sequence $\Theta(G)=\langle f_{0}, \dots, f_{k}\rangle$.
Any functional $\Theta$ satisfying $\SCF(\Theta)$ is called `a' special fan functional, as there is no unique such object.  By the following theorem, the existence of a special fan functional is indeed the Herbrandisation of $\STP$.    
\begin{thm}\label{karwinkel}
Let $\varphi$ be internal.  
From a proof $\P_{0}+\STP\vdash(\forall^{\st}x)(\exists^{\st}y)\varphi(x,y)$, 
a term $t$ can be extracted such that $\textsf{\textup{E-PRA}}^{\omega*}$ proves $ (\forall \Theta)\big(\SCF(\Theta)\di   (\forall x)(\exists y\in t(x, \Theta))\varphi(x,y)\big)$.  
Furthermore, $\P_{0}$ proves $(\exists^{\st}\Theta)\SCF(\Theta)\di \STP$.  
\end{thm}
\begin{proof}
The first part is a straightforward application application of $\CI$ given \eqref{frukkklk}.  
Indeed, let $(\forall^{\st}g^{2})(\exists^{\st}w^{1^{*}})B(g,w)$ be the normal form of \eqref{frukkklk} and apply Corollary~\ref{consresultcor} to the following normal form:
\[
(\forall^{\st}x, \Theta)(\exists^{\st}y)\big[(\forall g^{2})B(g, \Theta)\di \varphi(x,y) \big].
\]
The second part is immediate in light of \eqref{frukkklk} and Definition \ref{debs}.
\end{proof}
To gauge the strength of special fan functionals, we need the `intuitionistic' fan functional, defined as follows (see e.g.\ \cite{kohlenbach2}*{\S3}):
\begin{princ}[$\MUC(\Phi)$]$
(\forall Y^{2}) (\forall f, g\leq_{1}1)(\overline{f}\Phi(Y)=\overline{g}\Phi(Y)\notag \di Y(f)=Y(g))$.  
\end{princ}
We have the following theorem, where $\RCA_{0}^{2}$ is just $\RCA_{0}$ formulated in the language of finite types (see \cite{kohlenbach2}*{\S2} for the exact definition).  
\begin{thm}\label{conske}
The systems $\RCA^{\omega}_{0}+(\exists \Theta)\SCF(\Theta)$ and $\P_{0}+\STP$ are conservative over $\RCA_{0}^{2}+\WKL$ for sentences in the latter's second-order language.  
\end{thm}    
\begin{proof}
For the first system, by \cite{samGH}*{Cor.\ 3.4}, a special fan functional can be defined in terms of $\Phi$ as in $\MUC(\Phi)$.  
By \cite{kohlenbach2}*{Prop.\ 3.15}, the first conservation result is now immediate.  For the second conservation result, let $\varphi$ be a (necessarily internal) sentence in the language of $\RCA_{0}^{2}+\WKL$.  If $\P_{0}+\STP\vdash \varphi$, then $\P_{0}\vdash (\exists^{\st}\Theta)\SCF(\Theta)\di \varphi$ by the second part of Theorem \ref{karwinkel}.  Applying Corollary \ref{consresultcor} to $\P_{0}\vdash (\forall^{\st}\Theta)(\SCF(\Theta)\di \varphi)$ yields $\RCA_{0}^{\omega}\vdash (\forall \Theta)(\SCF(\Theta)\di \varphi) $, and we are done.
\end{proof}
Thus, the existence of a special fan functional is not really stronger\footnote{However, combining a special fan functional with $(\exists^{2})$ already yields $\ATR_{0}$ (see \cites{dagsam, dagsamII, dagsamIII}).} than $\WKL_{0}$ \emph{in isolation}.  
The same holds for the following effective version of Heine's theorem.  
\begin{thm}[$\HEI_{\her}(t)$] 
For all $k^{0}$ and all $f:\R\di \R$, $g^{2}$, if for all $k'\leq t(g, k)(1)$
\[\textstyle
(\forall x\in [0,1]\cap t(g,k)(2))(\forall y\in [0,1])(|x-y|\leq\frac{1}{g(x, k')}\di |f(x)-f(y)|\leq\frac{1}{k'}),
\]
then $(\forall x, y\in [0,1])(|x-y|\leq\frac{1}{t(g,k)(3)}\di |f(x)-f(y)|\leq\frac1k)$.
\end{thm}  
We have the following corollary to Theorem \ref{floggen345}.  The same result for Dini's theorem is straightforward, but the formulation of the latter is (even more) messy.  
\begin{cor}
From the proof $\P_{0}\vdash\STP\asa \HEI_{\ns}$, terms $s,t$ can be extracted such that \textsf{\textup{E-PRA}}$^{\omega*}$ proves 
\be\label{thirdbase}
(\forall \Theta)(\SCF(\Theta)\di\HEI_{\her}(t(\Theta)) \wedge (\forall \Psi)(\HEI_{\her}(\Psi)\di \SCF(s(\Psi))).
\ee
Furthermore, $\P_{0}$ proves $(\exists^{\st}\Theta)\HEI_{\her}(\Theta)\di \HEI_{\ns}$.  
\end{cor}
\begin{proof}
In a nutshell, one applies $\CI$ to $\STP\asa\HEI_{\ns}$ using the normal form \eqref{frukkklk} for $\STP$, and the following normal form for $\HEI_{\ns}$: $(\forall^{\st}k', g^{2})(\exists^{\st}N, k'', v^{1^{*}})A(k', g, N, k'', v)$, where the internal formula $A$ is as follows:
\begin{align*}\textstyle
(\forall f:\R\di \R)\big[(\forall k\leq k'')(\forall x\in &[0,1]\cap v )\textstyle(\forall y\in [0,1])(|x-y|\leq \frac{1}{g(x, k')}\di |f(x)-f(y)|\leq\textstyle\frac1k)\\
&\di     \textstyle(\forall z,u\in [0,1] )(|u-z|\leq \frac{1}{N}\di |f(u)-f(z)|\leq\textstyle\frac{1}{k'})   \big].
\end{align*}
In particular, let $(\forall^{\st}g^{2})(\exists^{\st}w^{1^{*}})B(g,w)$ be the normal form of \eqref{frukkklk} and apply Corollary~\ref{consresultcor} to the following normal form:
\[
(\forall^{\st}k', g^{2}, \Theta)(\exists^{\st}N, k'', v^{1^{*}})\big[(\forall g^{2})B(g, \Theta)\di A(k', g, N, k'', v) \big].
\]
The second part is immediate in light of the normal form of $\HEI_{\ns}$ and the basic axioms from Definition \ref{debs}.
\end{proof}
In light of the final part of the theorem,  $(\exists \Phi)\HEI_{\her}(\Phi)$ is the Herbrandisation of $\HEI_{\ns}$.
Thus, we have so far observed as much as \emph{three} kinds of explicit equivalences to which nonstandard equivalences may give rise.  
\begin{enumerate}
\renewcommand{\theenumi}{\roman{enumi}}
\item Explicit equivalences between uniform versions of RM-theorems, like \eqref{frood}. 
\item The Herbrandisation of RM-equivalences like $\MTE_{\her}(s,t)$ involving \eqref{zion}.
\item Explicit equivalences between Herbrandisations of RM-theorems, like \eqref{thirdbase}.\label{krafi}  
\end{enumerate}
Despite their seemingly innocuous provenance, special fan functionals turn out to exhibit \emph{very strange behaviour}, as follows.
In joint work with Dag Normann (\cites{dagsam, dagsamII, dagsamIII}), we have shown that any special fan functional is \emph{not} computable (in the sense of Kleene's S1-S9 as in \cite{longmann}*{\S5.1}) in any type two functionals.  
In turn, the non-computability result between $(\mu^{2})$ and special fan functionals translates into the non-implication $\paai\not\di \STP$ over $\P$; the same holds for much stronger fragments of \emph{Transfer}.  
Similarly, the following basic theorem, which is just \eqref{hro} without `st', is only provable in \emph{full} second-order arithmetic (in terms of comprehension axioms), while having the first-order strength of $\WKL_{0}$:
\be\tag{$\textup{\textsf{HBU}}$}
(\forall G^{2})(\exists w^{1^{*}})(\forall f\leq_{1}1)(\exists i<|w|)\big(  f\in \big[\overline{w(i)}G(w(i))\big]  \big).
\ee
Furthermore, there are many functionals similar to special fan functionals:  the principle $\STP$ is equivalent to \eqref{fanns}, which is just $\WKL^{\st}$ with the outermost `st' dropped, and we can obtain a similarly modified version for (almost) any theorem in the RM zoo, inside $\P_{0}+\STP+\paai$.  From this modified version, a normal form like \eqref{frukkklk} can be obtained, giving rise to a functional similar to a special fan functional, though generally not computing the latter.  As an example, $\textsf{LMP}$ in \cite{pimpson} is the nonstandard version of $\textsf{WWKL}$ and has the same form as \eqref{fanns}, up to minor cosmetic details.  Thus, $\textsf{LMP}$ gives rise to a weak version of special fan functionals, which is however still not (Kleene S1-S9) computable in any type two functional.  In this way, the original RM zoo seems to give rise to a new `higher-order' RM zoo.         

\smallskip

An explanation of the strange behaviour of special fan functionals is as follows: Karel Hrbacek has communicated to the author a proof that Nelson's axiom \emph{Transfer} does not imply \emph{Standard Part} over $\ZFC$.  
Similarly, as mentioned in the previous paragraph, one can prove that $\paai\not\di\STP$ over $\P$, and this non-implication actually derives straightforwardly from the fact that special fan functionals cannot be expressed in terms of $(\mu^{2})$ via a term from G\"odel's $\textsf{T}$ (which follows from the results in \cite{dagsam, dagsamII}).   
In other words, the fact that Nelson's axiom \emph{Standard Part} is not implied by \emph{Transfer} (and the same for fragments like $\STP$ and $\paai$) 
translates to the fact that special fan functionals cannot be computed by $(\mu^{2})$, and vice versa.  Similar results for RM zoo theorems give rise to a plethora of new functionals, populating a new `higher-order' RM zoo.  
We discuss a related functional in the following remark.
\begin{rem}\label{snark}\rm 
We recall that in the proof of Theorem \ref{floggen345}, we strengthened the antecedent of \eqref{flacka}, i.e.\ the latter is a weakening of $\STP\di \HEI_{\ns}$.
Thus, we expect that $h$ as in $\textsf{HBL}(h)$ is in general weaker than (certain) special fan functionals.  
As studied in detail in \cite{dagsamV}, this `great expectation' is correct.  In a nutshell, the aforementioned functional $h$ \emph{does not} compute a finite sub-cover in the way that special fan functionals do.  
However, functionals like $h$ do compute \emph{an upper bound} for the length of such a finite sub-cover.  This observation arose from the study of \emph{Pincherle's theorem}, as discussed in detail in \cite{dagsamV}.  
\end{rem}
\subsubsection{Constructive Dini's theorem}
In this section, we study the \emph{constructive version} of Dini's theorem as formulated in \cite{bridi}*{Theorem 4}.  
The latter adds the extra data that a very specific sequence $x_{n}$ (defined in terms of $f_{n}$ and $f$) has to have a convergent subsequence if we are to conclude uniform convergence.  

\smallskip

A natural question is whether we can obtain the constructive version of Dini's theorem from a suitable nonstandard version, 
like we did for nonstandard compactness and totally boundedness in Section \ref{cont1}.   
To this end, let $\DINI_{\ns}'$ be $\DINI_{\ns}$ limited to \emph{uniformly} continuous functions and 
with the following consequent:
\begin{align*}
(\forall^{\st}x_{(\cdot)}, g)\Big[
&\big[(\forall n)(g(n)<(g(n+1)))\wedge (\forall N\in \Omega)(\forall^{\st} n)\big(|f(x_{n})-f_{n}(x_{n})|\approx  \|f-f_{n}\|_{N}  \big) \\
&\wedge (\forall K,M\in \Omega)(x_{g(M)}\approx x_{g(K)})\big] \di (\forall x\in [0,1], N\in \Omega)[f_{N}(x)\approx f(x)]  \Big]
\end{align*}
Clearly, our nonstandard version $\DINI_{\ns}'$ is inspired by \cite{bridi}*{Theorem 4}, but is perhaps more intuitive:  in order to guarantee uniform convergence, the sequence $x_{(\cdot)}$ has to be such that $|f(x_{n})-f_{n}(x_{n})|$ witnesses the sup norm $\|f-f_{n}\|_{N}$ and such that $x_{g(\cdot)}$ is a convergent subsequence.  
Note that the limitation to uniformly (nonstandard) continuous functions is necessary for the step \eqref{hagg} in the proof, and also for the `general' existence of the supremum norm.    
\begin{thm}\label{varou2}
The system $\textup{\textsf{E-PRA}}^{\omega*}_{\st} $ proves $\DINI_{\ns}'$.  
\end{thm}
\begin{proof}
Let $f, f_{n}, x_{(\cdot)}, g$ be as in $\DINI_{\ns}'$ and fix $y_{0}\in [0,1]$ and $N_{0}\in \Omega$.  
Applying overspill\footnote{Note that `$\gtrapprox$' and `$\approx$' in \eqref{flig} can also be replaced by a universal formula.} to the fact that for all standard $m$, we have
\be\label{flig}
g(m)\leq N_{0} \wedge f(x_{g(m)})-f_{g(m)}(x_{g(m)})\approx \|f-f_{g(m)}\|_{N_{0}} \gtrapprox  f(y_{0})-f_{g(m)}(y_{0}),
\ee
yields that \eqref{flig} holds for $m\leq M_{0}\in \Omega$.  Now consider the following inequalities:
\begin{align}
0\leq f(y_{0})-f_{N_{0}}(y_{0})\leq f(y_{0})-f_{g(M_{0})}(y_{0})&\lessapprox \|f-f_{g(M_{0})}\|_{N_{0}} \label{hagg}\\
&\approx f(x_{g(M_{0})})-f_{g(M_{0})}(x_{g(M_{0})}).\label{hagg2}
\end{align}
Note that the final two steps follow from \eqref{flig} for $m=M_{0}$.  By $\Omega\textsf{-CA}$, there is a standard real $a$ such that $(\forall N\in \Omega)(a\approx x_{N})$.    
Clearly, we have by continuity that $f(x_{g(M_{0})})\approx f(a)$ and $(\forall^{\st}n)(f_{n}(x_{g(M_{0})})\approx f_{n}(a))$.  Applying overspill to the latter (technically with `$\approx$' resolved), we obtain some nonstandard $N_{1}$ such that   
$(\forall n\leq N_{1})(f_{n}(x_{g(M_{0})})\approx f_{n}(a))$.  Now, clearly $f_{N_{1}}(x_{g(M_{0})})\approx f_{N_{1}}(a)\approx f(a)$ by pointwise nonstandard convergence.  
However, by the monotone behaviour of $f_{n}$ and the continuity of $f$, we have for $m\geq N_{1}$ that $f_{N_{1}}(x_{g(M_{0})})\leq f_{m}(x_{g(M_{0})})\leq f(x_{g(M_{0})})\approx f(a)$.  
The previous implies that $(\forall m\in \Omega)(f_{m}(x_{g(M_{0})}))\approx f(a)$.  Hence, the number in \eqref{hagg2} is an infinitesimal, i.e.\ $f(x_{g(M_{0})})\approx f(a)\approx f_{g(M_{0})}(x_{g(M_{0})})$.  
Then \eqref{hagg} implies $f(y_{0})\approx f_{N_{0}}(y_{0})$ and we are done.  
\end{proof}
It is clear that for the previous proof, we drew inspiration from the final part of the proof of \cite{bridi}*{Theorem 4}.  
However, our proof is arguably shorter, as we do not need \cite{bridi}*{Prop.\ 3}.  
It is now a tedious but straightforward derivation to obtain the effective version of $\DINI'_{\ns}$, which is highly similar to \cite{bridi}*{Theorem 4}.

\smallskip

Finally, the proof of Theorem \ref{varou2} also goes through in the constructive system $\H$, and there appears to be a `recursive' version of Dini's theorem (\cite{urbankamo}) in which compactness is witnessed by a recursive function.    

\subsection{Stone-Weierstrass theorem}\label{stoner}
In this section, we study the \emph{Stone-Weierstrass theorem} $\SWT$ (see \cite{rudin}*{7.32}).  
In this context, $S$ is called a \emph{separating set} for $C$ if 
\be\label{snavel}
(\forall x, y \in C)(\exists f\in S)(x\ne y\di f(x)\ne f(y)
\ee
\begin{thm}[$\SWT$]
An algebra\footnote{We require an algebra to contain the constant functions so as to satisfy Rudin's non-vanishing criterion from \cite{rudin}*{Theorem 7.32}.} of continuous functions \emph{with a separating set for the compact domain} is dense in the continuous functions on the domain.
\end{thm}
Now, there are constructive versions of $\SWT$ (see \cite{bridge1}*{p.\ 106}, \cite{jay}*{\S4}, and \cite{browner}*{\S4}), but they all involve the same complicated notion of `constructive separating set', and it is a natural question (in light of our treatment of compactness and constructive Dini's theorem) whether there is a more elegant \emph{nonstandard} notion of separating set which gives rise to the notion of constructive separating set after applying $\CI$.  We shall provide a positive answer to this question in this section.   

\smallskip
 
First of all, we introduce the following definition, arguably quite similar to \eqref{snavel}.  
\bdefi[nonstandard-separating set]
The set $S$ of functions $f:C\di D$ is a \emph{nonstandard-separating} set for $C$ if 
\be\label{fredi}
(\forall^{\st}x, y\in C)(\exists^{\st}f\in S)(x\not\approx y\di f(x)\not\approx f(y)). 
\ee
\edefi
We shall now obtain a normal form for \eqref{fredi}.  
To this end, note that \textsf{SWT} deals with separating sets for \emph{compact} domains for \emph{continuous} functions;  thus, nonstandard compactness (as studied in the previous section) seems particularly useful to remove the outermost `st' in \eqref{fredi}, as we prove in the next theorem.  
The latter can be proved in $\H$ or $\P_{0}$.  
\begin{thm}
If $C$ is nonstandard compact and has a nonstandard-separating set $S$ of nonstandard continuous functions, then there is standard $d$ with
\be\label{Stronghold}\textstyle
(\forall^{\st}k)(\forall^{\st} x, y\in C)(\exists^{\st}f\in S)\big[|x- y|\geq \frac{1}{k}\di |f(x)- f(y)|\geq \frac1{d(k)}\big]. 
\ee
\end{thm}
\begin{proof}
Let $C, S$ be as in the theorem and consider \eqref{fredi}.  Resolving `$\approx$', we obtain
\be\label{okil}\textstyle
(\forall^{\st}k)(\forall^{\st}x, y\in C)(\exists^{\st}f\in S, N)(|x- y|\geq \frac{1}{k}\di |f(x)- f(y)|\geq \frac{1}{N}), 
\ee
with all standard quantifiers pushed forward as much as possible (which can also be done inside $\H$ by following the proof of Corollary \ref{cordejedi}).  As $C$ is nonstandard compact and $f$ as in \eqref{okil} is nonstandard continuous, we obtain 
\be\label{henfrie}\textstyle
(\forall^{\st}k)(\forall x, y\in C)(\exists^{\st}f\in S, N)[|x- y|\geq \frac{1}{k}\di |f(x)- f(y)|\geq \frac{1}{N}].   
\ee
Indeed, for every $x, y\in C$ there are standard $x', y'$ such that $x'\approx x$ and $y'\approx x$ by nonstandard compactness.  
Nonstandard continuity yields $f(x)\approx f(x')$ and $f(y)\approx f(y')$ for the $f$ from \eqref{okil}, and \eqref{henfrie} now follows.  

\smallskip

Applying idealisation $\textsf{I}$ (or $\textsf{NCR}$ in the case of $\H$) to \eqref{henfrie}, we obtain
\[\textstyle
(\forall^{\st}k)(\exists^{\st} f', N')(\forall x, y\in C)(\exists f\in f', N\in N')\big(f\in S \wedge \big[ |x- y|\geq \frac{1}{k}\di |f(x)- f(y)|\geq \frac{1}{N}\big]\big).   
\]
Now apply $\HAC_{\INT}$ to obtain standard $\Phi$ such that $f', N' \in \Phi(k)$.  Let $d(k)$ be the maximum of all entries for $N'$ and note that 
\[\textstyle
(\forall^{\st}k)(\forall x, y\in C)(\exists^{\st} f )\big(f\in S \wedge \big[ |x- y|\geq \frac{1}{k}\di |f(x)- f(y)|\geq \frac{1}{d(k)}\big]\big),   
\]
by ignoring the components for $f'$ and $f$ of $\Phi$.
\end{proof}
As is clear from the proof, since we are working over a nonstandard compact set, the distinction between pointwise and uniform continuity is not that relevant.  

\smallskip

Secondly, as we have pointed out before, $\HAC_{\INT}$ only provides witnessing functionals for type zero existential quantifiers (and similarly monotone objects).  
Hence, it seems impossible to obtain a witnessing functional for $(\exists f\in S)$ as in \eqref{Stronghold}.  Perhaps surprisingly, we \emph{can} obtain a modulus of continuity for this function $f$.  
\begin{cor}\label{dorko}
If $C$ is nonstandard compact with a nonstandard-separating set $S$ of nonstandard continuous functions, then there is standard $d, e$ with
\[\textstyle
(\forall^{\st}k, k')(\forall^{\st} x, y\in C)(\exists^{\st}f\in S)\big( C(f, e(k'), k')\wedge\big[   |x- y|\geq \frac{1}{k}\di |f(x)- f(y)|\geq \frac1{d(k)}\big]\big), 
\]
where $C(f, e(k'), k')\equiv (\forall \textstyle x', y' \in C)(|x'-y'|<\frac{1}{e(k')} \di |f(x')-f(y')|\leq\frac{1}{k'})$.
\end{cor}
\begin{proof}
Similar to the proof of the theorem, but with the addition of `$f$ is uniformly nonstandard continuous on $C$' inside the square brackets in \eqref{Stronghold}.  
Note that uniform continuity follows from usual continuity due to the nonstandard compactness of $C$.  
After resolving `$\approx$', the new version of formula \eqref{henfrie} becomes 
\[\textstyle
(\forall^{\st}k, k')(\forall x, y\in C)(\exists^{\st}f\in S, N, N')\big(  C(f, N', k') \wedge   [|x- y|\geq \frac{1}{k}\di |f(x)- f(y)|\geq \frac{1}{N}]\big), 
\]
and the rest of the proof is now straightforward.
\end{proof}
In short, Corollary \ref{dorko} tells us that the images of $x$ and $y$ are not only `effectively separated' by $d$ as in \eqref{Stronghold}, 
but also that we effectively know the continuity of $f$ around $x$ and $y$ thanks to $e$.   
Furthermore, it is easy to guarantee that for certain $x, y\in C$, $f(x)=0$ and $f(y)=1$ by defining $g(\cdot):=\frac{f(\cdot)-f(x_{0})}{f(y_{0})-f(x_{0})}$, where $x_{0}, y_{0}$ are the standard parts of $x, y$ provided by nonstandard compactness.  

\smallskip

We now invite the reader to compare the complicated definition of a \emph{constructive} separating set as in \cite{bridge1}*{Def.\ 5.13. p.\ 106} to Corollary~\ref{dorko}.  
Clearly, the functional `$e$' from Corollary \ref{dorko} plays the role of $\delta$ in \cite{bridge1}*{Def.\ 5.13. p.\ 106}, 
and we observe that the normal form as in Corollary \ref{dorko} of the notion of nonstandard-separating set gives rise to the constructive notion of separating set.   

\smallskip

In conclusion, like in the case of compactness, the constructive version of `separating set' falls out of the nonstandard one, again with surprisingly little effort.  
Furthermore, if we formulate \emph{nonstandard density} as $(\forall^{\st} f\in C(X))(\exists g\in G)(\forall x\in X)(f(x)\approx g(x))$, then the corresponding nonstandard version of $\SWT$ gives rise to the constructive version of $\SWT$ as in \cite{bridge1}*{p.\ 106, 5.14}.    

\subsection{The strongest Big Five}\label{suskeenwiske}
In this section, we study equivalences relating to $\ATR_{0}$ and $\FIVE$, the strongest Big Five systems from RM.  
The associated results show that the template $\CI$ also works for the fourth and fifth Big Five system.  

\smallskip

We shall work with the \emph{Suslin functional} $(S^{2})$, the functional version of $\FIVE$.  
\be\label{suske}
(\exists S^{2})(\forall f^{1})\big[   S(f)=_{0} 0 \asa (\exists g^{1})(\forall x^{0}) (f(\overline{g}x)\ne 0)\big]. \tag{$S^{2}$}
\ee
Feferman has introduced the following version of the Suslin functional (see e.g.\ \cite{avi2}):
\be\label{suske2}
(\exists \mu_{1}^{1\di 1})\big[(\forall f^{1})\big(    (\exists g^{1})(\forall x^{0}) (f(\overline{g}x)\ne 0)\di (\forall x^{0}) (f(\overline{\mu_{1}(f)}x)\ne 0)\big)\big], \tag{$\mu_{1}$}
\ee
where the formula in square brackets is denoted $\MUO(\mu_{1})$.  We shall also require another instance of \emph{Transfer}, as follows:
\be\tag{$\Paai$}
(\forall^{\st} f^{1})\big[    (\exists g^{1})(\forall x^{0}) (f(\overline{g}x)\ne 0)\di  (\exists^{\st} g^{1})(\forall^{\st} x^{0}) (f(\overline{g}x)\ne 0)\big].
\ee
We shall obtain an effective version of the equivalence proved in \cite{yamayamaharehare}*{Theorem~4.4}.  
The relevant (non-uniform) principle pertaining to the latter is $\PST$, i.e.\ the statement that \emph{every tree with uncountably many paths has a non-empty perfect subtree}.   
The latter has the following nonstandard and effective versions. 
\begin{thm}[$\PST_{\ns}$]For all standard trees $T^{1}$, there is standard $P^{1}$ such that 
\[
(\forall f_{(\cdot)}^{0\di 1})(\exists f\in T)(\forall n)(f_{n}\ne_{1} f) \di \textup{$P$ is a non-empty perfect subtree of $T$} \big],
\]
\end{thm}
\begin{thm}[$\PST_{\ef}(t)$]
For all trees $T^{1}$, we have
\[
(\forall f_{(\cdot)}^{0\di 1})(\exists f\in T)(\forall n)(f_{n}\ne_{1} f) \di \textup{$t(T)$ is a non-empty perfect subtree of $T$} \big].
\]
\end{thm}
Note that the property of being a non-empty perfect subtree consitutes an arithmetical formula as in \cite{simpson2}*{Definition V.4.1} and \eqref{nonrobustbelow} below.  
As a technicality, we require that $P$ as in the previous two principles consists of a pair $(P', p')$ such that $P'$ is a perfect subtree of $T$ such that $p'\in P'$.      
We have the following theorem.
\begin{thm}\label{sef2334}
From the proof of $\PST_{\ns}\asa \Paai$ in $\P_{0} $, two terms $s, u$ can be extracted such that $\textup{\textsf{E-PRA}}^{\omega*}$ proves:
\be\label{frood3}
(\forall \mu_{1})\big[\textsf{\MUO}(\mu_{1})\di \PST_{\ef}(s(\mu_{1})) \big] \wedge (\forall t^{1\di 1})\big[ \PST_{\ef}(t)\di  \MUO(u(t))  \big].
\ee
\end{thm}
\begin{proof}  
We first establish the `easy' implication $\Paai\di\PST_{\ns}$.  Clearly, $\Paai$ implies $\paai$ and the former implies:  
\be\label{julliette}
(\forall^{\st} f^{1})(\exists^{\st} g^{1})\big[    (\exists g^{1})(\forall x^{0}) (f(\overline{g}x)\ne 0)\di  (\forall x^{0}) (f(\overline{g}x)\ne 0)\big].
\ee
Applying $\HAC_{\INT}$, there is standard $\Phi$ such that $(\exists g^{1}\in \Phi(f))$.  As noted in the proof of Theorem \ref{sef}, $\paai$ implies $(\exists^{2})^{\st}$ and the latter can be used to 
check which of the possible witnesses $g\in \Phi(f)$ is such that $(\forall^{\st} x^{0}) (f(\overline{g}x)\ne 0)$.  Hence, there is standard $\Psi$ such that 
\[
(\forall^{\st} f^{1})\big[    (\exists g^{1})(\forall x^{0}) (f(\overline{g}x)\ne 0)\di  (\forall^{\st} x^{0}) (f(\overline{\Psi(f)}x)\ne 0)\big].
\]
Since $\paai$ is available, this implies
\[
(\forall^{\st} f^{1})\big[    (\exists^{\st} g^{1})(\forall^{\st} x^{0}) (f(\overline{g}x)\ne 0)\di  (\forall^{\st} x^{0}) (f(\overline{\Psi(f)}x)\ne 0)\big],
\]
which is exactly $(\mu_{1})^{\st}$.  The latter implies $\FIVE$ relative to `\st', and we obtain $\PST^{\st}$ by performing the proof in \cite{simpson2}*{V.5.5} relative to `st'.  As the property of being a perfect subtree is arithmetical, we can use $\paai$ to drop all instances of `\st' in this property.  Hence, we obtain
\[
(\forall^{\st}T)(\exists^{\st}P )\big[(\forall^{\st} f_{n})(\exists^{\st} f\in P)(\forall^{\st} n)(f_{n}\approx_{1} f) \di \textup{$P$ is a non-empty perfect subtree of $T$} \big].
\]  
Thanks to $\Paai$, we now obtain:
\[
(\forall^{\st}T)(\exists^{\st}P )\big[(\forall f_{n})(\exists f\in P)(\forall n)(f_{n}=_{1} f) \di \textup{$P$ is a non-empty perfect subtree of $T$} \big],
\]  
which is exactly $\PST_{\ns}$.  The implication $\Paai\di \PST_{\ns}$ is easily brought into the normal form in light of \eqref{julliette}.  
Applying Corollary \ref{consresultcor} now provides a witnessing term, from which the right witness can be selected using $(\exists^{2})$, as `being a perfect subtree' is arithmetical.      
    
\smallskip

Next, we establish the remaining implication $\PST_{\ns}\di \Paai$.  First, assume $\paai$ and fix standard $f$ such that $(\exists g^{1})(\forall x^{0})f(\overline{g}x)=0$ and define: 
\be\label{treeke}
\sigma\in T_{0}\asa (\forall i<|\sigma|/2)(f(\sigma(0)*\sigma(2)*\dots \sigma(2i))=0).
\ee
By assumption, the tree $T_{0}$ has uncountably many paths (in the internal sense of the antecedent of $\PST_{\ns}$), and we may conclude the existence of a standard non-empty subtree $P_{0}=(P_{0}', p_{0}')$.  
The definition of perfect subtree for $P_{0}'$ is
\be\label{nonrobustbelow}
(\forall \sigma \in P_{0}')(\exists \tau_{1}, \tau_{2}\in P_{0}')[\sigma\subseteq \tau_{1}\wedge \sigma\subseteq \tau_{2}\wedge \tau_{1}, \tau_{2} \textup{ are incompatible}],
\ee
where `incompatible' means that neither sequence is an extension of the other.  Note that the formula in square brackets in \eqref{nonrobustbelow} is quantifier-free.  
We obtain
\be\label{klonk}
(\forall^{\st} \sigma \in P_{0}')(\exists^{\st} \tau_{1}, \tau_{2}\in P_{0}')(\sigma\subseteq \tau_{1}\wedge \sigma\subseteq \tau_{2}\wedge \tau_{1}, \tau_{2} \textup{ are incompatible}),
\ee
using $\paai$ due to the standardness of $P_{0}'$.  Since $p_{0}'\in P_{0}'$ and $p_{0}'$ is standard, \eqref{klonk} implies that there is a standard sequence $g_{0}$ such that $(\forall^{\st} x^{0})(\overline{g_{0}}x\in P_{0}')$.
By the definition of $T_{0}$, we obtain that $(\exists^{\st} g^{1})(\forall^{\st} x^{0})f(\overline{g}x)=0$.  
Hence, we have obtained $\Paai$ assuming $\paai$, and we now show that the latter is also implied by $\PST_{\ns}$.  To this end, fix standard $h^{1}$ and suppose $(\exists n)h(n)\ne0$.  Define $f(\sigma)=0$ iff $h(\sigma(0))\ne0$ and note that 
$(\exists g^{1})(\forall x^{0})f(\overline{g}x)=0$.  We again consider $T_{0}$ and obtain a standard perfect subtree $P_{0}'$ of $T_{0}$ such that standard $p_{0}'\in P_{0}'$.  The latter implies $(\exists^{\st}n)h(n)\ne 0$ by the definition of $f$, and $\paai$ follows.              

\smallskip   

Finally, the implication $ \PST_{\ns}\di [\Paai~ \wedge ~\paai]$ is easily brought into the normal form in light of \eqref{noniesimpel} and \eqref{julliette}.   
Applying Corollary \ref{consresultcor} now provides a witness to $(\mu^{2})$ (as in the proof of Theorem \ref{sef}) and a finite sequence of possible witnesses for $(\mu_{1})$, and the right one can be selected by the first witness.  
\end{proof}
By the previous, $(\mu_{1})$ is the Herbrandisation of $\Paai$, and it is straightforward to obtain the Herbrandisation of $\Paai\asa \PST_{\ns}$ and obtain the associated `equivalence' as in Corollary~\ref{herken}.  
\begin{cor}
In $\RCAo$, $(\mu_{1})\asa (\exists t^{1\di 1})\PST_{\ef}(t)$ and this equivalence is explicit.  
\end{cor}
In light of the intimate connection between theorems concerning perfect kernels of trees and the Cantor-Bendixson theorem for Baire space (see \cite{simpson2}*{IV.1}), a version of Theorem \ref{sef2334} for the former can be obtained in a straightforward way.  
Another \emph{more mathematical} statement which can be treated along the same lines is \emph{every countable Abelian group is a direct sum of a divisible and a reduced group}.  The latter is called $\DIV$ and equivalent to $\FIVE$ by \cite{simpson2}*{VI.4.1}.  By the proof of the latter, the reverse implication is straightforward;  we shall study $\DIV\di \FIVE$.  

\smallskip

To this end, let $\DIV(G, D, E)$ be the statement that the countable Abelian group $G$ satisfies $G=D\oplus E $, where $D$ is a divisible group and $E$ a reduced group.  
The nonstandard version of $\DIV$ is as follows:
\be\tag{$\DIV_{\ns}$}
(\forall^{\st}G)(\exists^{\st} D, d, E)\big[\DIV(G, D, E)\wedge D\ne\{0_{G}\}\di d\in D\big],
\ee
where we used the same technicality as for $\PST_{\ns}$.  The effective version is:
\be\tag{$\DIV_{\ef}(t)$}
(\forall G)\big[\DIV(G, t(G)(1), t(G)(2))\wedge t(G)(1)\ne\{0_{G}\}\di t(G)(3)\in t(G)(1)\big].
\ee
We have the following (immediate) corollary.  
\begin{cor}\label{sef23345}
From the proof of $\DIV_{\ns}\di \Paai$ in $\P_{0} $, a term $u$ can be extracted such that $\textup{\textsf{E-PRA}}^{\omega*}$ proves:
\be\label{frood3555}
(\forall t^{1\di 1})\big[ \DIV_{\ef}(t)\di  \MUO(u(t))  \big].
\ee
\end{cor}
\begin{proof}
The proof of the nonstandard implication is similar to that of the theorem.  Indeed, define $T_{0}$ as in \eqref{treeke} with the same assumptions on $f$.  
Define the Abelian group $G_{0}$ with generators $x_{\tau}$, $\tau \in T_{0}$ and relations $px_{\tau}=x_{\rho}$, $\tau =\rho*\langle i \rangle$, and $x_{\langle\rangle}=0$ (This definition is taken from \cite{simpson2}*{p.\ 231}).     
Since $f$ is standard, so is $T_{0}$ and $G_{0}$, and since $(\exists g^{1})(\forall x^{0})f(\overline{g}x)=0$, the divisible group $D$ provided by $\DIV_{\ns}$ is non-trivial.  Hence, there is standard $d \in D$, and using $\paai$ on the definition of divisible subgroup as in the proof of the theorem, we obtain a \emph{standard} path through $T_{0}$ and hence $(\exists^{\st} g^{1})(\forall x^{0})f(\overline{g}x)=0$, and $\Paai$ follows.  
To obtain \eqref{frood3555}, proceed as in the proof of the theorem.  
\end{proof}
It is straightforward to obtain the Herbrandisation of $\Paai\asa \DIV_{\ns}$ and obtain the `equivalence' between these two as in Corollary~\ref{herken}.  
We trust the reader to discern a template for equivalences related to $\FIVE$ from the above.  

\smallskip

Finally, we point out a uniform version of $\PST$ \emph{at the level of }$\ATR_{0}$, as follows:   
\begin{thm}[$\PST_{\ef}'(t)$]  For all $ T^{1}, g^{1\di 1}$, we have
\[
(\forall f_{(\cdot)})(\forall n)\big[g(f_{(\cdot)})\in P\wedge f_{n}\ne_{1} g(f_{(\cdot)})\big] \di \textup{$t(T, g)$ is a non-empty perfect subtree of $T$}.
\]
\end{thm}
Clearly, $\PST_{\ef}'(t)$ is just $\PST_{\ef}(t)$ with a witnessing functional for the antecedent.  It is straightforward to derive $(\exists t^{(1\di 1)\di 1})\PST_{\ef}'(t)$ in $\RCAo+\QFAC+(\exists^{2})+\ATR_{0}$.  The nonstandard version $\PST_{\ns}$ can be weakened similarly by introducing such \emph{standard} $g$, and the associated term extraction result is straightforward.        

\subsection{Algebraic theorems}\label{algea}
In this section, we apply the template $\CI$ to some theorems from the RM of $\ACA_{0}$ which are \emph{algebraic} in nature, like for instance \cite{simpson2}*{I.9.3.5-8} or \cite{simpson2}*{I.10.3.9-11}.  
The vague and heuristic notion of \emph{algebraic} theorem is meant to capture those (internal) theorems $T$ of the form $T\equiv (\forall X^{1})(\exists Y^{1})\varphi(X, Y)$ where $\varphi$ is arithmetical.  We shall obtain two kinds of explicit equivalences, one using the `standard extensionality trick' from Section \ref{X}.    

\smallskip
  
We study $\TOR$, which is the statement that \emph{a countable abelian group has a subgroup consisting of the torsion elements} (see \cite{simpson2}*{III.6}), and show how other theorems can be treated analogously in Template~\ref{effin}.  In other words, we show that $\CI$ applies to a large group of theorems from the third Big Five category.

\smallskip

First of all, in contrast to notions from analysis like continuity, there is no obvious nonstandard version of the notion of torsion subgroup.  Nonetheless, 
there \emph{is} an `obvious' nonstandard version of $\textsf{TOR}$ itself, as we discuss next.  
Now, $\TOR$ has the form $(\forall G^{1})(\exists T^{1})\TOR(G, T)$, where $\TOR(G, T)$ expresses that $T$ is a torsion subgroup of $G$, i.e.\ we have that
\be\label{turfff}
T\subseteq G \wedge (\forall g\in G)[g\in T \asa (\exists n)(ng=0_{G})].
\ee
In $\IST$ and using the axiom \emph{Transfer}, $\TOR$ implies $(\forall^{\st} G^{1})(\exists^{\st} T^{1})\TOR(G, T)$, our cherished normal form.  
We can however obtain more information:  the formula \eqref{turfff} implies that $(\forall^{\st} g\in G)[g\in T \di (\exists^{\st} n)(n\times_{G}g)=0_{G}]$, for standard $G$ and $T$, again using \emph{Transfer}.  
Applying the axiom of choice (relative to `st'), $\TOR$ implies the following normal form, which we shall refer to as $\TOR_{\ns}$.  
\begin{align}
(\forall^{\st}G^{1})(\exists^{\st} h^{1}, T^{1}\subseteq G)\big[ (\forall g\in G, m)[m&g=0_{G} \di g\in T] \notag\\
&\wedge[ g\in T\di h(g)g= 0_{G}]\big].\label{forba}
\end{align}
In conclusion, $\TOR_{\ns}$ is just $\TOR$ brought into the normal form \emph{with a functional $h$ witnessing the `torsion-group-ness' of $T$} as in \eqref{forba}.   

\smallskip

Secondly, we argue that \eqref{forba} is natural from the point of view of Brattka's framework from \cite{bratger}.  Indeed, in the latter, mathematical theorems are interpreted as operations mapping certain input data to output data in a computable fashion.  
Now, such operations may be assumed to be \emph{standard} in \IST;  
indeed, for a theorem of the form $(\forall x)(\exists y)\varphi(x, y)$, it is the `explicit' version $(\forall x)\varphi(x, \Phi(x))$ which is studied and classified according to the computational content of $\Phi$ in \cite{bratger};  assuming Transfer from \IST~(while the much weaker $\textsf{PF-TP}_{\forall}$ from \cite{bennosam} suffices), the sentence $(\exists \Phi)(\forall x)\varphi(x, \Phi(x))$ immediately yields $(\exists^{\st} \Phi)(\forall x)\varphi(x, \Phi(x))$.  Since standard inputs yield standard outputs for standard operations, we obtain $(\forall^{\st} x)(\exists^{\st} y)\varphi(x, y)$.  
In other words, the nonstandard version \eqref{forba} naturally emerges from Brattka's framework as in \cite{bratger}.   

\smallskip

In light of the above considerations, $\TOR_{\ns}$ is natural and we define $\TOR_{\ef}(t)$ as the statement that $(\forall G^{1})A(G, t(G)(1), t(G)(2))$ where $A(G, T, h)$ is the formula in big square brackets in \eqref{forba}.  
\begin{thm}\label{floggen346}
From the proof of $\paai \asa  \TOR_{\ns}$ in $\P_{0}$, terms $ s, u$ can be extracted such that 
\be\label{daric}
(\forall \mu^{2})\big[\textsf{\MU}(\mu)\di \TOR_{\ef}(s(\mu)) \big] \wedge (\forall t^{1\di 1})\big[ \TOR_{\ef}(t)\di  \MU(u(t))  \big].
\ee
\end{thm}
\begin{proof}
To prove $\paai\di \TOR_{\ns}$, note that the antecedent implies $(\exists^{2})^{\st}$ (as in the proof of Theorem \ref{sef}).  Following the proof of \cite{simpson2}*{III.6.2} relative to `st', we obtain $\TOR^{\st}$.  
Alternatively, the latter can be obtained by defining the subgroup $T_{M}$ as $\{g\in G : (\exists n\leq M)(ng=0_{G}) \}$, and applying $\Omega$\textsf{-CA} after observing $(\forall N,M\in \Omega)(T_{M}\approx_{1}T_{N})$.    
From $\TOR^{\st}$, \eqref{forba} now follows using $\paai$ and $\HAC_{\INT}$.  

\smallskip

To prove $\TOR_{\ns} \di \paai$, suppose the consequent is false, i.e.\ there is standard $h^{1}$ such that $(\forall^{\st}n)h(n)=0\wedge (\exists m_{0})h(m_{0})\ne0$. 
Let $p_{0}$ be a prime number such that $(\exists i\leq p_{0})h(i)\ne 0$ and consider the \emph{multiplicative group of integers modulo $p_{0}$}, usually denoted $\mathbb{Z}/p_{0}\mathbb{Z}$.  
This group can be defined in terms of $h$ (only) and is therefore \emph{standard}.  Applying $\TOR_{\ns}$, every standard element in $\mathbb{Z}/p_{0}\mathbb{Z}$ is in the torsion subgroup and becomes zero after being multiplied standardly many times.
This contradiction yields  the required nonstandard equivalence, and the rest of the theorem follows by applying $\CI$ as in the proof of Theorem~\ref{sef}.  
\end{proof}   
It is straightforward to obtain the Herbrandisation of $\paai\asa \TOR_{\ns}$ and obtain the `equivalence' between these two as in Corollary~\ref{herken}.  
\begin{cor}
The explicit equiv.\ $(\mu^{2})\asa (\exists t^{1\di 1})\TOR_{\ef}(t)$ is provable in $\RCAo$.  
\end{cor}
Overall, the proof of $\TOR_{\ns}\asa \paai$ seems simpler than the proof of $\TOR\asa \ACA_{0}$ in \cite{simpson2}*{III.6.2}.  On top of that, we obtain an explicit equivalence in a purely algorithmic fashion.  
Again, the RM of Nonstandard Analysis seems simpler \emph{and provides more explicit information}.  

\smallskip

Next, we use the `standard extensionality trick' from Corollary \ref{dergggg2} to obtain an even simpler explicit implication.   
\begin{cor}\label{floggen3467}
From the proof in $\P_{0}$ of 
\be\label{bef}
(\exists^{\st} \Phi)\big[(\forall^{\st} G^{1})\TOR(G, \Phi(G))\wedge \Phi \textup{ is standard extensional}\big] \di \paai,
\ee
a term $u$ can be extracted such that 
\be\label{daric2}
 (\forall t^{1\di 1})\big[(\forall G^{1}) \TOR(G, t(G))\di  \MU(u(t, \Xi))  \big].
\ee
where $\Xi$ is an extensionality functional for $t$.  
\end{cor}
\begin{proof}
To prove \eqref{bef}, assume the latter's antecedent and suppose $\paai$ is false.  Now consider $h$ and $\mathbb{Z}/p_{0}\mathbb{Z}$ as defined in the proof of the theorem.
Note that the multiplicative group $(\mathbb{Z}, \times )$ satisfies, with slight abuse of notation, that $(\mathbb{Z}/p_{0}\mathbb{Z})  \approx_{1} \mathbb{Z} $.  
However, the torsion subgroup of $\mathbb{Z}$ (resp.\ $(\mathbb{Z}/p_{0}\mathbb{Z}) $) is $\{0\}$ (resp.\ includes $1$), and we obtain $\Phi(\mathbb{Z}/p_{0}\mathbb{Z})  \not\approx_{1} \Phi(\mathbb{Z}) $, a contradiction.  
In light of the latter, we obtain \eqref{bef}, and \eqref{daric2} now follows by bringing standard extensionality in the former into the normal form and applying $\CI$ as in the proof of Theorem~\ref{sef}.
\end{proof}
Comparing \eqref{daric} and \eqref{daric2}, we note that the latter does not require effective information like the function $h$ in \eqref{forba}, while the former 
does not require an extensionality functional.  As the latter is merely an unbounded search, the version \eqref{daric2} seems preferable.  
However, the existence of an extensionality functional is not an entirely innocent assumption, as discussed in \cite{kooltje}*{Remark 3.6}.    

\smallskip

We now sketch a template for treating other algebraic theorems, like those mentioned at the beginning of this section, in the same way.  
Recall that `algebraic theorem' is meant to capture those (internal) theorems $T$ which are of the form $T\equiv (\forall X^{1})(\exists Y^{1})\varphi(X, Y)$, where $\varphi$ is arithmetical.  
\begin{tempie}[Algebraic theorems]\rm\label{effin}
If $\varphi$ involves existential quantifiers, we bring those outside (if possible) or introduce functionals to remove them.  For instance, if $\varphi(X, Y)\equiv (\forall n)(\exists m)\varphi_{0}(X,Y)$ with $\varphi_{0}$ quantifier-free, then we consider the formula
$(\forall X)(\exists Y, h)(\forall n)\varphi_{0}(X,Yn ,h(n))$.  If $\varphi(X, Y)\equiv (\forall k)\psi(X, Y,k)\di (\forall n)\varphi_{0}(X,Y)$ with $\varphi_{0}, \psi_{0}$ quantifier-free, then we consider
\[
(\forall X)(\exists Y, k)[\psi_{0}(X, Y, k)\di (\forall n)\varphi_{0}(X,Yn ,h(n))].  
\]
Hence, we obtain a formula of the form $(\forall X^{1})(\exists Y^{1})(\forall n)\varphi_{0}(X, Y,n)$ with $\varphi_{0}$ quantifier-free.  
Our nonstandard version ${T}_{\ns}$ is then defined as $(\forall^{\st} X^{1})(\exists^{\st} Y^{1})(\forall n)\varphi_{0}(X, Y,n)$.  If $T$ is equivalent to $\ACA_{0}$, or equivalent to $\WKL_{0}$ and has the same syntactical structure as the latter, 
the proof of $\paai\asa {T}_{\ns}$ gives rise to an explict equivalence involving $(\mu^{2})$ as in Theorem \ref{floggen346}.  Alternatively, prove 
\[
(\exists^{\st}\Phi)\big[(\forall X^{1})\varphi(X, \Phi(X))\wedge \Phi \textup{ is standard extensional}\big] \di \paai,
\]
and proceed as in the proof of Corollary \ref{floggen3467}.  
\end{tempie}
We finish this section with some remarks.
\begin{rem}[Other frameworks]\rm
As noted in the proof of Corollary \ref{consresultcor}, the approach from \cite{brie} is quite modular in that it does not depend on 
one particular formal system, but goes through in any system in which finite sequences can be coded easily, like $\textsf{EFA}$.  
Hence, it should be possible to formulate Theorem \ref{consresult} for a system based on Brattka's framework from \cite{bratger}.
Alternatively, one could `break' functional extensionality (absent in the Brattka's framework) by working with finite \emph{sets} rather 
than finite \emph{sequences}.  
\end{rem}
\begin{rem}[Measure theory]\rm
The usual definition of (Borel) measure is used in RM (see \cite{simpson2}*{X.1.2}).  However, the existence of a measure for all (RM codes of) open sets (as the usual supremum) is equivalent to $\ACA_{0}$.  
Of course, we can run the nonstandard version of this equivalence through $\CI$ with predictable results, but the essential role played by arithmetical comprehension does not (seem to) bode well for the development of measure theory in any `computable' fashion.  

\smallskip

Nonetheless, while $\lambda(A)$ may not be meaningful in $\RCA_{0}$ (where $\lambda$ is the Lebesgue integral), the \emph{inequality} $\lambda(A)>_{\R}0$ always makes sense in the latter system (see \cite{simpson2}*{X.1}) in a kind of `comparative' sense.  As it turns out, such inequalities suffice for 
many measure theoretic theorems, and there even exists a similar treatment of the \emph{Loeb measure} (\cite{pimpson}).  With this `relative' definition of measure in place, a number of theorems from measure theory have been classified as equivalent to $\textsf{WWKL}_{0}$, a weaker version of $\WKL_{0}$ (see \cite{simpson2}*{X.1}).  
It would be quite interesting to study both the Lebesgue and Loeb measure using $\CI$.  
\end{rem}

\section{The bigger picture of Nonstandard Analysis}\label{foef}
In this section, we discuss the constructive and non-constructive aspects of Nonstandard Analysis, especially in light of our results and some current developments.  

\smallskip

First of all, it is important to point out the existence of \emph{Constructive Nonstandard Analysis} (see \cite{palmlist} for an incomplete list).  
In particular, both Robinson's model-theoretic/semantic approach (\cites{robinson1, lux1}) and Nelson's syntactic approach (\cite{wownelly}) to Nonstandard Analysis 
have associated \emph{constructive} versions.  For instance, a nonstandard model is constructed \emph{inside Martin-L\"of's constructive type theory} in \cites{dijkpalm, opalm}, while a version of Nelson's internal set theory \emph{compatible with Bishop's Constructive Analysis} is studied in \cite{brie}*{\S5}.  It is also worth mentioning the nonstandard type theory introduced by Martin-L\"of himself in \cite{MLNSA}.  

\smallskip
  
Now, in the constructive approach to Nonstandard Analysis, the principles \emph{Transfer, Standard Part, Saturation} which connect the standard and nonstandard universe, are sufficiently weakened so as to be compatible with constructive mathematics.     
Hence, the \emph{presence} of nonstandard objects (in a model or system) is not necessarily non-constructive, but the \emph{principles connecting the standard and nonstandard universe} {can be}.  
For instance, a rather weak instance of the Transfer principle already implies Turing's \emph{Halting problem} by \cite{bennosam}*{Cor.\ 12}.     

\smallskip

Secondly, despite the observations from the previous paragraph, Bishop (see \cite{kluut}*{p.\ 513}, \cite{bishl}*{p.\ 1}, and \cite{kuddd}, which is the review of \cite{keisler3}) and Connes (see \cite{conman2}*{p.\ 6207} and \cite{conman}*{p.\ 26}) have made rather strong claims regarding the non-constructive nature of Nonstandard Analysis.  Their arguments have been investigated in remarkable detail and were mostly refuted (see e.g.\ \cites{gaanwekatten, keisler4, kano2,samsynt}).  
Following the results in this paper (and especially given our study of Herbrandisations), we can conclude that the praxis of Nonstandard Analysis \emph{is highly constructive in nature}, in direct opposition to the Bishop-Connes claims.  

\smallskip

Thirdly, the program \emph{Univalent foundations of mathematics} (see \cite{hottbook} for a comprehensive treatment) is usually described as:
\begin{quote}
Vladimir Voevodsky's new program for a comprehensive, computational foundation for mathematics based on the homotopical interpretation of type theory. (see \cite{hottweb})
\end{quote}
Our results suggest that Nonstandard Analysis \emph{already provides} a computational foundation for mathematics \emph{inside the usual foundational system $\ZFC$}.  
In particular, our approach does not require one to strictly adhere to intuitionistic logic.  

\smallskip

Fourth, Tao has on numerous occasions discussed the connection between so-called hard and soft analysis, and how Nonstandard Analysis connects the two (see \cite{taote}*{\S2.3 and \S2.5}).   
Intuitively speaking, soft (resp.\ hard) analysis deals with qualitative (resp.\ quantitative) information and continuous/infinite (resp.\ discrete/finite) objects.  
It seems the `hard versus soft' distinction was made by Hardy (\cite{hardynicht}*{p.\ 64}) in print for the first time.  
It goes without saying that the template $\CI$ from Section~\ref{detail} provides a `direct one way street' from soft analysis stemming from (pure) Nonstandard Analysis to hard analysis embodied by the effective theorems.
In turn, the Herbrandisation of a theorem allow us to jump to soft analysis from the hard version, as discussed at the end of Section~\ref{frakkk}.  

\smallskip

Fifth, it is fitting that the results in this paper combine two of Leibniz' well-known research interests, namely the infinitesimal calculus and his \emph{calculemus} views;  the latter in the guise of the template $\CI$.        
Indeed, following Turing's negative solution (\cite{tur37}) to Hilbert's \emph{Entscheidungsproblem}, it is impossible to `compute' the truth of mathematical theorems; moreover, the classification provided by the program \emph{Reverse Mathematics} (see Section~\ref{RM}) suggests that most mathematical objects and theorems are non-computable (in a specific technical sense).  
Despite this pervasive and seemingly ubiquitous non-computability of mathematics, the theorems proved in (pure) Nonstandard Analysis turn out to have lots of computational content, and to bring out the latter, one need only follow Leibniz' dictum \emph{Let us calculate!}, i.e.\ follow the template $\CI$.   

\smallskip

Finally, this paper was completed somewhere near the end of 2015, implying that (at least) five years have passed during which new results may have come to the fore.  
We are particularly pleased to mention Zoltan Kocsis' dissertation (\cite{zoltanphd}) in which $\CI$ is mentioned.  The actual content of the aforementioned PhD thesis is definitely worth the interested reader's attention.    We have mentioned the papers with Dag Normann (\cites{dagsam, dagsamII}) in which we investigate the class of special fan functionals vis \`a vis nonstandard compactness.   
The author and Benno van den Berg have extended $\P_{0}$ in \cite{bennosam} with \emph{Transfer} for parameter-free formulas.  
The resulting system does not have any interesting term extraction theorem, but gives rise to many nice equivalences, like e.g.\ $\paai\asa (\mu^{2})$ (see \cites{bennosam, samcie18}).

\begin{ack}\rm
This research was supported by the following funding bodies: FWO Flanders, the John Templeton Foundation, the Alexander von Humboldt Foundation, the University of Oslo, and the Japan Society for the Promotion of Science.  
The author expresses his gratitude towards these institutions. 
The author would like to thank Anil Nerode, Grigori Mints, Ulrich Kohlenbach, Horst Osswald, Helmut Schwichtenberg, Stephan Hartmann, Dag Normann, and Karel Hrbacek for their valuable advice.  
\end{ack}

\bye